\documentclass[12pt,twoside]{article}
\usepackage[hmargin=0.8in,vmargin=0.9in]{geometry}
\usepackage{fancyhdr}
\usepackage{graphicx}
\usepackage{subfigure}
\usepackage{amssymb}
\usepackage{amsmath}
\usepackage{amsfonts}
\usepackage{theorem}
\usepackage{mathrsfs}
\usepackage{mathtools}
\usepackage{bm}
\usepackage{color}
\usepackage{setspace}
\usepackage{epstopdf}
\usepackage{float}
\usepackage{cite}

\usepackage{booktabs} 
\usepackage{array} 
\usepackage{paralist} 
\usepackage{verbatim} 
\usepackage{subfigure} 

\usepackage{multirow}
\usepackage{hhline}

\newcommand\be{\begin{equation}}
\newcommand\ee{\end{equation}}
\newcommand\bes{\begin{eqnarray}}
\newcommand\ees{\end{eqnarray}}

\newcommand{\old}[1]{{}}

\newenvironment{proof}{{\flushleft \bf Proof:}}{}

{\theorembodyfont{\rmfamily}\newtheorem{algorithm}{Algorithm}[section]}
\newtheorem{theorem}{Theorem}[section]

{\theorembodyfont{\rmfamily}\newtheorem{remark}{Remark}[section]}
{\theorembodyfont{\rmfamily}\newtheorem{example}{Example}[section]}
\numberwithin{equation}{section}

\newcommand*\xbar[1]{%
  \hbox{%
    \vbox{%
      \hrule height 0.5pt 
      \kern0.5ex
      \hbox{%
        \kern-0.1em
        \ensuremath{#1}%
        \kern-0.0em
      }%
    }%
  }%
}

\numberwithin{equation}{section}
\numberwithin{figure}{section}
\numberwithin{table}{section}

\title{Well-Balanced Central Schemes on Overlapping Cells with Constant Subtraction Techniques for the Saint-Venant Shallow Water System}

\author{Suo Yang \footnotemark[2], Alexander Kurganov \footnotemark[3], Yingjie Liu \footnotemark[4]}
\date{}

\begin{document}

\maketitle

\renewcommand{\thefootnote}{\fnsymbol{footnote}}
\footnotetext[2]{School of Mathematics, Georgia Institute of Technology, Atlanta, GA 30332; {\tt syang305@gatech.edu}. Research supported in
part by NSF grant DMS-1115671.}

\footnotetext[3]{Mathematics Department, Tulane University, New Orleans, LA 70118; {\tt kurganov@math.tulane.edu}. Research supported in
part by NSF grant DMS-1216957 and ONR grant N00014-12-1-0833.}

\footnotetext[4]{School of Mathematics, Georgia Institute of Technology, Atlanta, GA 30332; {\tt yingjie@math.gatech.edu}. Research
supported in part by NSF grant DMS-1115671.}

\renewcommand{\thefootnote}{\arabic{footnote}}

\begin{abstract}
We develop well-balanced central schemes on overlapping cells for the Saint-Venant shallow water system and its variants. The main challenge
in deriving the schemes is related to the fact that the Saint-Venant system contains a geometric source term due to nonflat bottom
topography and therefore a delicate balance between the flux gradients and source terms has to be preserved. We propose a constant
subtraction technique, which helps one to ensure a well-balanced property of the schemes, while maintaining arbitrary high-order of
accuracy. Hierarchical reconstruction limiting procedure is applied to eliminate spurious oscillations without using characteristic
decomposition. Extensive one- and two-dimensional numerical simulations are conducted to verify the well-balanced property, high-order of
accuracy, and non-oscillatory high-resolution for both smooth and nonsmooth solutions.
\end{abstract}

\section{Introduction}
The Saint-Venant system \cite{SaintVenant1871} is one of the most commonly used models of shallow water flows in rivers or coastal areas.
Let $h$ represent the water depth, $u$ represent the velocity, $B$ represent the bottom elevation, and $g$ be the gravitational acceleration
constant. Then the one-dimensional (1-D) Saint-Venant shallow water system has the following form:
\begin{equation}
\left\{\begin{aligned}&h_t+(hu)_x=0,\\&(hu)_t+\Big(hu^2+\frac{1}{2}gh^2\Big)_x=-ghB_x.\end{aligned}\right.
\label{1DSV}
\end{equation}

Similarly, let $u$ and $v$ represent the $x$- and $y$-velocities. Then the two-dimensional (2-D) Saint-Venant shallow water system reads
\begin{equation}
\left\{\begin{aligned}&h_t+(hu)_x+(hv)_y=0,\\&(hu)_t+\Big(hu^2+\frac{1}{2}gh^2\Big)_x+(huv)_y=-ghB_x,\\
&(hv)_t+(huv)_x+\Big(hv^2+\frac{1}{2}gh^2\Big)_y=-ghB_y.\end{aligned}\right.
\label{eqn:2DSV}
\end{equation}

The viscous effects were neglected by asymptotic analysis in the derivation of the Saint-Venant system \cite{SaintVenant1871}, and
therefore, in the case of flat bottom topography ($B_x\equiv B_y\equiv0$), its behavior and features are very similar to the Euler equations
of isentropic gas dynamics. For situations like dam breaking, strong shocks can be formed just like in the isentropic gas dynamics. For
these reasons, high-resolution non-oscillatory shock-capturing schemes are required to solve the Saint-Venant system, which in the case of
nonflat bottom topography is a hyperbolic system of balance laws. In principle, all numerical schemes designed for hyperbolic conservation
laws can be extended to balance laws. This motivates us to study applicability of central schemes on overlapping cells (CSOC), originally
developed for hyperbolic systems of conservation laws in \cite{Liu05,LiuShu07,LiuShu07b}, to the shallow water equations.

In the past three decades, many numerical methods for the Saint-Venant system have been proposed. Just like many other systems of balance
laws, the Saint-Venant system admits steady-state solution, in which the flux gradient is exactly balanced by the source term. The simplest
steady-state solutions are ``lake at rest'' ones. In the 1-D case, they are
\begin{equation*}
w:=h+B\equiv{\rm Const},\qquad u\equiv0,
\end{equation*}
where $w$ represents the water surface. Similarly, the 2-D ``lake at rest'' satisfies
\begin{equation}
w\equiv{\rm Const},\qquad u\equiv v\equiv0.
\label{ss2D}
\end{equation}
The ``lake at rest'' solutions are physically significant since most of physically relevant water waves are in fact small perturbations of
``lake at rest'' steady states. We therefore say that a scheme is well-balanced if it is capable of exactly preserving ``lake at rest''
steady states. Unfortunately, a straightforward discretization of the geometric source term typically leads to a non-well-balanced scheme,
which may produce artificial waves that are an order of magnitude larger than the physical waves especially when a coarse grid is used
(which is always the case in practical applications in which using sufficiently fine grids is typically unaffordable).

A second-order accurate quasi-steady wave-propagation scheme was proposed in \cite{LeVeque98}. In this scheme, a new Riemann problem is
introduced at the centroid of each cell such that the flux difference can exactly cancel the contribution of the source term. A
Riemann-problem-solver-free central-upwind scheme was proposed in \cite{KurganovLevy02,KurganovPetrova07}. In this scheme, $w$ rather than
$h$ is used in the reconstruction step to keep equilibrium variables constant and the numerical flux term can be approximated with
high-order accuracy, while the source term is only second-order accurate to preserve the well-balanced property. A higher-order
discretization of the source term was proposed in \cite{NPPN,NXS}. Another approach was introduced in \cite{XingShu05,XingShu06,XingShu06b},
where high-order well-balanced finite-difference and finite-volume weighted essentially non-oscillatory (WENO) schemes as well as
discontinuous Galerkin methods were derived using a special local splitting of the source term for which all operators are linear and thus
the well-balanced property can be easily achieved. For several other well-balanced schemes for the Saint-Venant system we refer the reader
to \cite{ABBKP,GHS03,Jin,LMNK,Rus05}.

In this paper, we develop CSOC with the hierarchical reconstruction (HR) limiter \cite{Liu05,LiuShu07,LiuShu07b} for both the 1-D and 2-D
Saint-Venant systems. Just like the schemes in \cite{XingShu05,XingShu06,XingShu06b}, our scheme can also be formulated to achieve arbitrary
high-order while still preserving well-balanced property using the constant subtraction technique, which is substantially easier to
implement than the well-balancing techniques used in \cite{XingShu05,XingShu06,XingShu06b}. Another attractive feature of the proposed CSOC
is that no (approximate) Riemann problem solver needs to be implemented and all significant spurious oscillations can be removed by the HR
limiter implemented together with a new remainder correction technique without local characteristic decomposition.

This paper is organized as follows. In Section \ref{sec:coc}, we briefly review the CSOC for general hyperbolic systems of balance laws. In
Section \ref{sec:wb}, we propose the constant subtraction technique and prove that it leads to well-balanced schemes. In Section
\ref{sec:HR}, we review the HR limiters and develop the remainder correction technique. Extensive numerical simulations are conducted in
Sections \ref{sec7} and \ref{sec8} for the 1-D and 2-D Saint-Venant systems, respectively. Finally, conclusions and perspectives of the
future work are given in Section \ref{sec9}.

\section{Central Schemes on Overlapping Cells---A Brief Overview}\label{sec:coc}
A general 1-D hyperbolic system of balance laws has the following form:
\begin{equation}
\bm{u}_t+\bm{f}(\bm{u})_x=\bm{S}(\bm{u},x,t).
\label{eqn:1DBalanceLaw}
\end{equation}
Let $D_{i+\frac{1}{2}}:=[x_i,x_{i+1}]$ be a cell of uniform ($x_{i+1}-x_i\equiv\Delta x$) partition of the real line, and let
$\bm{V}_{i+\frac{1}{2}}^n$ be the corresponding computed cell averages of $\bm{u}$ at time $t^n$:
$$
\bm{V}_{i+\frac{1}{2}}^n:\approx\frac{1}{\Delta x}\int\limits_{x_i}^{x_{i+1}}\bm{u}(x,t^n)\,dx.
$$
Let $C_i:=[x_{i-\frac{1}{2}},x_{i+\frac{1}{2}}]$ be a dual cell of staggered uniform partition, and let $\bm{U}_i^n$ be the corresponding
computed cell averages of $\bm{u}$ at time $t^n$:
$$
\bm{U}_i^n:\approx\frac{1}{\Delta x}\int\limits_{x_{i-\frac{1}{2}}}^{x_{i+\frac{1}{2}}}\bm{u}(x,t^n)\,dx.
$$
We can now apply CSOC from \cite{Liu05,LiuShu07,LiuShu07b} to \eqref{eqn:1DBalanceLaw} to get the following fully discrete form (for
conciseness, we will only show the updating formula for $\{\bm{V}_{i+\frac{1}{2}}\}$, the formula for $\{\bm{U}_i\}$ is similar):
\begin{align}
\bm{V}_{i+\frac{1}{2}}^{n+1}&=\theta\frac{1}{\Delta x}\int\limits_{x_i}^{x_{i+1}}\widetilde{\bm{U}}^n(x)\,dx+
(1-\theta)\bm{V}_{i+\frac{1}{2}}^n-\frac{\Delta t}{\Delta x}\big[\bm{f}(\widetilde{\bm{U}}^n(x_{i+1}))-\bm{f}(\widetilde{\bm{U}}^n(x_i))
\big]\nonumber\\
&+\frac{\Delta t}{\Delta x}\int\limits_{x_i}^{x_{i+1}}\bm{S}(\widetilde{\bm{U}}^n(x),x,t^n)\,dx.
\label{1DFullDiscScheme}
\end{align}
Here, $\widetilde{\bm{U}}$ is the reconstructed piecewise polynomial approximation of $\bm{U}(x,t^n)$, and
$\theta=\Delta t/\Delta\tau$, where $\Delta\tau$ is an upper bound for the current time stepsize $\Delta t$. $\Delta\tau$ is restricted by
the CFL condition $\frac{a\Delta\tau}{\Delta x}\le\frac{1}{2}$, where $a$ is the supremum of the spectral radius of the Jacobian
$\frac{\partial\bm{f}}{\partial\bm{u}}$ over all of the relevant values of $\bm{u}$. Also notice that CSOC with $\theta=1$ is a first-order
in time version of the (staggered) Nessyahu-Tadmor scheme \cite{NessyahuTadmor90}. For pure hyperbolic systems, $\theta$ in principle should
be as large as possible to allow large $\Delta t$ and hence reduce the computational cost. When the source term is stiff, one can also take
a smaller value of $\theta$.

If we subtract $\bm{V}_{i+\frac{1}{2}}^n$ from both sides of \eqref{1DFullDiscScheme}, divide by $\Delta t$, and take the limit as
$\Delta t\to0$, we obtain the following semi-discrete form of the CSOC:
\begin{equation}
\frac{d}{dt}\bm{V}_{i+\frac{1}{2}}=\frac{1}{\Delta\tau}\bigg[\frac{1}{\Delta x}\int\limits_{x_i}^{x_{i+1}}\widetilde{\bm{U}}(x)\,dx-
\bm{V}_{i+\frac{1}{2}}\bigg]-\frac{1}{\Delta x}\Big[\bm{f}(\widetilde{\bm{U}}(x_{i+1}))-\bm{f}(\widetilde{\bm{}U}(x_i))\Big]
+\frac{1}{\Delta x}\int\limits_{x_i}^{x_{i+1}}\bm{S}(\widetilde{\bm{U}}(x),x,t)dx.
\label{1DSemiDiscScheme}
\end{equation}
One should use a stable, sufficiently accurate ODE solver to evolve the solution in time.
\begin{remark}
In our numerical experiments, we have used the third-order strong-stability preserving Runge-Kutta (SSP-RK3) method \cite{GKS,GST,ShuOsh88}.
\end{remark}
\begin{remark}
Multidimensional CSOC can be derived similarly, see \cite{Liu05}.
\end{remark}

\section{Constant Subtraction Technique}\label{sec:wb}
Our goal is to design well-balanced CSOC. We first consider the 1-D case and denote the equilibrium variables by $\bm{a}:=(w,hu)^T$, which
remains constant at ``lake at rest'' steady states. Next, we rewrite the geometric source term using the equilibrium variable $w$ as
follows:
\begin{equation}
-ghB_x=-g(w-B)B_x=\Big(\frac{1}{2}gB^2\Big)_x-gwB_x,
\label{3.1}
\end{equation}
where the term $(\frac{1}{2}gB^2)_x$ is in conservative form.
\begin{remark}
Notice that the same source term decomposition was used in \cite{XingShu05} to maintain a well-balanced property of arbitrary high-order
finite-difference schemes.
\end{remark}

Since a direct application of the CSOC to the Saint-Venant system \eqref{1DSV} does not guarantee the resulting method to be
well-balanced, we modify the system and obtain the well-balanced CSOC using the following algorithm.

\begin{algorithm}[Constant Subtraction Technique]\label{alg31}
\item[]
\begin{enumerate}[\underline{Step \arabic{enumi}}.]
\item Let $\Omega$ be a computational domain of size $|\Omega|$. Denote the global spatial average of $w(x,t)$ by
$$
\xbar w(t):=\frac{1}{|\Omega|}\int\limits_\Omega w(x,t)\,dx
$$
and decompose the nonconservative term on the right-hand side (RHS) of \eqref{3.1} into the sum of a conservative term and a {\em constant
subtraction term} as follows:
\begin{equation}
-gwB_x=(-g\bar{w}B)_x+g(\bar{w}-w)B_x.
\label{3.2}
\end{equation}
\item Use \eqref{3.1} and \eqref{3.2} to rewrite the Saint-Venant system \eqref{1DSV} in terms of the equilibrium variables $\bm{a}$:
\begin{equation}
\left\{\begin{aligned}&w_t+(hu)_x=0,\\&(hu)_t+\bigg(\frac{(hu)^2}{w-B}+g[\xbar w(t)-w]B+\frac{g}{2}w^2\bigg)_x=
g[\xbar w(t)-w]B_x.\end{aligned}\right.
\label{WBsys}
\end{equation}
\item Apply the CSOC described in Section \ref{sec:coc} to the system \eqref{WBsys} in a straightforward manner.
\end{enumerate}
\end{algorithm}
\begin{remark}
The systems \eqref{WBsys} and \eqref{1DSV} are equivalent for both smooth and nonsmooth solutions.
\end{remark}
\begin{remark}
The term $g[\xbar w(t)-w]B_x$ will vanish at ``lake at rest'' steady states.
\end{remark}

\begin{theorem}\label{thm31}
The CSOC scheme with the forward Euler time discretization \eqref{1DFullDiscScheme} for the system \eqref{WBsys} is well-balanced.
\end{theorem}
\begin{proof}
Note that at ``lake at rest'' steady states $\xbar w(t)$ is independent of time and assume that at time $t=t^n$ the cell averages of
$\bm{a}$ over both $C_i$ and $D_{i+1/2}$ cells are equal to $(\xbar w,0)^T$. After performing a (high-order) piecewise polynomial
reconstruction for the equilibrium variables $\bm{a}$, we obtain that the polynomial pieces over both $C_i$ and $D_{i+1/2}$ still satisfy
$\widetilde{hu}\equiv0$ and $\widetilde w\equiv\xbar w$. Therefore, both the flux difference and source term in the CSOC
\eqref{1DFullDiscScheme} vanish and $\bm{V}_{i+\frac{1}{2}}^{n+1}$ becomes a convex combination of
$\frac{1}{\Delta x}\int_{x_i}^{x_{i+1}}\widetilde{\bm{U}}^n(x)\,dx$ and $\bm{V}_{i+\frac{1}{2}}^n$, both of which are at ``lake at rest''
steady state. Therefore, the cell averages at $t=t^{n+1}$ also satisfy $hu=0$ and $w=\xbar w$ and the proof of the theorem is complete.
$\hfill\blacksquare$
\end{proof}
\begin{remark}
We would like to stress that in order to guarantee the well-balanced property, it is important to reconstruct the equilibrium variables
$\bm{a}$ rather than the original ones, $(h,hu)^T$.
\end{remark}
\begin{remark}
Since SSP ODE solvers \cite{GKS,GST,ShuOsh88} are based on a convex combination of several forward Euler steps, Theorem \ref{thm31} is valid
for the semi-discrete CSOC \eqref{1DSemiDiscScheme} combined with a higher-order SSP solver.
\end{remark}
\begin{remark}
All of the results from Section \ref{sec:wb} can be directly extended to the 2-D case.
\end{remark}

\section{Non-Oscillatory Hierarchical Reconstruction (HR)}\label{sec:HR}
For discontinuous solutions, a nonlinear limiting procedure is typically required to eliminate spurious oscillations in the vicinities of
discontinuities. In the past few decades, a wide variety of nonlinear limiting techniques including the MUSCL \cite{Lee74,Lee77,Lee79}, ENO
\cite{OshChaHar87,Shu90,ShuOsh88,SO2} and WENO \cite{shu97a,Shu09} reconstructions and many others have been developed for solving this
problem. In this paper, we use the HR limiting technique originally designed in \cite{LiuShu07,LiuShu07b} for overlapping grid methods.

\subsection{HR Process---A Brief Overview}\label{subsec:HR}
Let us assume that we are given a set of cell averages, $\varphi_i$ and $\varphi_{i+\frac{1}{2}}$, of a certain computed quantity on
overlapping cells. Then, using a standard conservation technique one can build a central piecewise polynomial reconstruction of degree $d$
on each cell (see, e.g., \cite{LiuShu07b}). Unfortunately, a piecewise polynomial approximant reconstructed in such a linear, nonlimited way
may have spurious oscillations in nonsmooth regions and thus it must be corrected using a nonlinear limiter.

Suppose we have reconstructed polynomial pieces over the overlapping cells $C_i$ and $D_{i+\frac{1}{2}}$,
$\varphi_i(x)=\sum\limits_{m=0}^d\frac{\varphi_i^{(m)}(x_i)}{m!}(x-x_i)^m$ and
$\varphi_{i+\frac{1}{2}}(x)=\sum\limits_{m=0}^d\frac{\varphi_{i+\frac{1}{2}}^{(m)}(x_{i+\frac{1}{2}})}{m!}(x-x_{i+\frac{1}{2}})^m$,
expressed in terms of Taylor polynomials centered at $x_i$ and $x_{i+\frac{1}{2}}$, respectively. We now describe the HR process applied to
$\varphi_i(x)$ (an application of the HR to $\varphi_{i+\frac{1}{2}}(x)$ is similar). Using the HR to limit the polynomial $\varphi_i(x)$ is
to modify its coefficients $\varphi_i^{(m)}(x_i)$ to obtain their new values $\widetilde\varphi_i^{(m)}(x_i)$, thus generating a
non-oscillatory polynomial $\widetilde\varphi_i(x)$ with the same order of accuracy. In the following, we use a pointwise HR proposed in
\cite{XuLiuShu11} to explain the 1-D HR algorithm.

\begin{algorithm}[Pointwise HR]\label{alg41}
\item[]
\begin{enumerate}[\underline{Step \arabic{enumi}}.]
\item Suppose $d\ge2$. Then, for $m=d,d-1,\cdots,1$ do the following:
\begin{enumerate}
\item Take the $(m-1)$th derivatives of $\varphi_i$ and $\varphi_{i\pm\frac{1}{2}}$ and write $\varphi_i^{(m-1)}(x)=L_{m,i}(x)+R_{m,i}(x)$,
where $L_{m,i}(x)$ is the linear part and $R_{m,i}(x)$ is the remainder.
\item Compute the cell average of $\varphi_i^{(m-1)}$ over $C_i$ to obtain the cell average $\overline{\varphi_i^{(m-1)}}$. Also compute the
point values $\varphi_{i+\frac{1}{2}}^{(m-1)}(x_{i+\frac{1}{2}})$ and $\varphi_{i-\frac{1}{2}}^{(m-1)}(x_{i-\frac{1}{2}})$.
\item Let $\widetilde R_{m,i}(x)$ be $R_{m,i}(x)$ with its coefficients replaced by the corresponding modified values. Compute the cell
average of $\widetilde R_{m,i}$ over $C_i$ to obtain the cell average $\overline{\widetilde R_{m,i}}$. Also compute the point values
$\widetilde R_{m,i}(x_{i+\frac{1}{2}})$ and $\widetilde R_{m,i}(x_{i-\frac{1}{2}})$.
\item Let $\,\xbar L_{m,i}:=\overline{\varphi_i^{(m-1)}}-\overline{\widetilde R_{m,i}}\,$ and
$\,\xbar L_{m,i\pm\frac{1}{2}}:=\varphi_{i\pm\frac{1}{2}}^{(m-1)}(x_{i\pm\frac{1}{2}})-\widetilde R_{m,i}(x_{i\pm\frac{1}{2}})$.
\item Reconstruct a non-oscillatory linear function $L(x)$ on $C_i$ using $\xbar L_{m,i}$, $\xbar L_{m,i+\frac{1}{2}}$ and
$\xbar L_{m,i-\frac{1}{2}}$, and define the modified coefficient $\widetilde\varphi_i^{(m)}(x_i):=L'(x)$.
\end{enumerate}
\item The modified $0$th degree coefficient $\widetilde\varphi_i(x_i)$ is chosen such that the cell average of $\widetilde\varphi_i(x)$ over
$C_i$ is equal to that of $\varphi_i$.
\end{enumerate}
\end{algorithm}

After the set of modified coefficients $\widetilde\varphi_i^{(m)}(x_i)$ has been computed, we obtain a non-oscillatory polynomial piece
$\widetilde\varphi_i(x)$ on $C_i$.

\begin{remark}\label{rmk:SmoothDetect}
The HR is quite computationally expensive. To substantially reduce the overall computational cost, one can utilize a smoothness detector to
turn off the HR in smooth regions. In all of the numerical simulations reported below, we have used the same low cost smoothness detector as
in \cite{XinHu05} and \cite{LiuShu07b}.
\end{remark}

\subsection{Remainder Correction Technique}\label{sec42}
As any of the existing high-order limiting techniques, the HR is capable of limiting the spurious oscillations, which unfortunately cannot
be completely eliminated, especially in the most demanding shallow water models containing nonconservative source terms appearing on the RHS
of \eqref{WBsys} in the case of discontinuous bottom topography function $B$. Here, we introduce a technique to further regulate the
remainder term $\widetilde R_{m,i}(x)$ in Step 1(c) of Algorithm \ref{alg41}. This technique does not affect its approximation order of
accuracy and further reduces possible overshoots/undershoots near discontinuities.

In this paper, we will only consider the third-order HR. Let $\widetilde R_{m,i}(x)=\alpha_{m,i}(x-x_i)^2$ (with $m=1,d=2$). Obviously,
$\widetilde R_{m,i}(x)={\cal O}((\Delta x)^2)$ in $C_i$, where $\Delta x$ is the spatial grid size. Based on $\widetilde R_{m,i}$, we want
to construct a corrected remainder $\widetilde R_{m,i}^{\rm corr}$ satisfying the following two conditions:
\begin{equation}
\left\{\begin{aligned}&\widetilde R_{m,i}^{\rm corr}(x)=\widetilde R_{m,i}(x)+{\cal O}((\Delta x)^3),\quad\forall x\in C_i,\\
&|\widetilde R_{m,i}^{\rm corr}(x)|<M,\quad\forall x\in\mathbb{R},~\mbox{for some constant}~M.\end{aligned}\right.
\label{4.1}
\end{equation}
The first requirement in \eqref{4.1} is needed to avoid any loss of accuracy. The second condition in \eqref{4.1} is introduced to control
the spurious oscillations, because $\widetilde R_{m,i}$ grows quite fast away from $x_i$ and the values
$\widetilde R_{m,i}(x_{i\pm\frac{1}{2}})$ used in Step 1(d) of Algorithm \ref{alg41} may lead to oscillations. There are many different ways
to ensure \eqref{4.1}. In this paper, we take
\begin{equation}
\widetilde R_{m,i}^{\rm corr}(x)=\frac{\widetilde R_{m,i}(x)}{1+\sqrt{|\alpha_{m,i}|}\,|x-x_i|+|\alpha_{m,i}|(x-x_i)^2}.
\label{4.2}
\end{equation}

\begin{theorem}\label{thm41}
The corrected remainder $\widetilde R_{m,i}^{\rm corr}$ given by \eqref{4.2} satisfies the two conditions in \eqref{4.1}.
\end{theorem}
\begin{proof} The definition of $\widetilde R_{m,i}^{\rm corr}$, \eqref{4.2}, and the fact that
$\widetilde R_{m,i}(x)={\cal O}((\Delta x)^2)$ in $C_i$ imply that the first condition in \eqref{4.1} holds, namely:
\begin{equation*}
\begin{aligned}
\widetilde R_{m,i}^{\rm corr}(x)&=
\widetilde R_{m,i}(x)\left[1+{\cal O}\Big(\sqrt{|\alpha_{m,i}|}\,|x-x_i|+|\alpha_{m,i}|(x-x_i)^2\Big)\right]
=\widetilde R_{m,i}(x)\left[1+{\cal O}\left(\Delta x+(\Delta x)^2\right)\right]\\
&=\widetilde R_{m,i}(x)+{\cal O}((\Delta x)^3), \quad\forall x\in C_i.
\end{aligned}
\end{equation*}
The second condition in \eqref{4.1} holds because $\widetilde R_{m,i}^{\rm corr}$ is continuous and
\begin{equation*}
\lim_{|x|\to\infty}\widetilde R_{m,i}^{\rm corr}(x)=
\lim_{|x|\to\infty}\frac{\alpha_{m,i}(x-x_i)^2}{1+\sqrt{|\alpha_{m,i}|}\,|x-x_i|+|\alpha_{m,i}|(x-x_i)^2}={\rm sgn}(\alpha_{m,i})=\pm1.
\end{equation*}
$\hfill\blacksquare$
\end{proof}
\begin{remark}
In the HR process presented in Algorithm \ref{alg41} one has to compute cell averages. This can be done analytically when the averaged
quantities are polynomials. However, for other functions, for example, for the corrected remainder \eqref{4.2}, it may be not easy or even
impossible to evaluate the integral exactly. In such case, we compute the required cell averages numerically using a quadrature of an
appropriate order that will not reduce the accuracy of the HR process (see \cite{XuLiuShu11} for details).
\end{remark}
\begin{remark}
The remainder correction technique presented in this section can be extended to higher-order HR by increasing the degree of the polynomial
in the denominator on the RHS of \eqref{4.2}.
\end{remark}

\section{One-Dimensional Numerical Examples}\label{sec7}
In this section, we demonstrate performance of the well-balanced CSOC with the HR limiter. We use the third-order schemes though
higher-order well-balanced CSOC can also be constructed. In all of the examples, we take the CFL number 0.4 and the gravitational
acceleration constant $g=9.812$.

\begin{example}[Verification of the Well-Balanced Property]
This test problem is taken from \cite{XingShu05}. The computational domain is $0\le x\le10$, and the initial condition is the ``lake at
rest'' state with $w(x,0)\equiv10,\,(hu)(x,0)\equiv0$, which should be exactly preserved. We use absorbing boundary conditions and test two
different bottom topography functions. The first one is smooth:
\begin{equation*}
B(x)=5e^{-\frac{2}{5}(x-5)^2},
\end{equation*}
while the second one is nonsmooth:
\begin{equation*}
B(x)=\left\{\begin{aligned}&4,&&\mbox{if}~4\le x\le8,\\&0,&&\mbox{otherwise}.\end{aligned}\right.
\end{equation*}
We use $N=200$ uniform cells, and obtain that even at large final times both the $L^1$- and $L^\infty$-errors are machine zeros for both
smooth and nonsmooth bottom topographies.
\end{example}

\begin{example}[Accuracy Test]\label{ex72}
The goal of this numerical example, taken \cite{XingShu05}, is to experimentally verify the order of accuracy of the (formally) third-order
CSOC. The computational domain is $0\le x\le1$ and the boundary conditions are periodic. The initial data and bottom topography are
\begin{equation*}
\begin{split}
&w(x,0)=5.5-0.5\cos(2\pi x)+e^{\cos(2\pi x)},\quad(hu)(x,0)=\sin(\cos(2\pi x)),\quad B(x)=\sin^2(\pi x).
\end{split}
\end{equation*}
We compute the solution of this initial-boundary value problem up to time $t=0.1$ when the solution is still smooth (shocks will be
developed at a later time). Since we use the smoothness detector mentioned in Remark \ref{rmk:SmoothDetect}, the HR limiter is essentially
turned off for this smooth solution.

Since the exact solution is not available, we use Aitken's formula \cite{NumAnaly} to estimate the experimental order of accuracy $r$:
\begin{equation*}
r=\log_2\left(\frac{\|u_{\frac{\Delta x}{2}}-u_{\Delta x}\|}{\|u_{\frac{\Delta x}{4}}-u_{\frac{\Delta x}{2}}\|}\right),
\end{equation*}
where $u_{\Delta x}$ denotes the numerical solution computed using the uniform grid of size $\Delta x$. In Table \ref{tab71}, we show the
experimental orders of accuracy measured in the $L^1$- and $L^\infty$-norms. As one can clearly see, the expected third order of accuracy is
reached for both $w$ and $hu$.
\begin{table}[ht!]
\centering
\begin{tabular}{|c||c|c|c|c|}
\hline
\multirow{2}{*}{$\Delta x$} & \multicolumn{2}{|c|}{$w$} & \multicolumn{2}{|c|}{$hu$}\\ \hhline{~----}
& $L^1$-order & $L^\infty$-order & $L^1$-order & $L^\infty$-order\\\hline
$1/50$   & $2.3671$ & $1.5893$ & $1.7588$ & $1.1060$\\
$1/100$  & $2.4040$ & $1.8154$ & $2.5282$ & $1.8850$\\
$1/200$  & $2.8303$ & $2.3321$ & $2.8365$ & $2.3537$\\
$1/400$  & $2.9355$ & $2.7597$ & $2.9361$ & $2.7652$\\
$1/800$  & $2.9880$ & $2.9796$ & $2.9885$ & $2.9729$\\
$1/1600$ & $2.9982$ & $2.9942$ & $2.9982$ & $2.9943$\\
\hline
\end{tabular}
\caption{\sf Example \ref{ex72}: Experimental orders of accuracy.\label{tab71}}
\end{table}
\end{example}

\begin{example}[Tidal Wave Flow]\label{ex76}
This example is taken from \cite{Bermudez94} and \cite{XingShu05}. The computational domain is $0\le x\le L$ with $L=14000$, the initial
data are
\begin{equation*}
w(x,0)\equiv60.5,\quad(hu)(x,0)\equiv0,
\end{equation*}
the bottom topography is given by
\begin{equation*}
B(x)=10+\frac{40x}{L}+10\,\sin\!\Big(\frac{4\pi x}{L}-\frac{\pi}{2}\Big),
\end{equation*}
and the boundary conditions are
\begin{equation*}
w(0,t)=64.5-4\,\sin\!\Big(\frac{4\pi t}{86400}+\frac{\pi}{2}\Big),\quad(hu)(L,t)=0.
\end{equation*}

This is a good test problem since a very accurate asymptotic approximation of the exact solution was obtained in \cite{Bermudez94}:
\begin{equation}
w(x,t)=64.5-4\,\sin\!\Big(\frac{4\pi t}{86400}+\frac{\pi}{2}\Big),\quad(hu)(x,t)=
\frac{\pi(x-L)}{5400}\,\cos\Big(\frac{4\pi t}{86400}+\frac{\pi}{2}\Big).
\label{7.1}
\end{equation}
We compute the numerical solution at time $t=7552.13$ using 200 uniform cells and compare the obtained results with \eqref{7.1}. As one can
see in Figure \ref{fig79}, the numerical and analytic approximate solutions are in a very good agreement.
\begin{figure}[ht!]
\centerline{\includegraphics[width=2.65in]{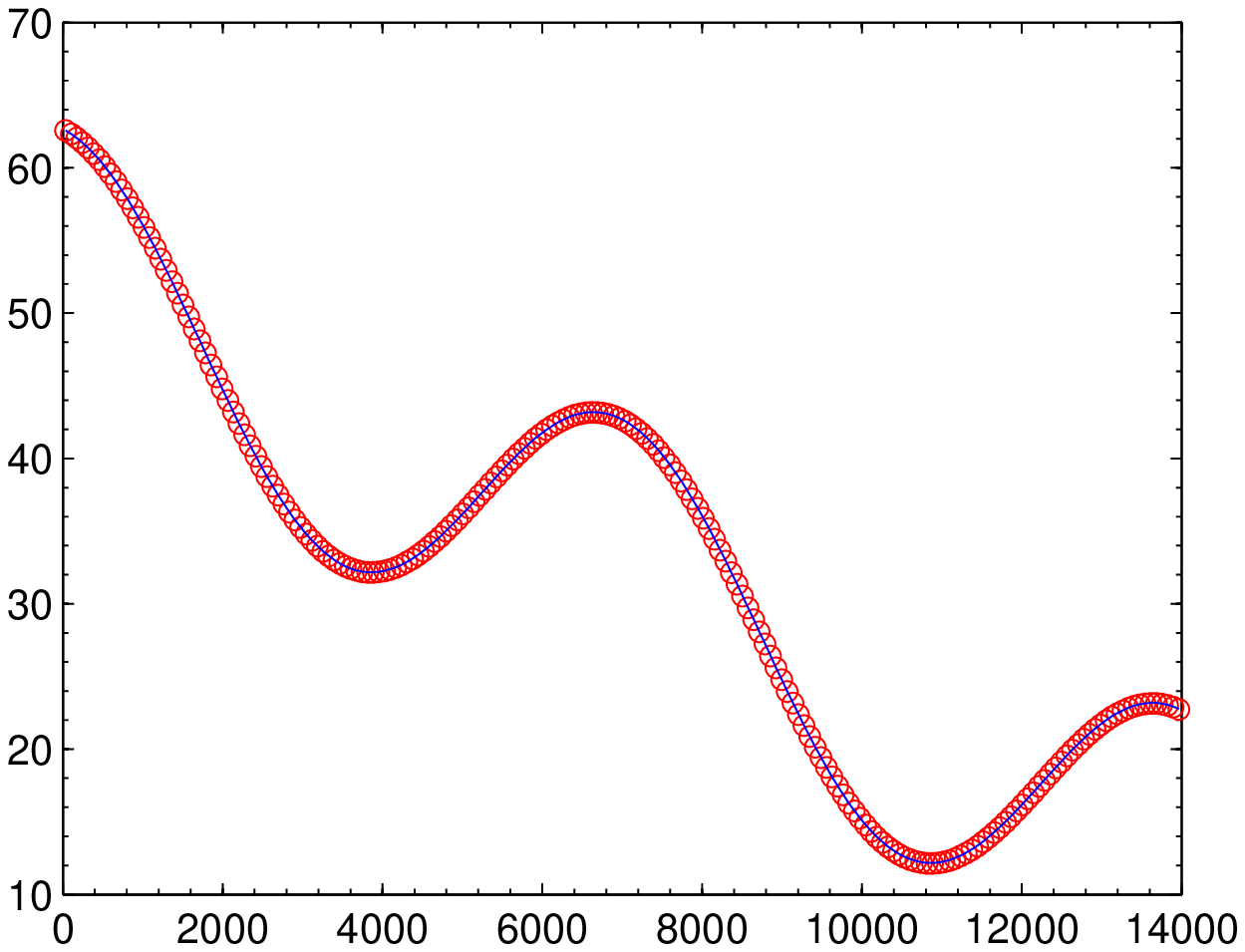}\includegraphics[width=2.65in]{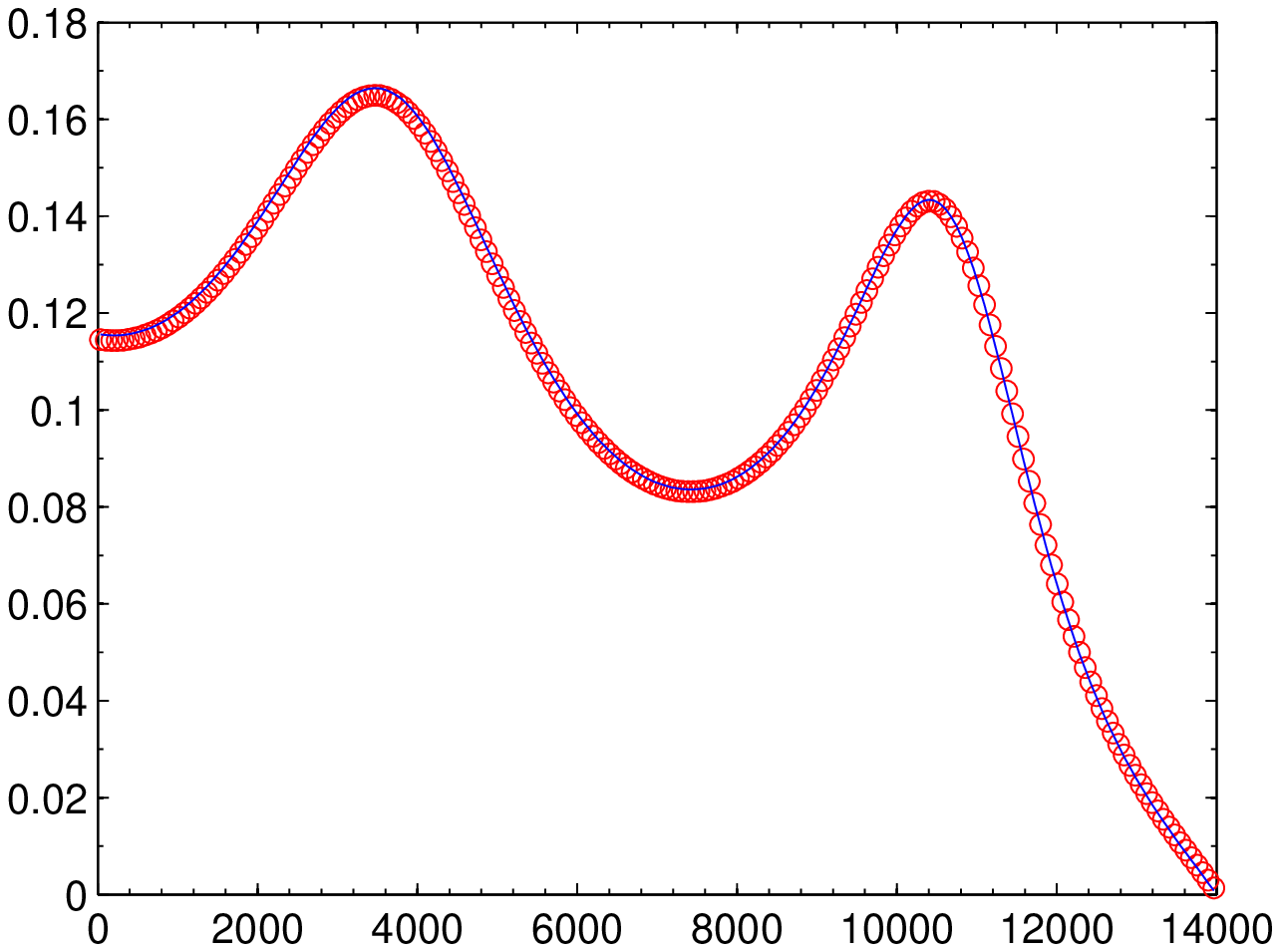}}
\caption{\sf Example \ref{ex76}: Numerical (circles) and analytic (solid line) approximate solutions ($h$ is on the left, $u$ is on the
right).\label{fig79}}
\end{figure}
\end{example}

\begin{example}[Perturbations of the ``Lake at Rest'' Steady State]\label{ex73}
This example is a slightly modified test problem, proposed in \cite{LeVeque98}, which was designed to verify the ability of tested scheme to
accurately capture quasi steady-state solutions. The computational domain is $0\le x\le2$, and the initial data are
\begin{equation*}
(hu)(x,0)\equiv0,\quad w(x,0)=\left\{\begin{aligned}&1+\varepsilon,&&\mbox{if}~1.1\le x\le1.2,\\&1,&&\mbox{otherwise},\end{aligned}\right.
\end{equation*}
where $\varepsilon$ is a small perturbation constant. We use absorbing boundary conditions and consider both large ($\varepsilon=0.2$) and
small ($\epsilon=0.001$) perturbations. The bottom topography contains a hump and is given by
\begin{equation*}
B(x)=\left\{\begin{aligned}&0.25\left[{\rm cos}(10\pi(x-1.5))+1\right],&&\mbox{if}~1.4\le x\le1.6,\\&0,&&\mbox{otherwise}.\end{aligned}
\right.
\end{equation*}

In this setting, the small perturbation of size $\varepsilon$ will split into two waves, one of which will propagate to the left, while the
other one will move to the right. The final time is set to be $t=0.2$, by which the right-going wave has already passed the bottom hump. It
is well-known (see, e.g., \cite{KurganovLevy02,LeVeque98,XingShu05}) that when $\varepsilon$ is small, non-well-balanced schemes cannot
capture the right-going wave without producing large magnitude artificial (nonphysical) waves unless an extremely fine mesh is used.

We compute the numerical solution by both the well-balanced and non-well-balanced CSOC on a 200 uniform grid and compare the obtained
results with the reference numerical solution computed using 3000 uniform cells. The results for $\varepsilon=0.2$ and $\varepsilon=0.001$
are shown in Figures \ref{fig71} and Figure \ref{fig72}, respectively. As one can see there, when $\varepsilon=0.2$ (relatively large
perturbation), there is no significant difference between the solutions computed by the well-balanced and non-well-balanced CSOC. On the
contrary, when $\varepsilon=0.001$ (much smaller perturbation), the non-well-balanced CSOC generates significant artificial waves, while the
well-balanced CSOC leads to a quite accurate non-oscillatory solution.
\begin{figure}[ht!]
\centerline{\includegraphics[width=2.65in]{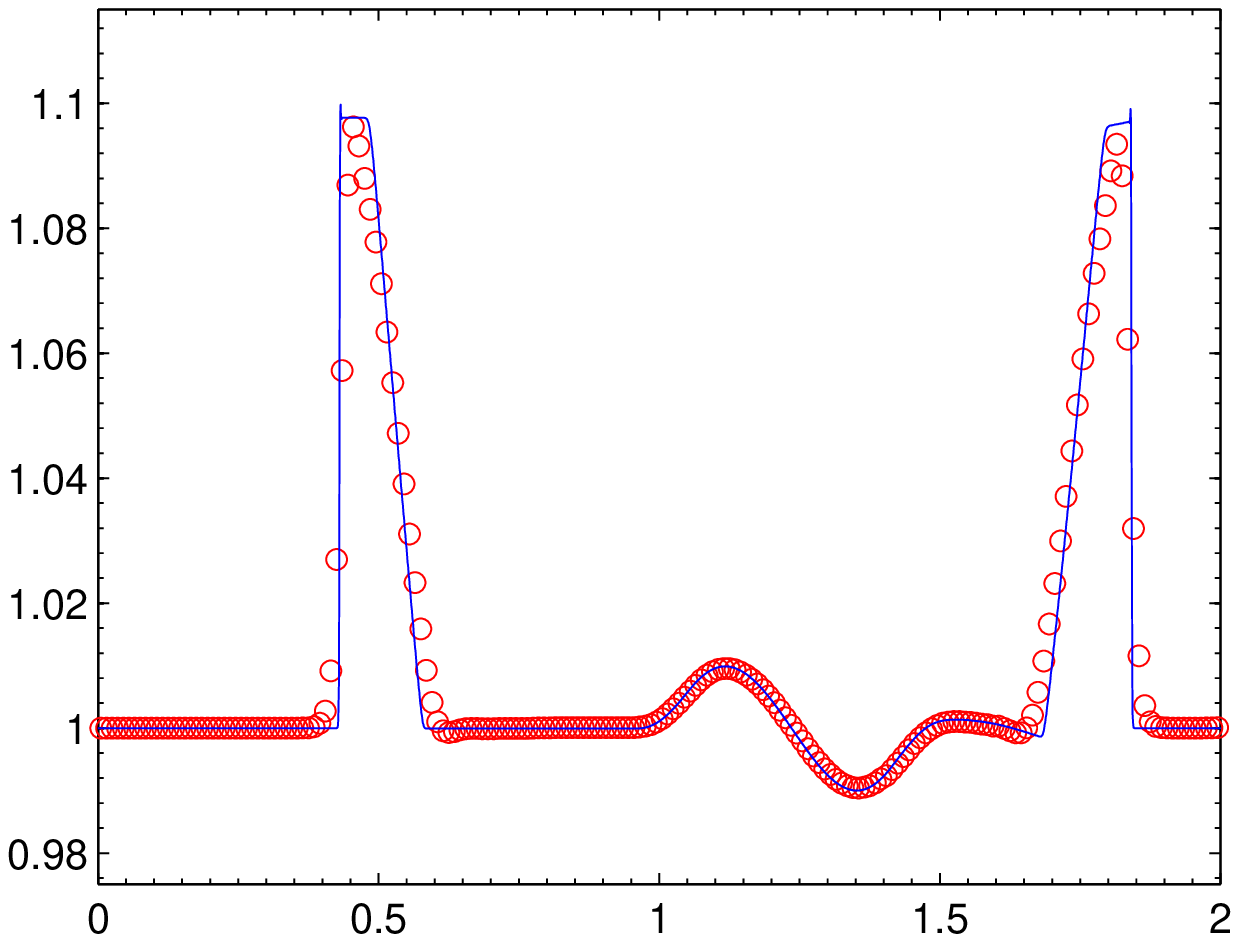}\includegraphics[width=2.65in]{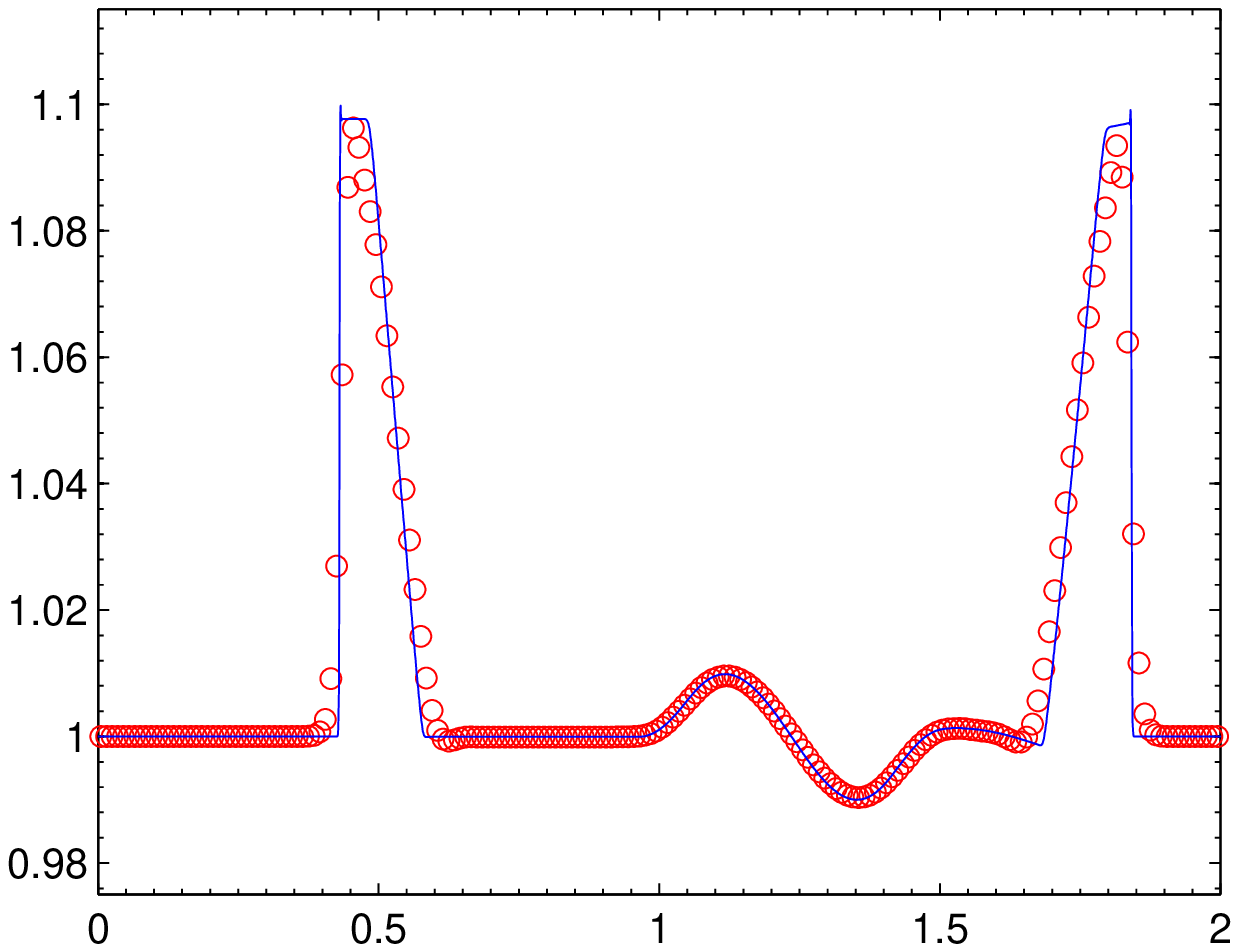}}
\centerline{\includegraphics[width=2.65in]{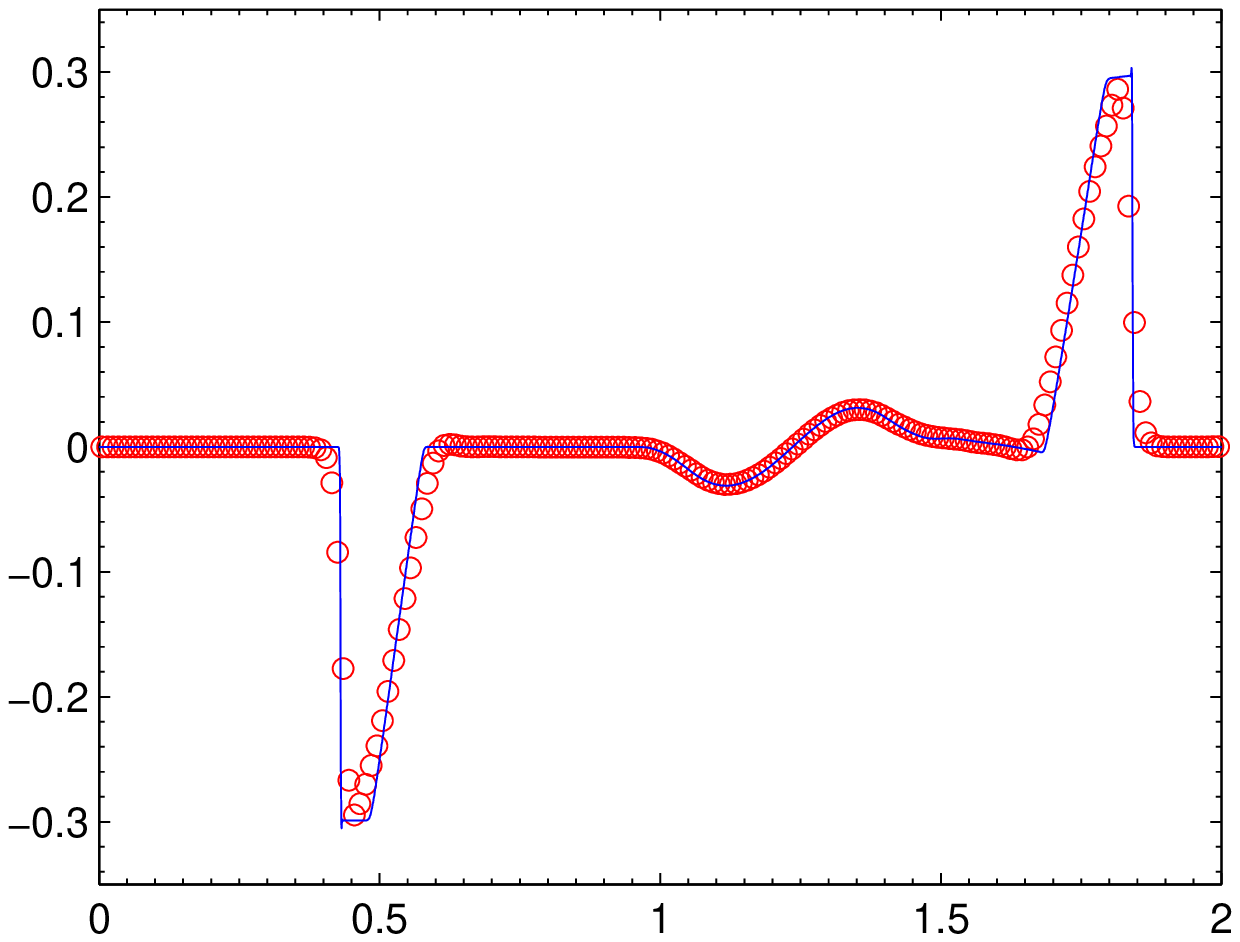}\includegraphics[width=2.65in]{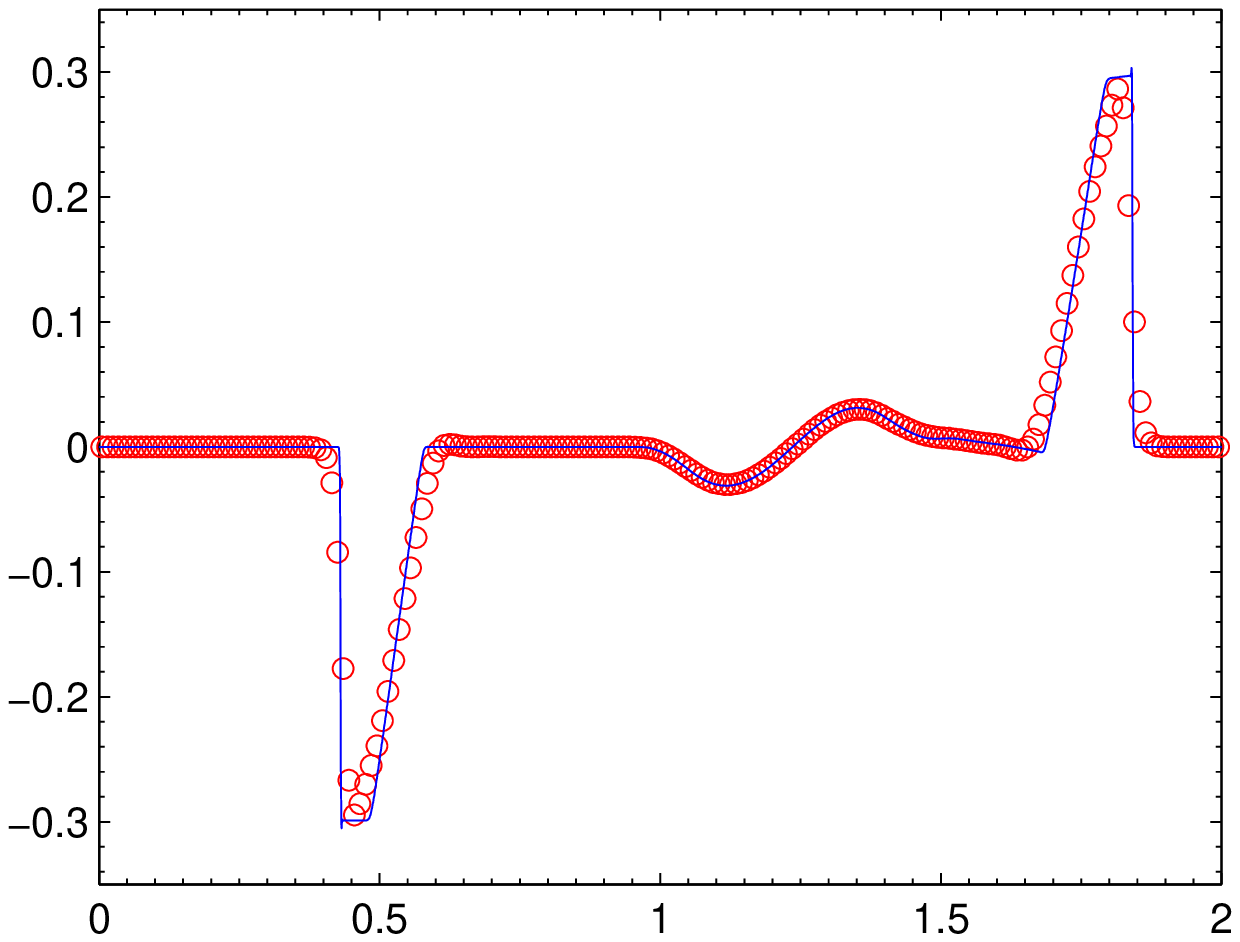}}
\caption{\sf Example \ref{ex73}: $\varepsilon=0.2$ (relatively large perturbation). Solutions ($w$ in the top row, $hu$ in the bottom row)
computed by the non-well-balanced (left column) and well-balance (right column) CSOC using uniform grids with 200 (circles) and 3000 (solid
line, reference solution) cells.\label{fig71}}
\end{figure}
\begin{figure}[ht!]
\centerline{\includegraphics[width=2.65in]{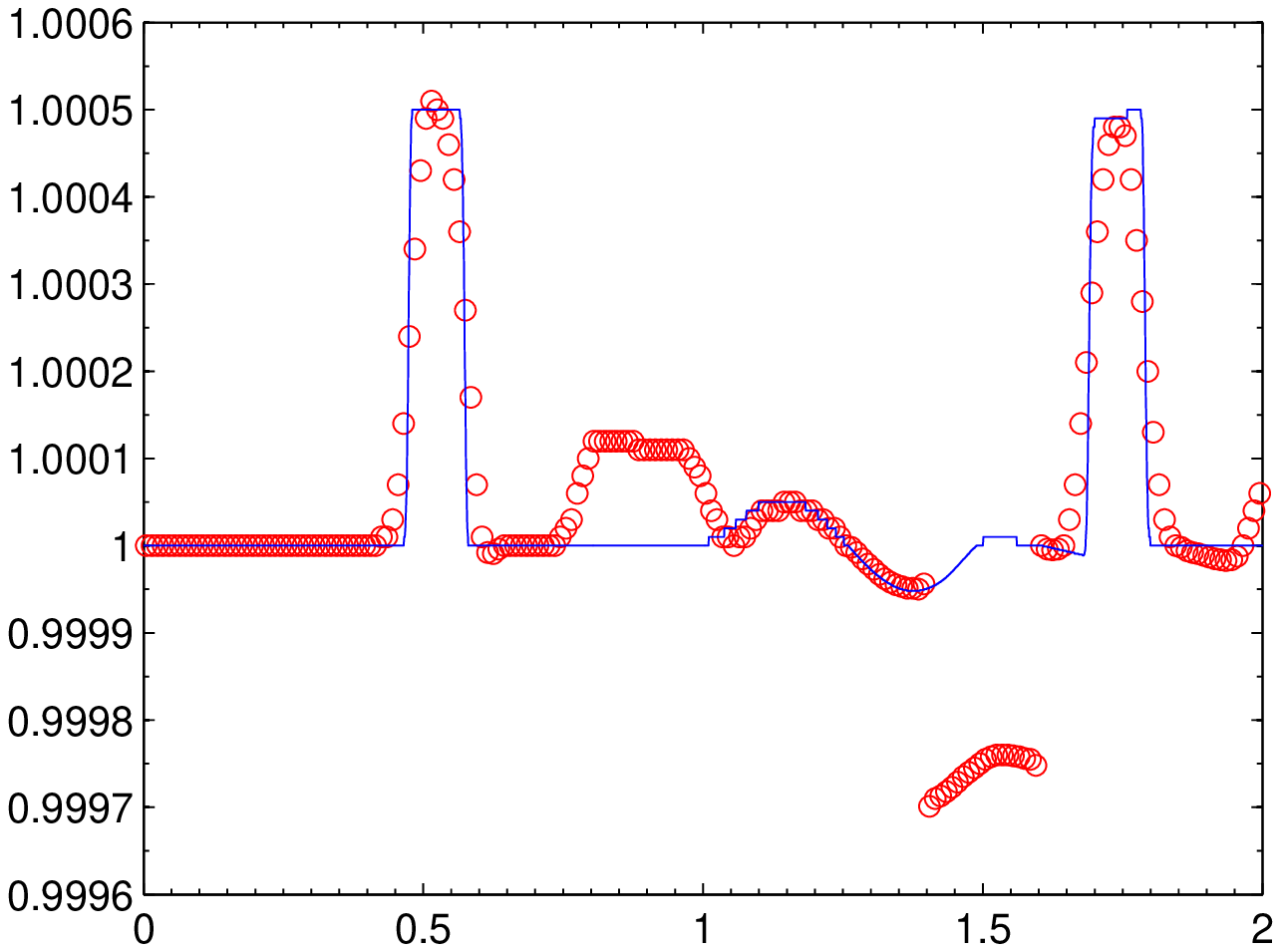}\includegraphics[width=2.65in]{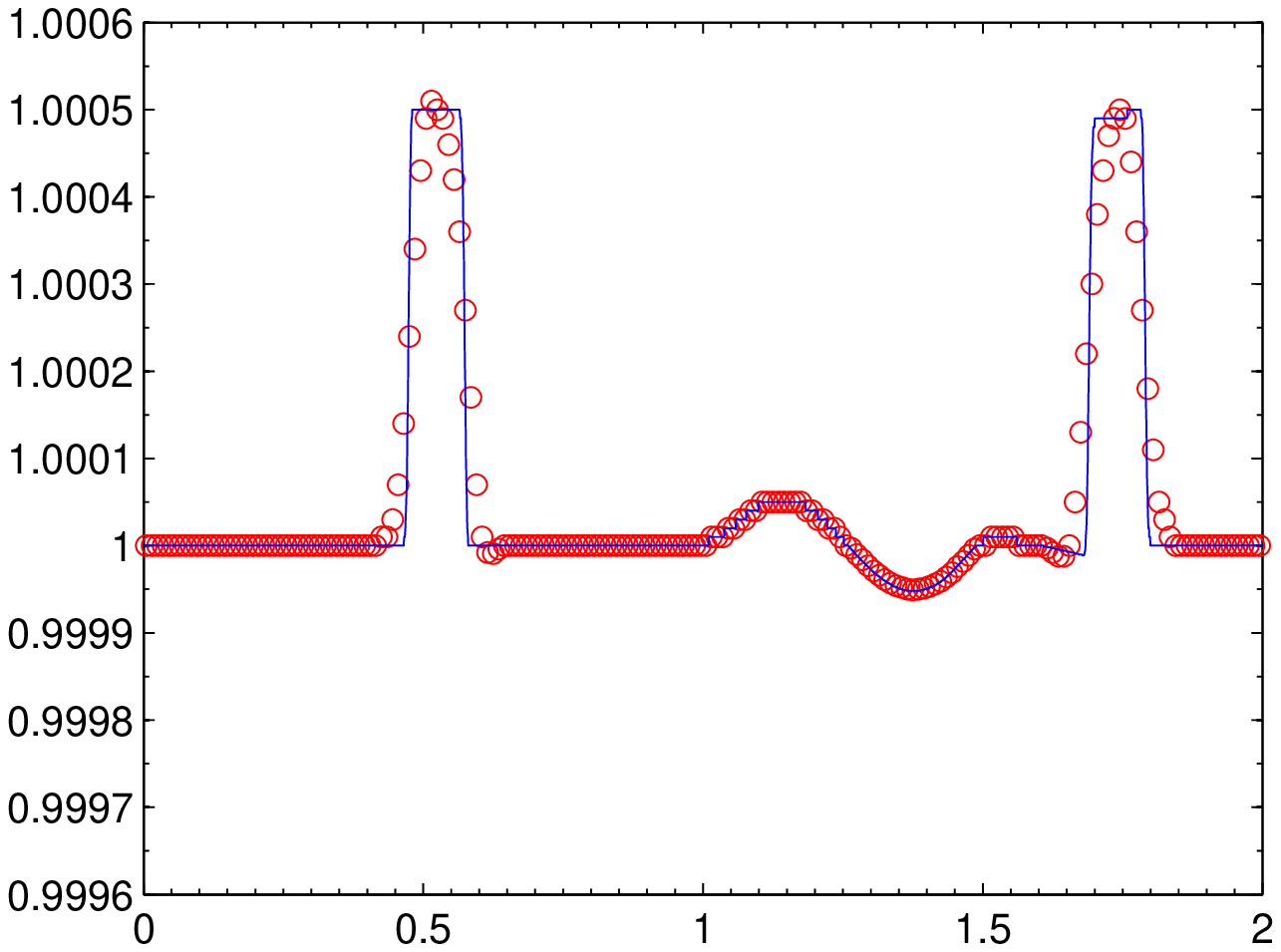}}
\centerline{\includegraphics[width=2.65in]{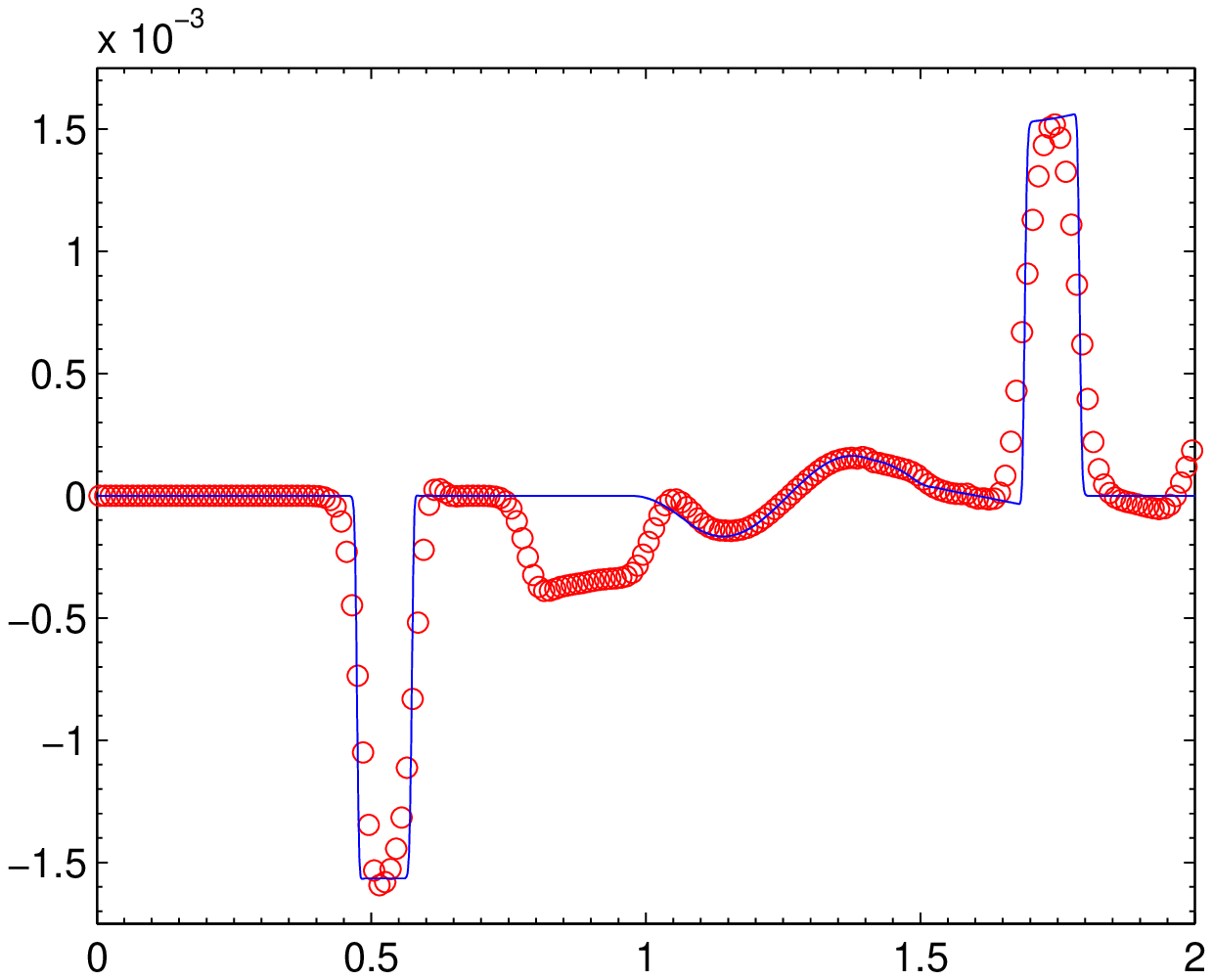}\includegraphics[width=2.65in]{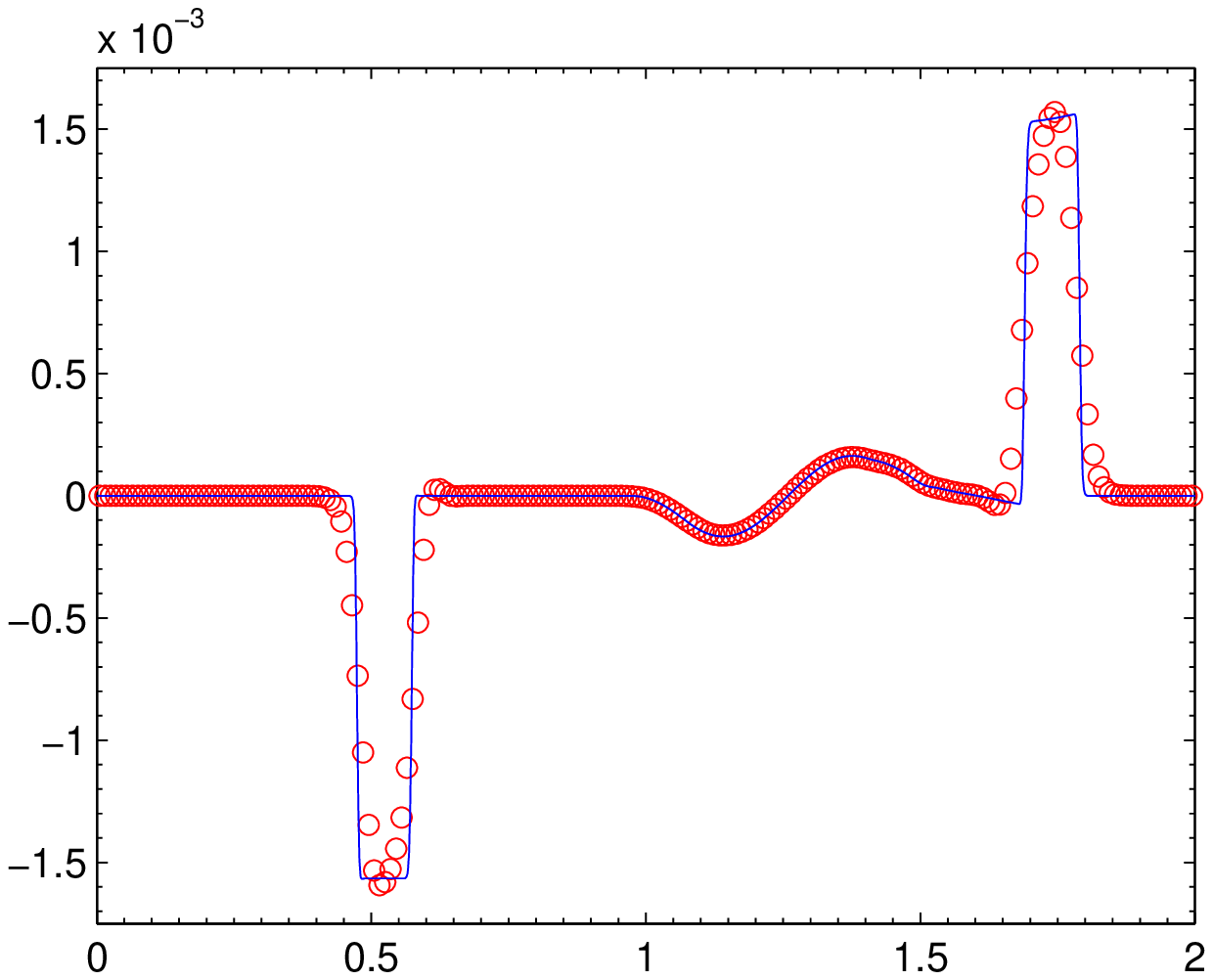}}
\caption{\sf Example \ref{ex73}: $\varepsilon=0.001$ (smaller perturbation). Solutions ($w$ in the top row, $hu$ in the bottom row) computed
by the non-well-balanced (left column) and well-balanced (right column) CSOC using uniform grids with 200 (circles) and 3000 (solid line,
reference solution) cells.\label{fig72}}
\end{figure}
\end{example}

\begin{example}[Dam Break over a Discontinuous Bottom]\label{ex74}
This problem is taken from \cite{Manning02} and \cite{XingShu05} to simulate a fast changing flow over a nonsmooth bottom. The computational
domain is $0\le x\le1500$, the initial data are
\begin{equation*}
(hu)(x,0)\equiv0,\quad w(x,0)=\left\{\begin{aligned}&20,&&\mbox{if}~x\le750,\\&15,&&\mbox{otherwise},\end{aligned}\right.
\end{equation*}
and absorbing boundary conditions are used at both ends of the computational domain. The bottom topography contains a rectangular bump and
is given by
\begin{equation*}
B(x)=\left\{\begin{aligned}&8,&&\mbox{if}~562.5\le x\le937.5,\\&0,&&\mbox{otherwise}.\end{aligned}
\right.
\end{equation*}
We compute the numerical solutions using 500 and 5000 uniform cells at two different final times: $t=15$ (Figure \ref{fig73}) and $t=55$
(Figure \ref{fig74}). As one can clearly see, the obtained results are very accurate and practically oscillation-free.
\begin{figure}[ht!]
\centerline{\includegraphics[width=2.65in]{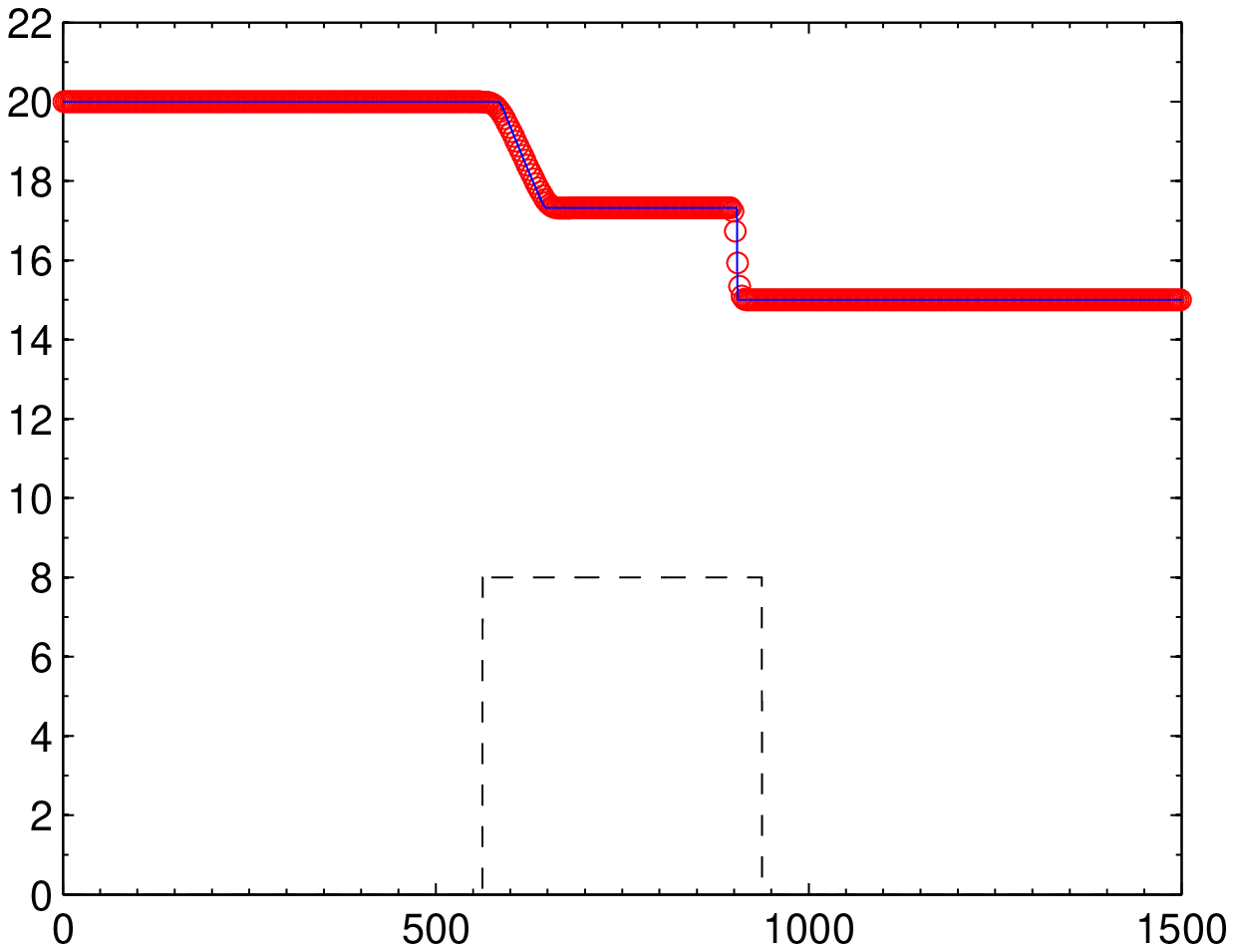}\includegraphics[width=2.65in]{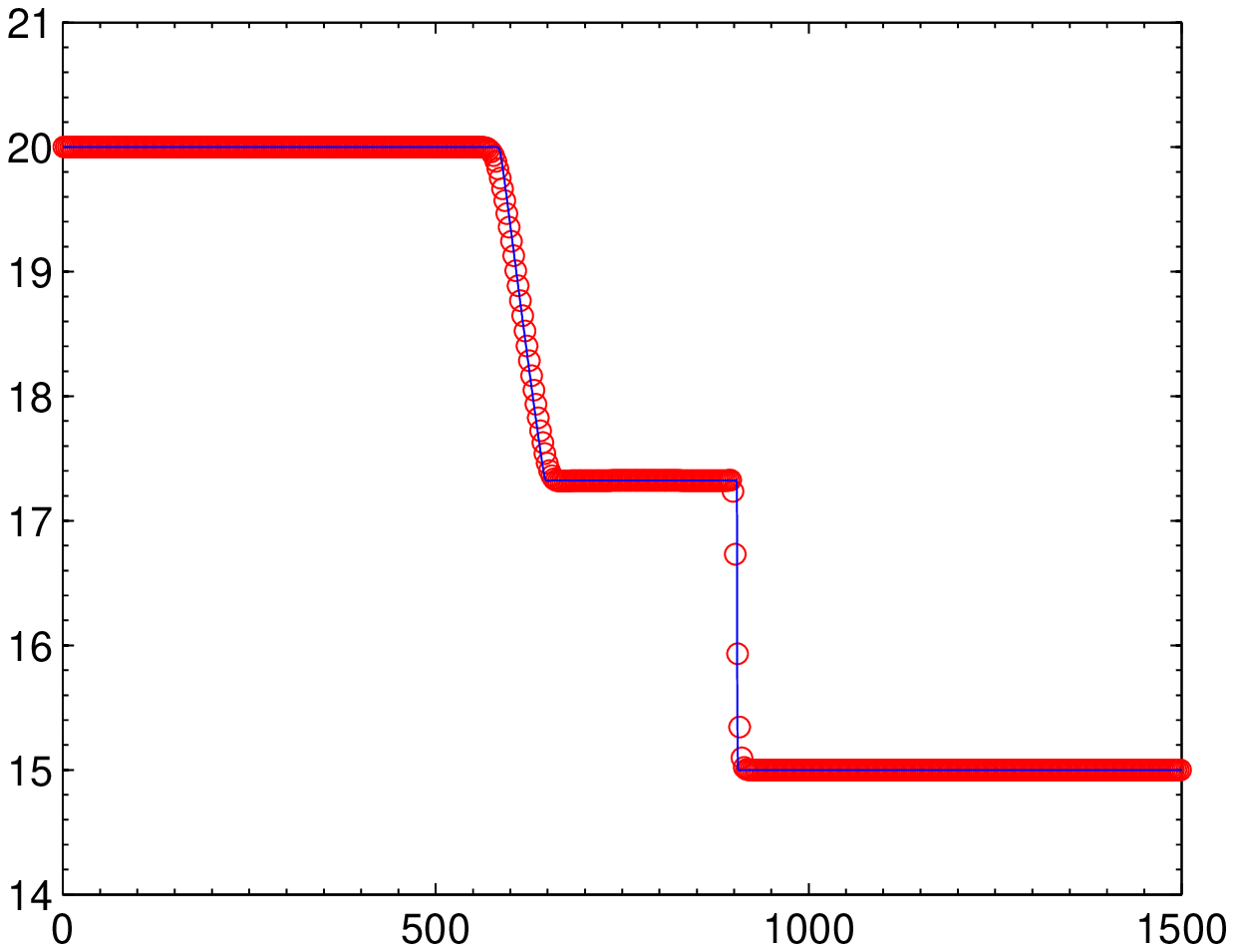}}
\centerline{\includegraphics[width=2.65in]{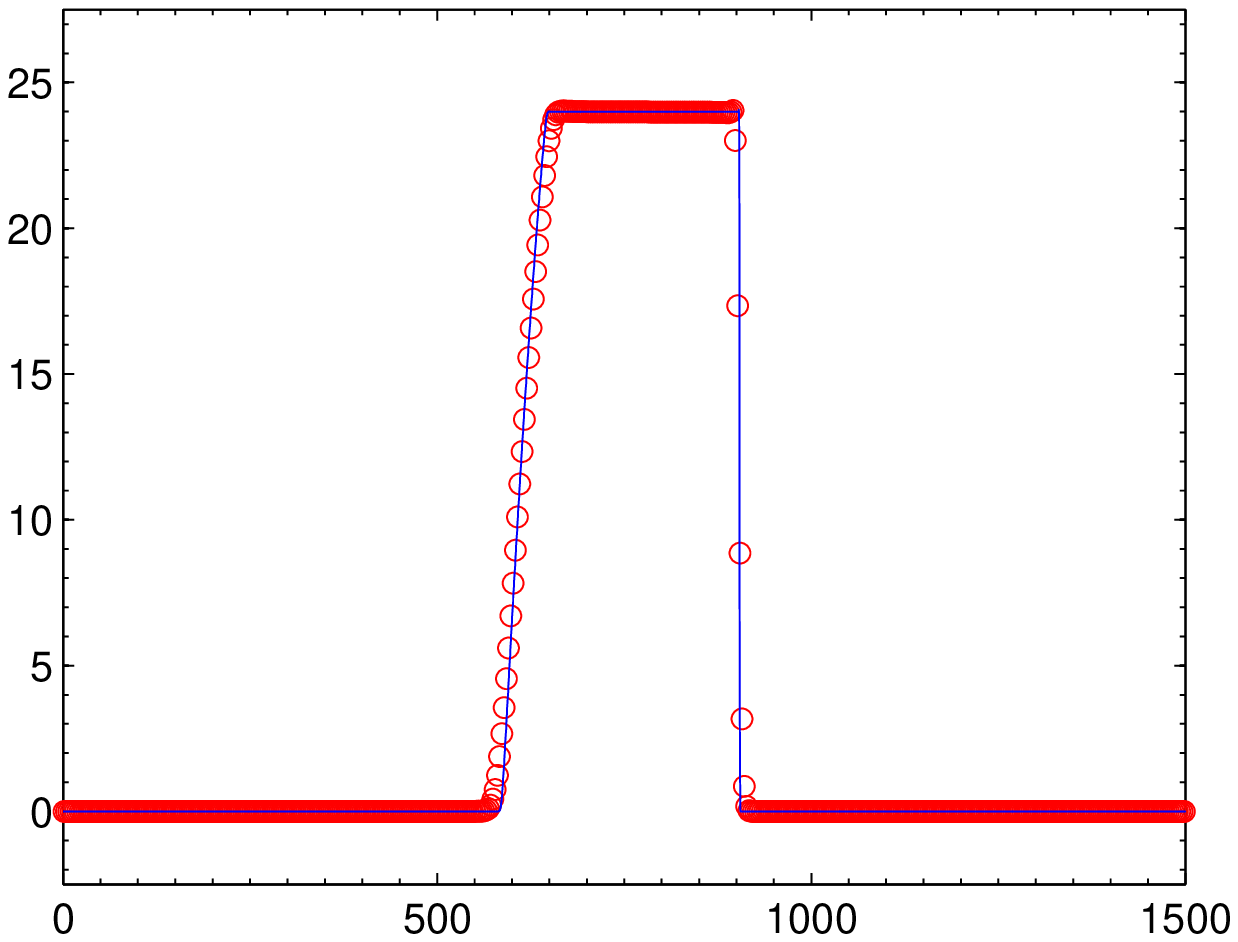}\includegraphics[width=2.65in]{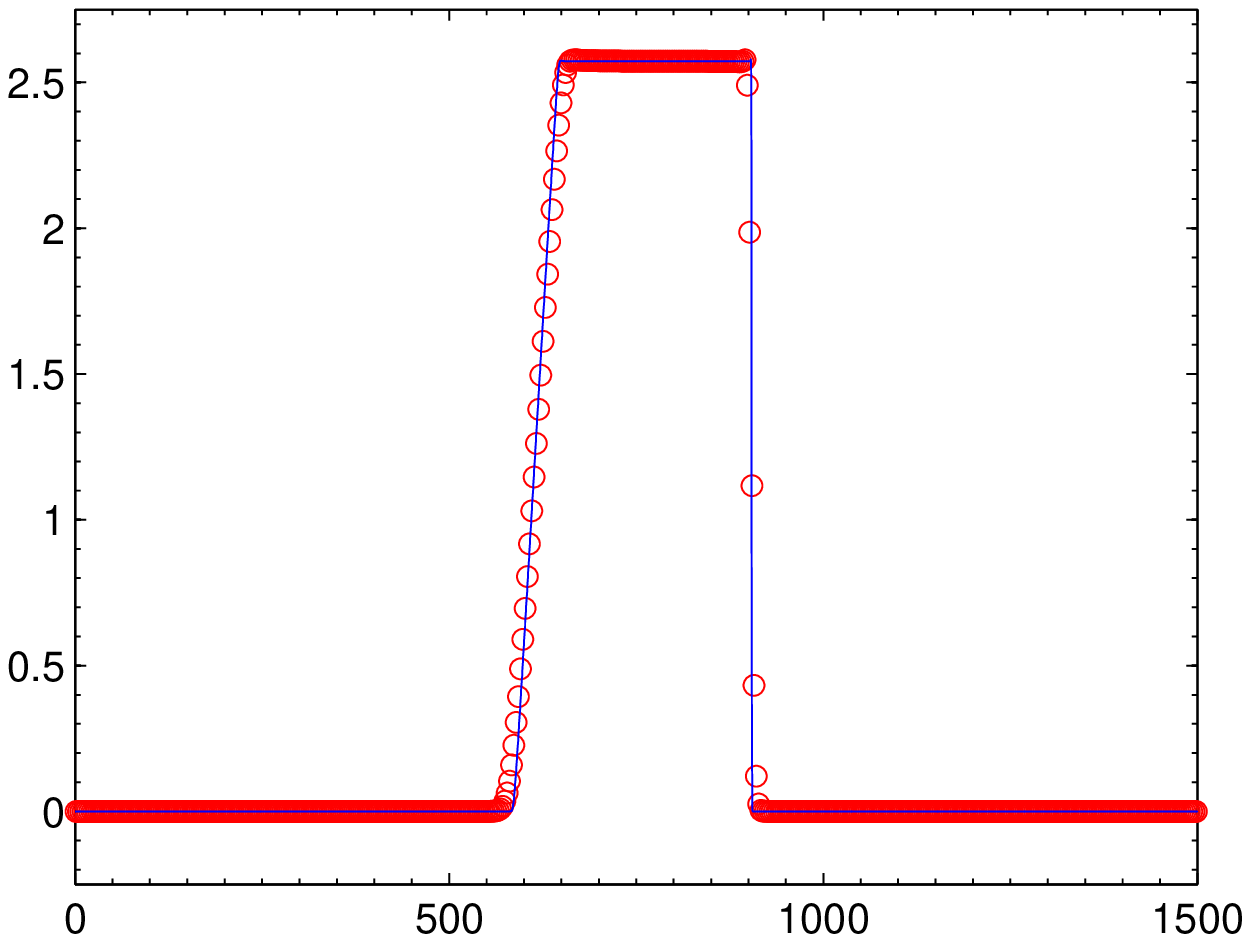}}
\caption{\sf Example \ref{ex74}: $w(x,15)$ together with $B(x)$ (top left), $w(x,15)$ (top right), $hu(x,15)$ (bottom left) and $u(x,15)$
(bottom right), computed using uniform grids with 500 (circles) and 5000 (solid line) cells. The bottom topography $B$ is plotted with the
dashed line.\label{fig73}}
\end{figure}
\begin{figure}[ht!]
\centerline{\includegraphics[width=2.65in]{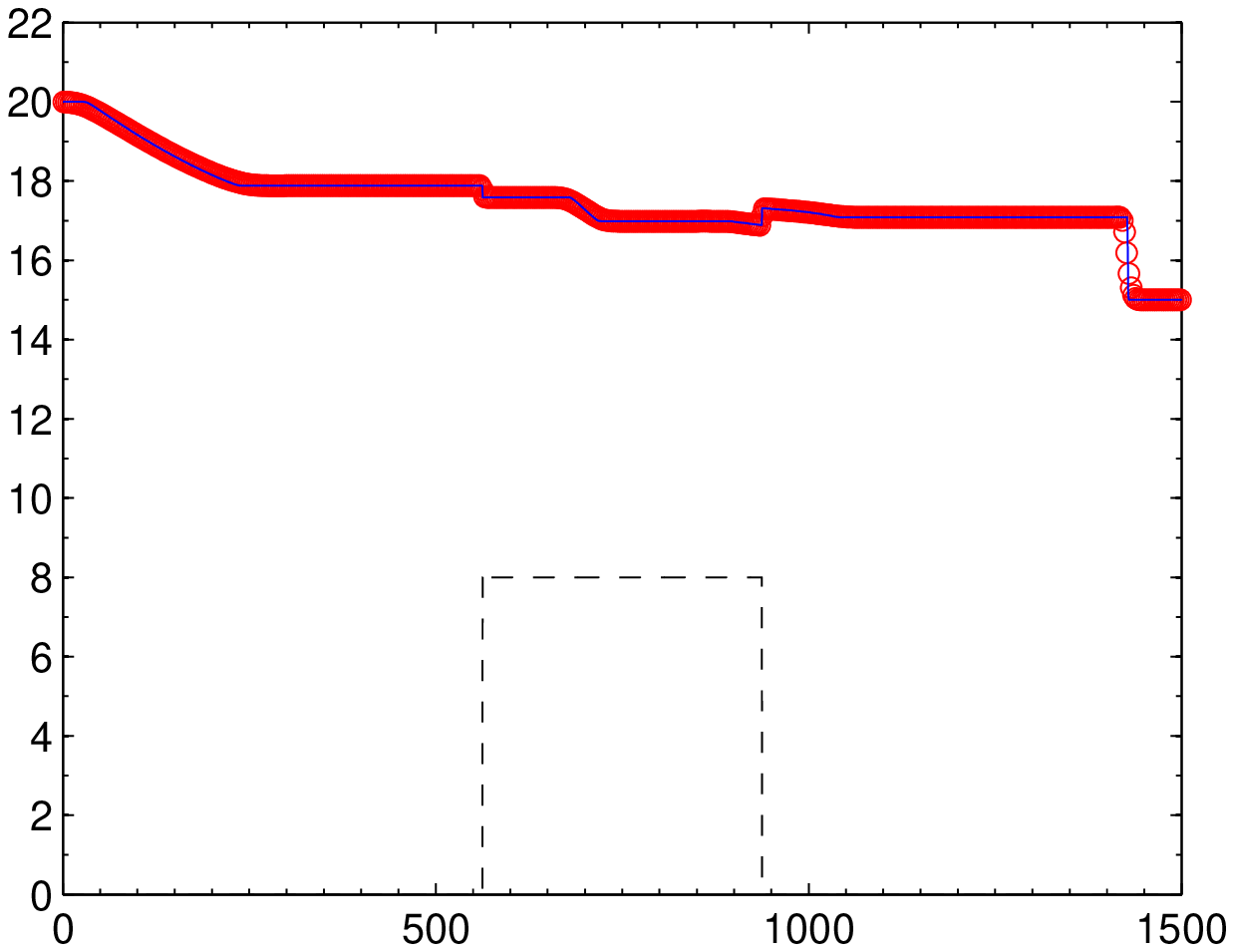}\includegraphics[width=2.65in]{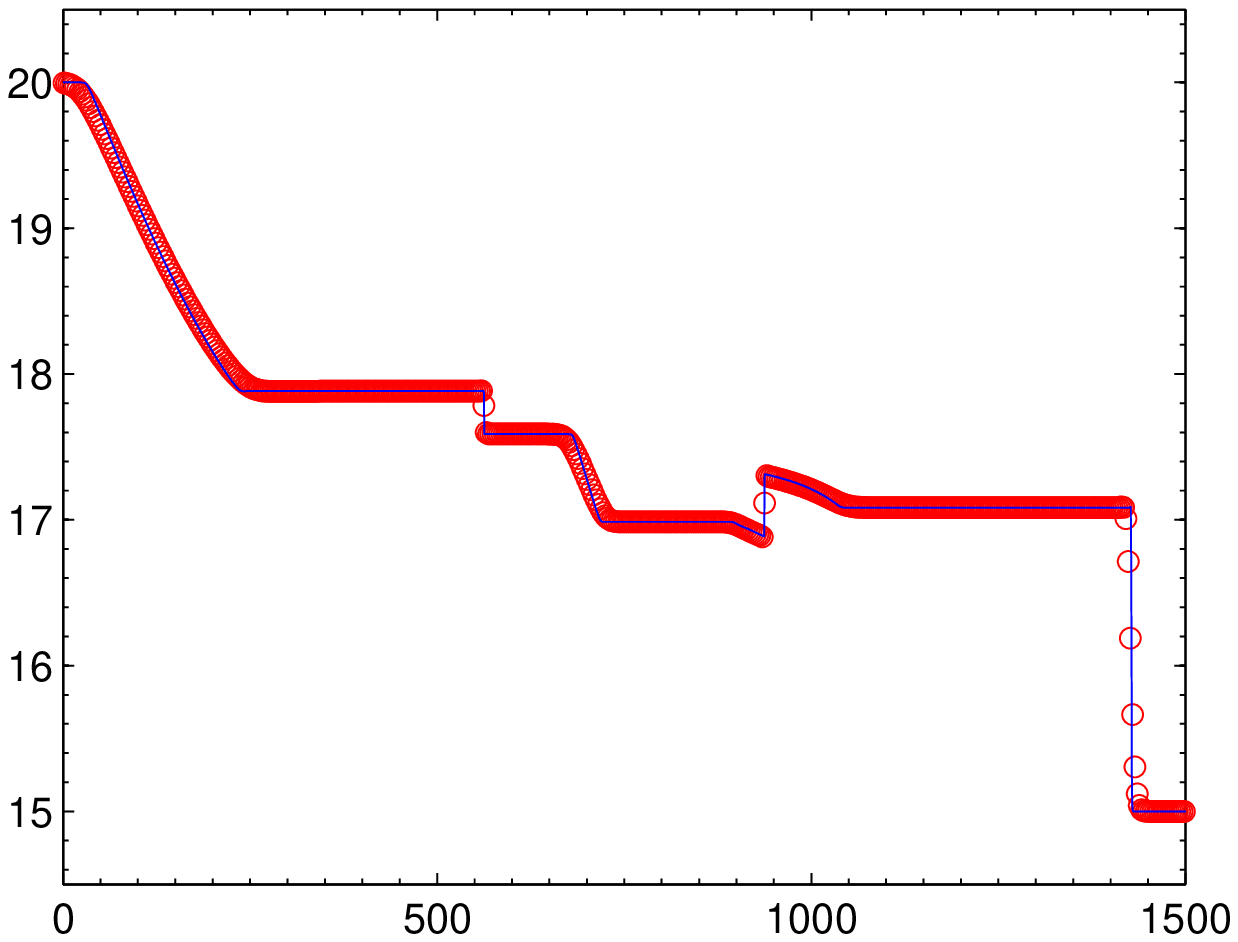}}
\centerline{\includegraphics[width=2.65in]{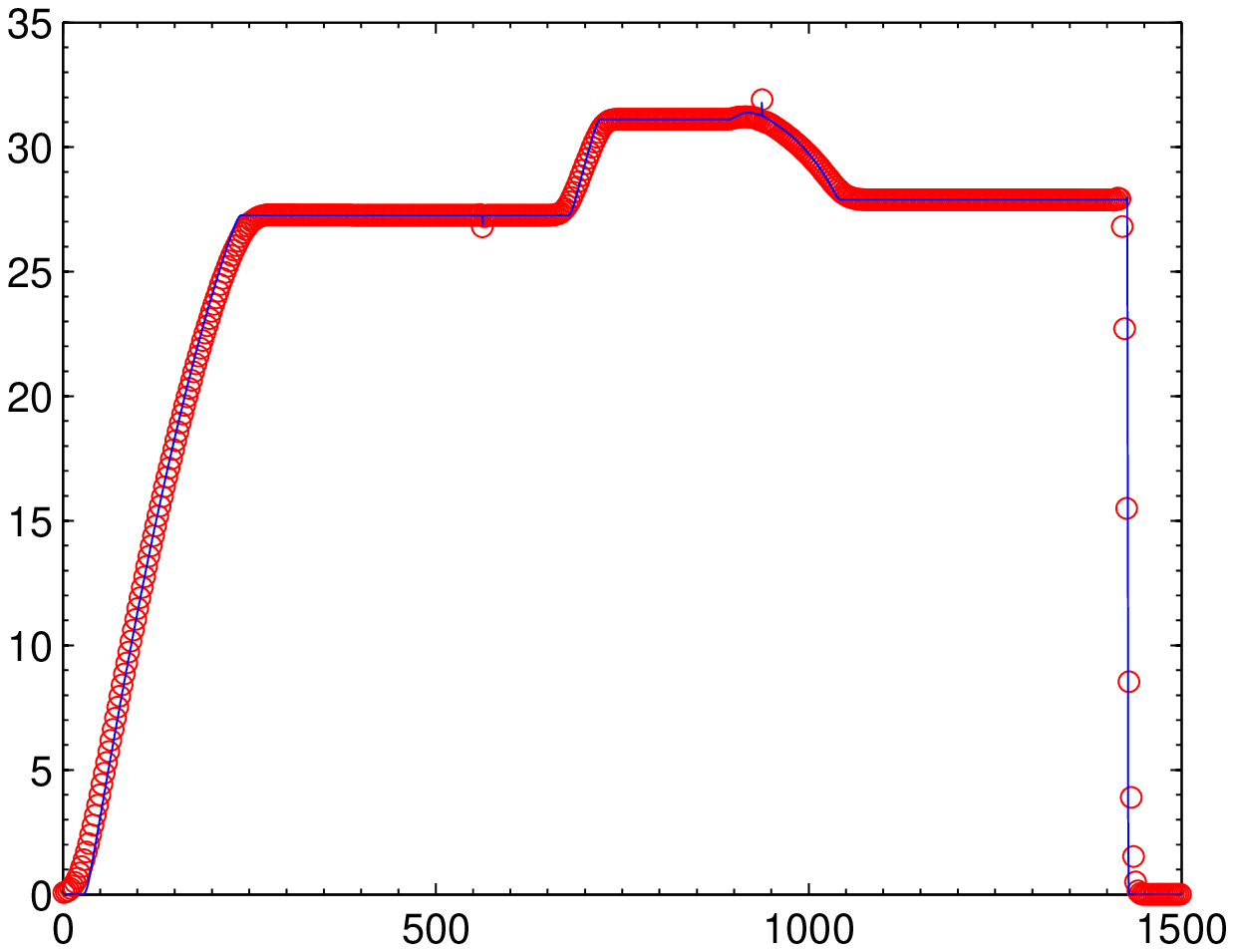}\includegraphics[width=2.65in]{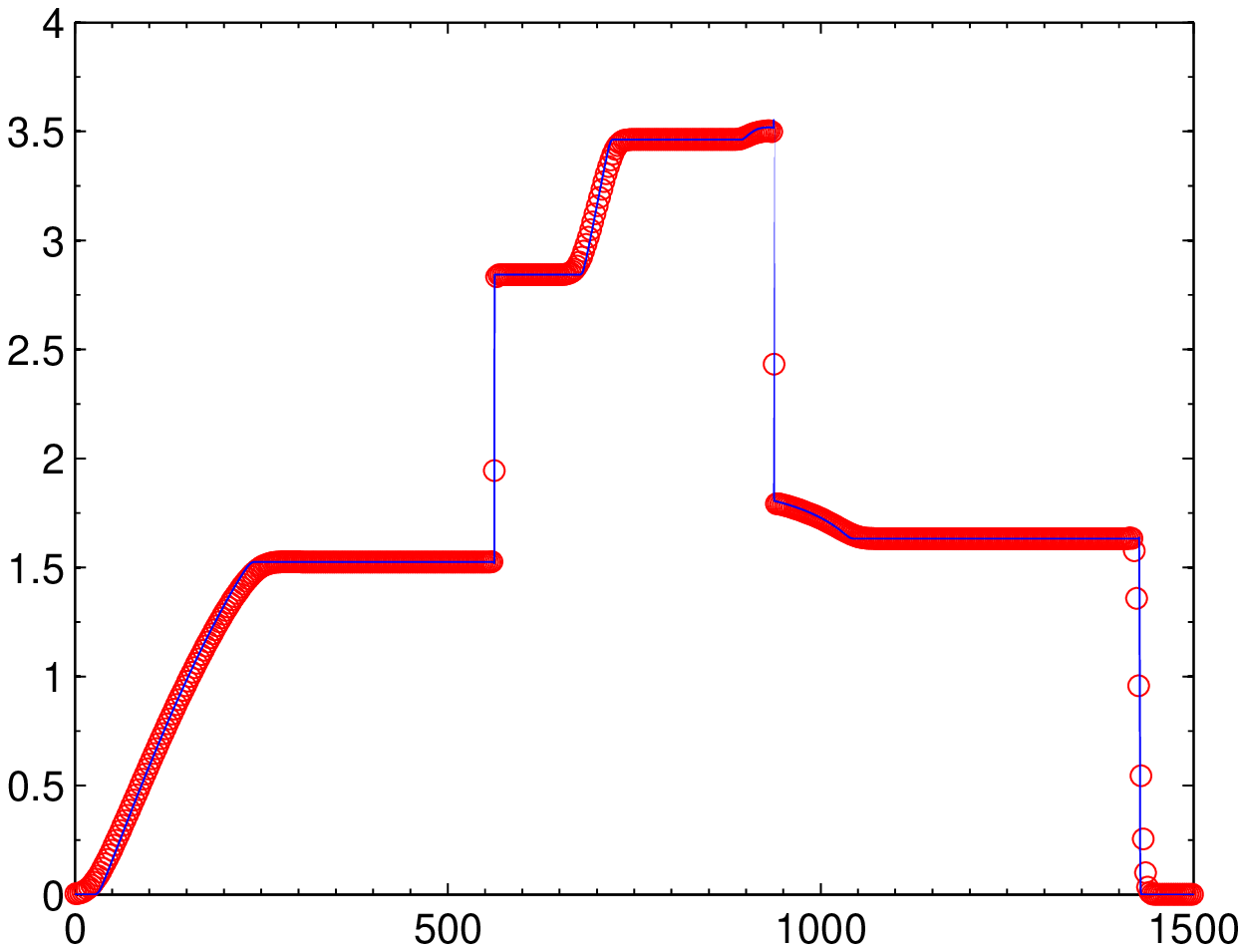}}
\caption{\sf Example \ref{ex74}: $w(x,55)$ together with $B(x)$ (top left), $w(x,55)$ (top right), $hu(x,55)$ (bottom left) and $u(x,55)$
(bottom right), computed using uniform grids with 500 (circles) and 5000 (solid line) cells. The bottom topography $B$ is plotted with the
dashed line.\label{fig74}}
\end{figure}
\begin{remark}
In this problem, the bottom topography $B$ is a discontinuous step function, which needs to be treated carefully due to the appearance of
the Dirac delta function in the source term $g[\xbar w(t)-w]B_x$. Therefore, the source term must be treated in a special way. We follow the
approach in \cite{KurganovPetrova07} and replace $B$ with its continuous piecewise linear approximation,
\begin{equation*}
\widetilde B(x)=B_{i-\frac{1}{2}}+(B_{i+\frac{1}{2}}-B_{i-\frac{1}{2}})\cdot\frac{x-x_{i-\frac{1}{2}}}{\Delta x},\quad
\forall x\in C_i,\ \forall i,
\end{equation*}
where
$$
B_{i+\frac{1}{2}}:=\frac{B(x_{i+\frac{1}{2}}+0)+B(x_{i+\frac{1}{2}}-0)}{2}.
$$
Notice that $\widetilde B\to B$ as $\Delta x\to0$.
\end{remark}
\end{example}

\begin{example}[Saint-Venant System with Manning's Friction]\label{ex76a}
In this example, we consider the 1-D Saint-Venant system with Manning's friction term (see, e.g., \cite{Fla,Manning02}):
\begin{equation}
\left\{\begin{aligned}&h_t+(hu)_x=0,\\&(hu)_t+\Big(hu^2+\frac{1}{2}gh^2\Big)_x=-ghB_x-g\frac{M^2}{h^{1/3}}u|u|,\end{aligned}\right.
\label{6.1}
\end{equation}
where $M=M(x)$ is a given Manning's friction coefficient.

We note that in addition to the ``lake at rest'' steady states, the system \eqref{6.1} admits another physically relevant set of
steady-state solutions corresponding to the water flowing down a slanted surface of a constant slope (see, e.g., \cite{CGP10,CVC12,CCKW}).
However, in this paper, we only consider the ``lake at rest'' steady states and therefore, the equilibrium variables are the same as for the
original Saint-Venant system \eqref{1DSV}, namely, $\bm{a}:=(w,hu)^T$.

We now apply Algorithm \ref{alg31} and rewrite the system \eqref{6.1} as
\begin{equation}
\left\{\begin{aligned}&w_t+(hu)_x=0,\\&(hu)_t+\bigg(\frac{(hu)^2}{w-B}+g[\xbar w(t)-w]B+\frac{g}{2}w^2\bigg)_x=g[\xbar w(t)-w]B_x-
g\frac{M^2(hu)|hu|}{(w-B)^{7/3}},\end{aligned}\right.
\label{6.2}
\end{equation}
and obtain a well-balanced scheme by a direct application of the CSOC to \eqref{6.2}. To illustrate the performance of the resulting scheme,
we follow \cite{Manning02} and consider the same setting as in Example \ref{ex74}, but with Manning's friction term with $M(x)\equiv0.1$.
The solutions computed at times $t=15$ and $t=55$ are shown in Figures \ref{fig76a} and \ref{fig77a}, respectively. As one can clearly see,
the obtained results are well-resolved and almost non-oscillatory, and the coarse and fine grid solutions are in a very good agreement.
\begin{figure}[ht!]
\centerline{\includegraphics[width=2.65in]{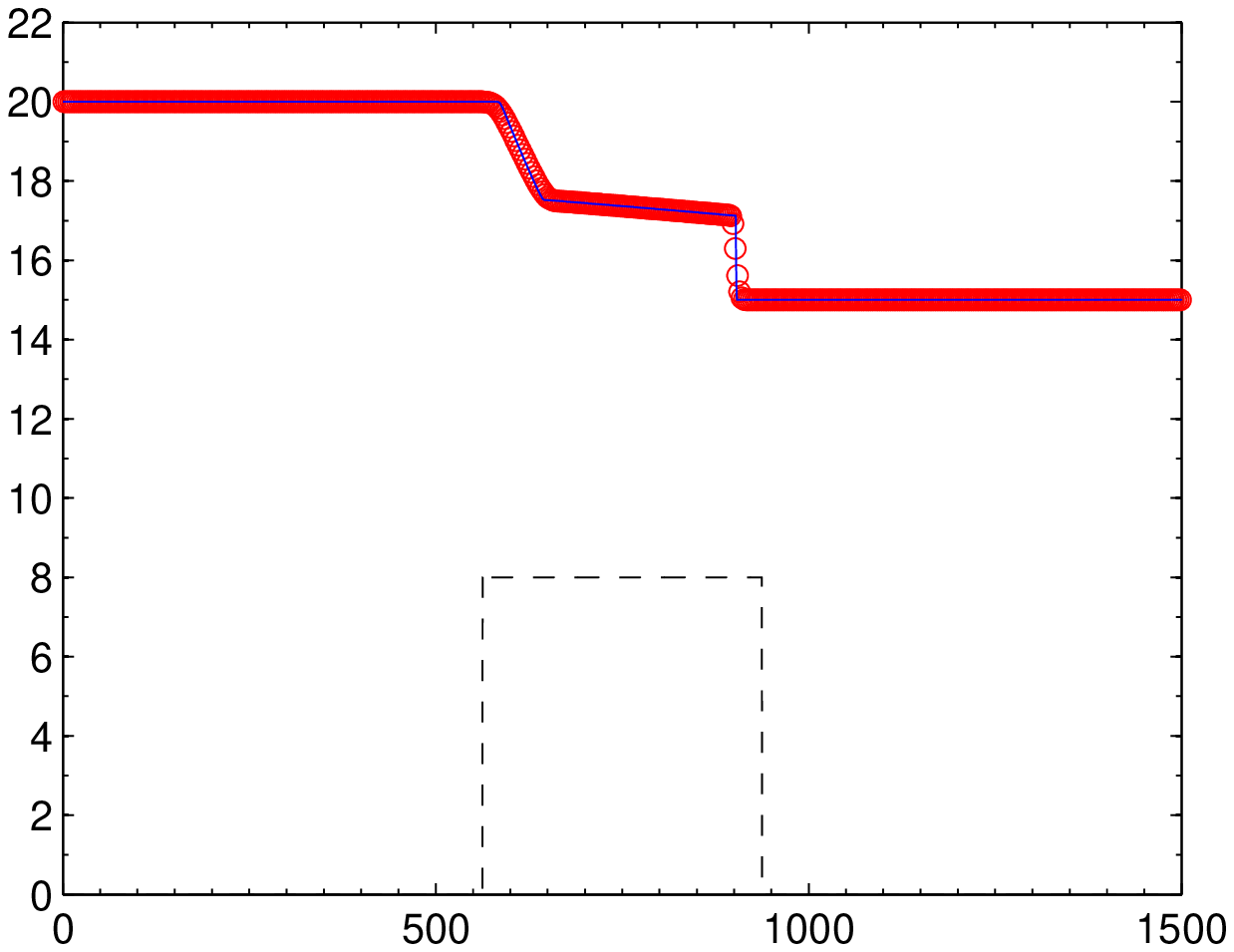}\includegraphics[width=2.65in]{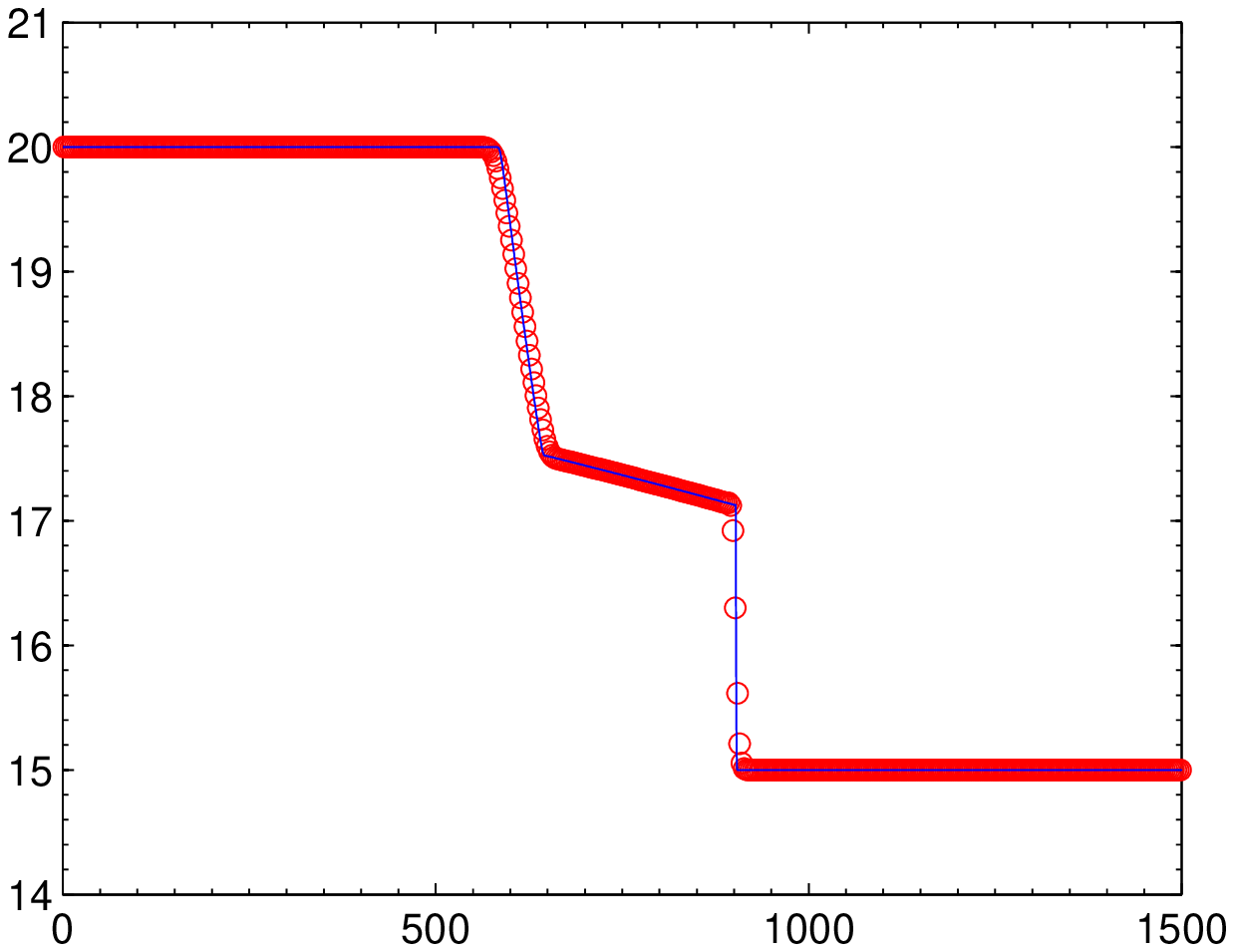}}
\centerline{\includegraphics[width=2.65in]{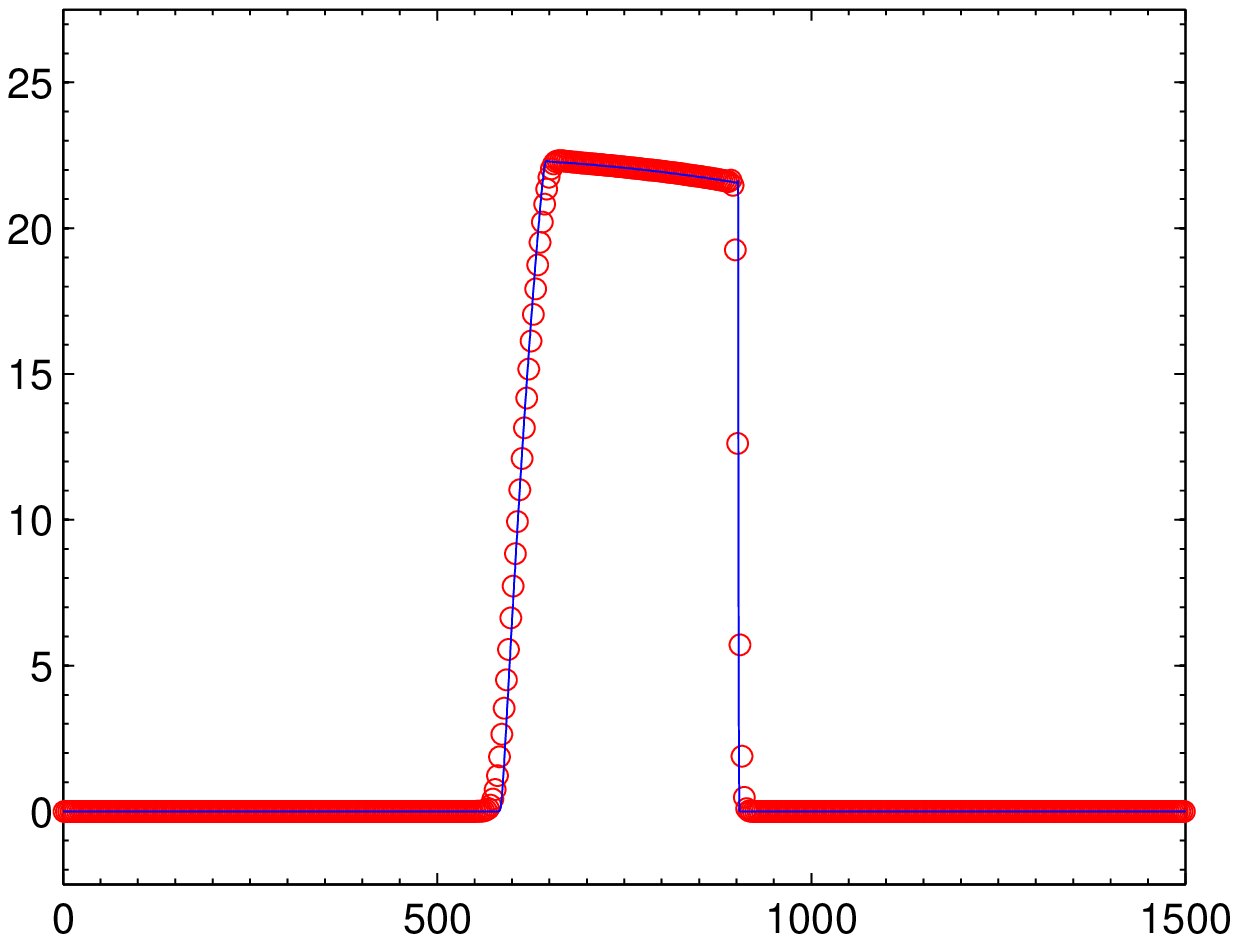}\includegraphics[width=2.65in]{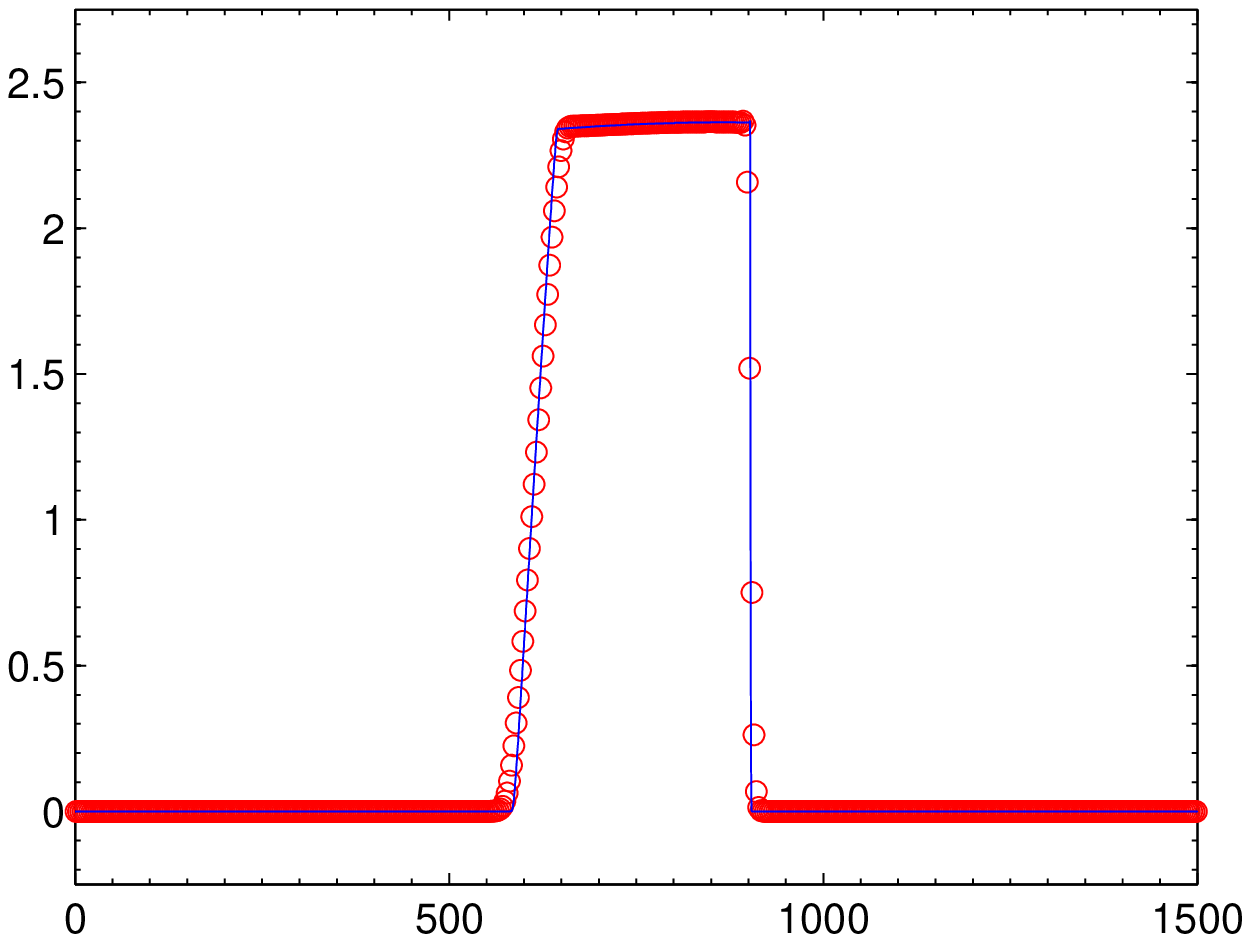}}
\caption{\sf Example \ref{ex76a}: $w(x,15)$ together with $B(x)$ (top left), $w(x,15)$ (top right), $hu(x,15)$ (bottom left) and $u(x,15)$
(bottom right), computed using uniform grids with 500 (circles) and 5000 (solid line) cells. The bottom topography $B$ is plotted with the
dashed line.\label{fig76a}}
\end{figure}
\begin{figure}[ht!]
\centerline{\includegraphics[width=2.65in]{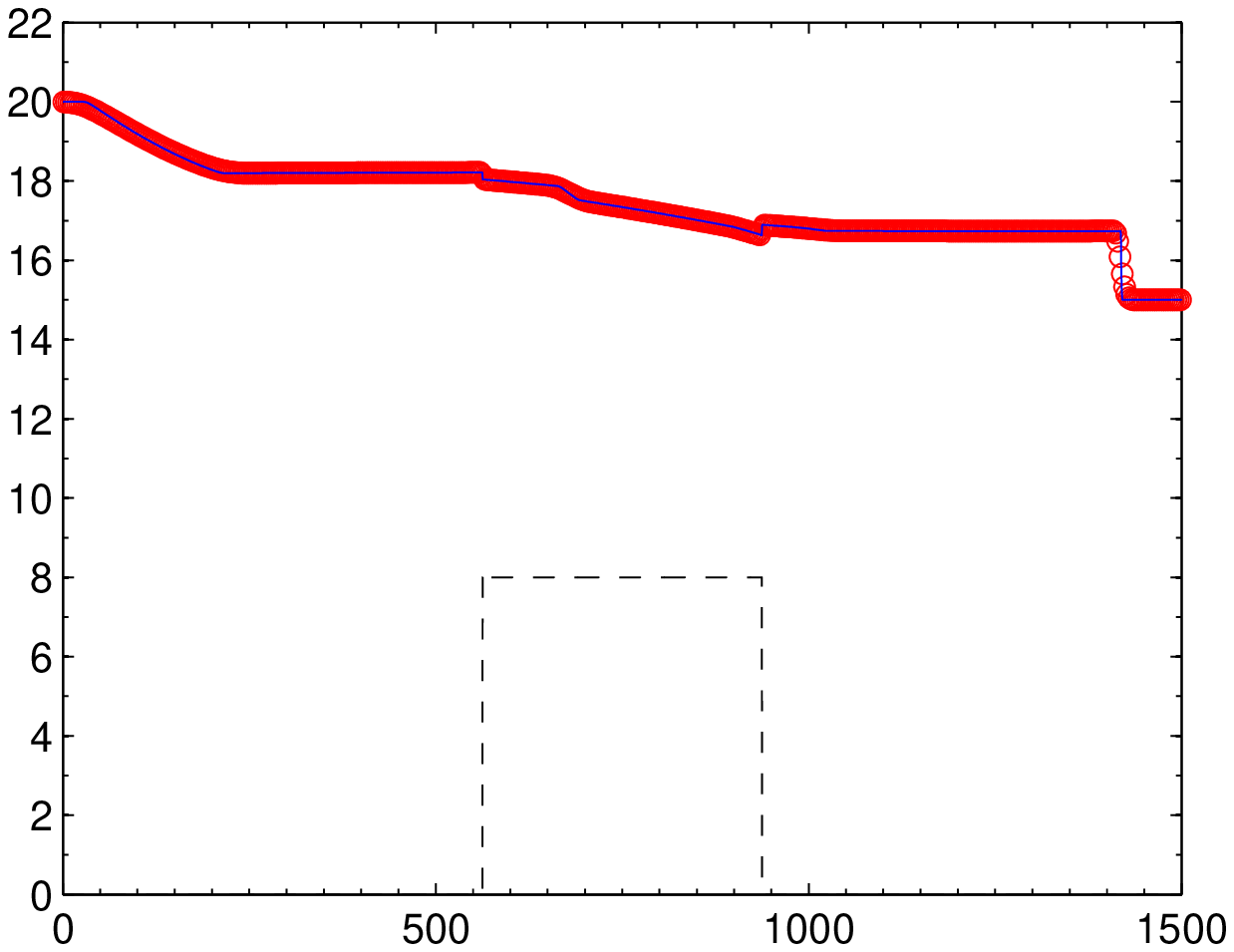}\includegraphics[width=2.65in]{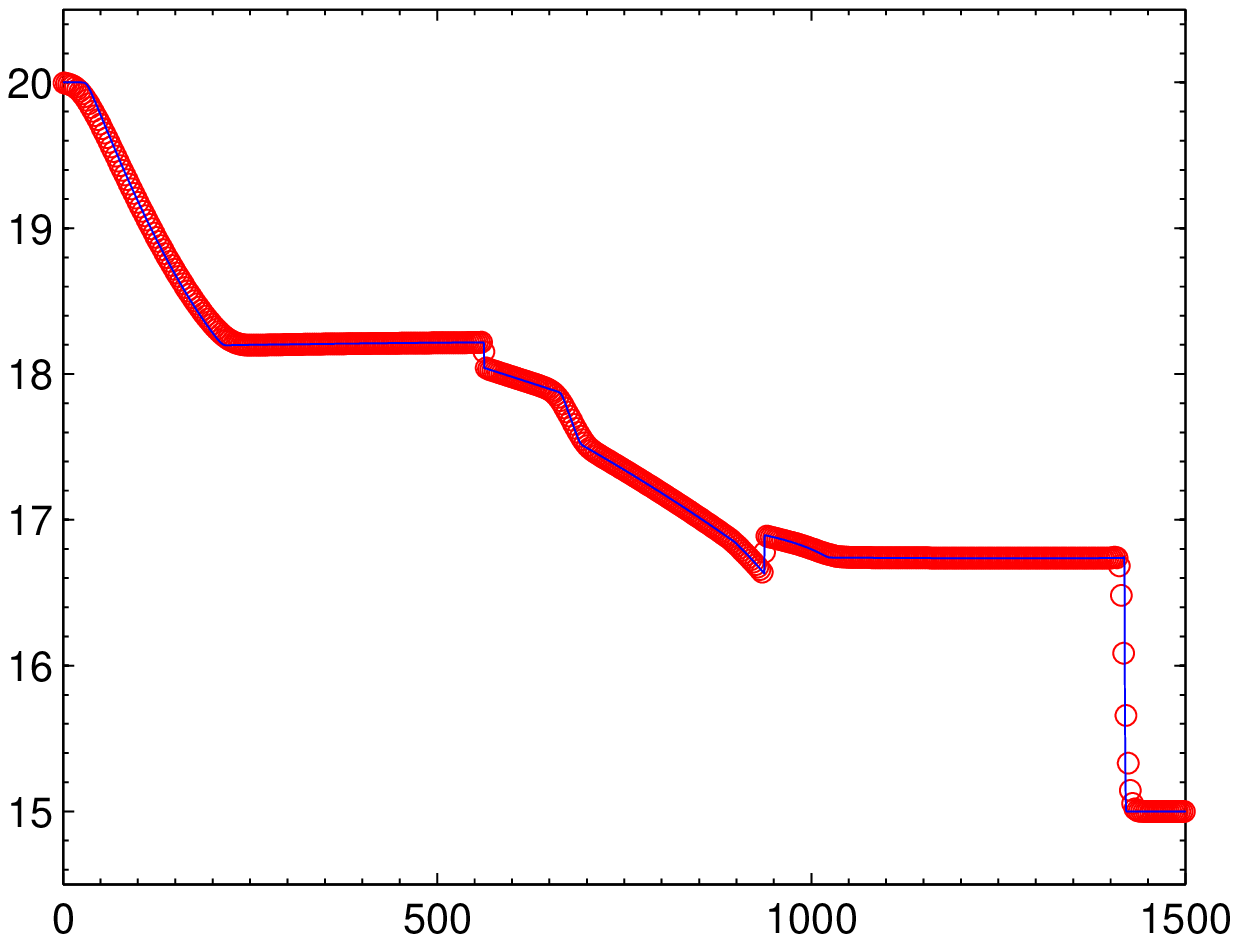}}
\centerline{\includegraphics[width=2.65in]{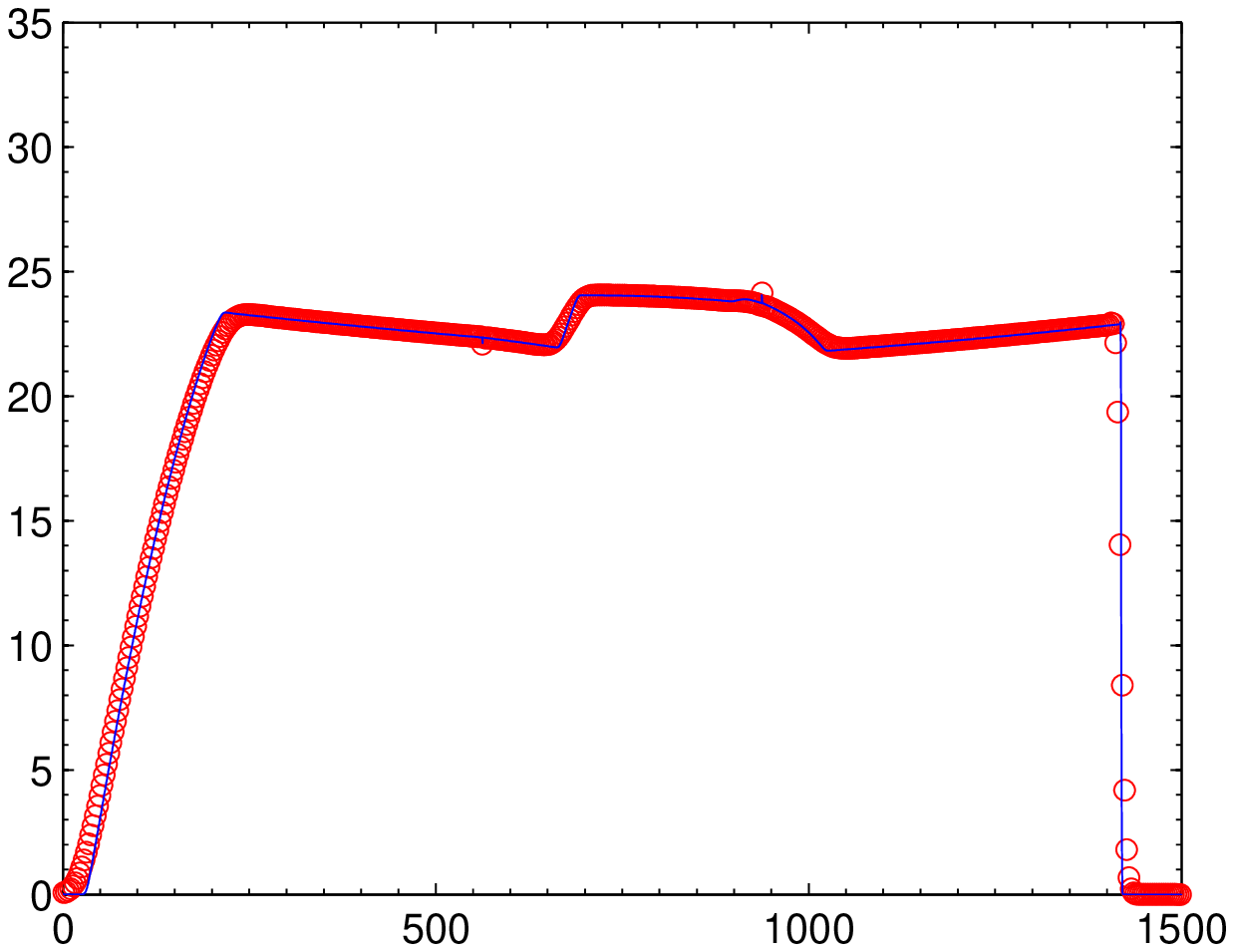}\includegraphics[width=2.65in]{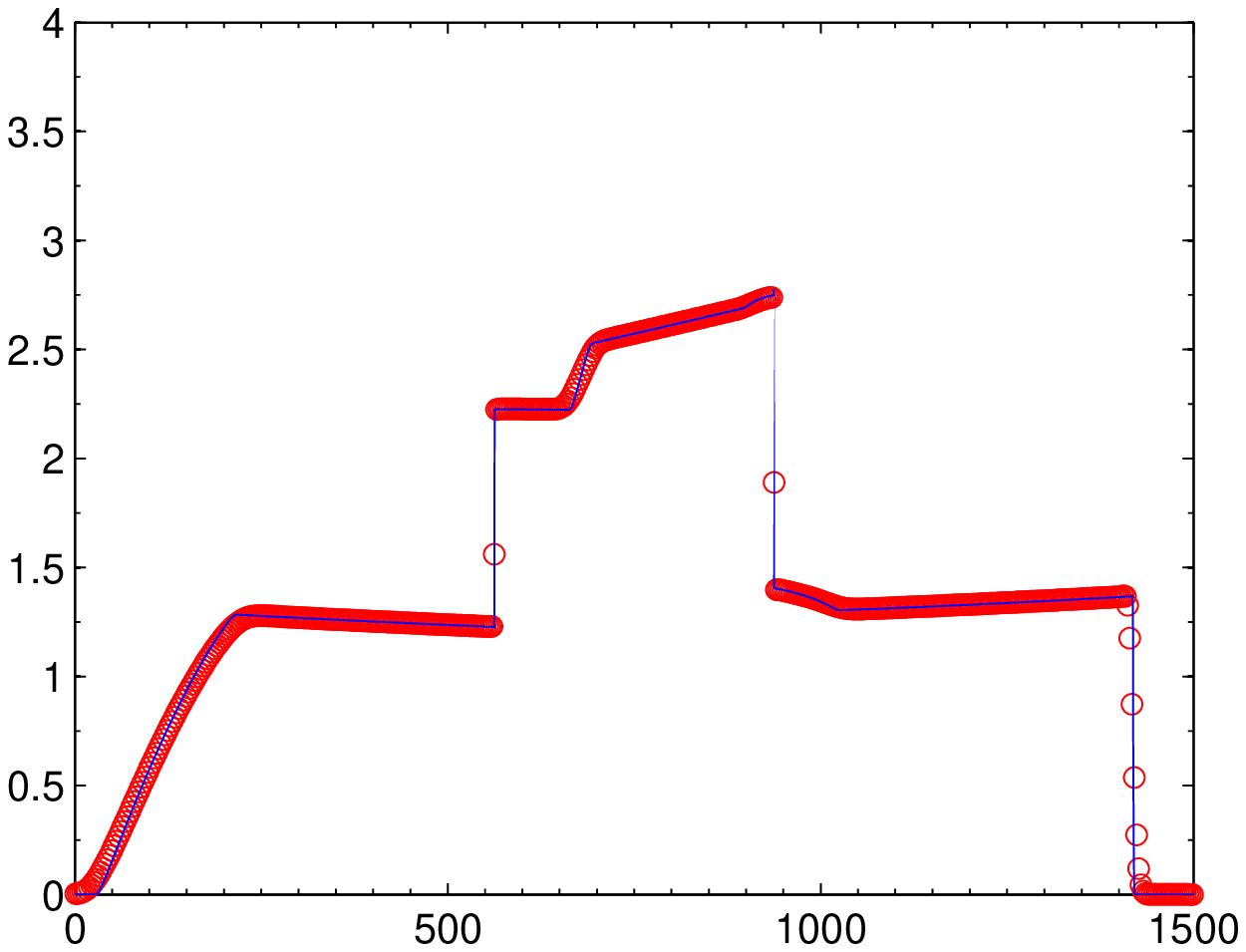}}
\caption{\sf Example \ref{ex76a}: $w(x,55)$ together with $B(x)$ (top left), $w(x,55)$ (top right), $hu(x,55)$ (bottom left) and $u(x,55)$
(bottom right), computed using uniform grids with 500 (circles) and 5000 (solid line) cells. The bottom topography $B$ is plotted with the
dashed line.\label{fig77a}}
\end{figure}
\end{example}

We would like to point out that Manning's friction is only a damping term which does not smear the discontinuities. Compared with the
numerical results in Example \ref{ex74}, one can see that the effect of Manning's friction is that the original horizontal line above the
bump becomes oblique and the velocity magnitude decreases, which are typical effects of a damping term. Our results are in good agreements
with the results reported in \cite{Manning02}, where exactly the same phenomenon has been observed.

\begin{example}[Steady Flows over a Hump]\label{ex75}
In this example, we study steady states with the nonzero discharge $hu$. The properties of such flows depend on the bottom topography and
free-stream Froude number $Fr=u/\sqrt{gh}$. If $Fr<1$ (subcritical flow) or $Fr>1$ (supercritical flow) everywhere, then the steady-state
solution will be smooth. Otherwise, the flow is transcritical with transitions at the points where $Fr$ passes through 1, and thus one of
the eigenvalues $u\pm\sqrt{gh}$ of the Jacobian matrix passes through zero. In such case, the steady-state solution may contain a stationary
shock. Steady flows over a hump are classical benchmarks for transcritical and subcritical steady flows, and are widely used to test
numerical schemes for the shallow water system, see, for example, \cite{KurganovLevy02,LeVeque98,Vazquez99,XingShu05}.

The computational domain is $0\le x\le25$, and the initial data are
\begin{equation*}
w(x,0)\equiv0.5,\quad(hu)(x,0)\equiv0.
\end{equation*}
The bottom topography contains a hump and is given by
\begin{equation*}
B(x)=\left\{\begin{aligned}&0.2-0.05(x-10)^2,&&\mbox{if}~8\le x\le12,\\&0,&&\mbox{otherwise}.\end{aligned}
\right.
\end{equation*}
The nature of the solution depends on the boundary condition: The flow can be subcritical or transcritical with or without a stationary
shock. The final time is set to be $t=200$ by which all of the solutions reach their corresponding steady states.

\medskip
\noindent
{\em Case 1:} \underline{Subcritical Flow}.

\smallskip
\noindent
We set the following upstream and downstream boundary conditions: $(hu)(0,t)=4.42$ and $w(25,t)=2$. In Figure \ref{fig75} (left), we plot
the obtained Froude number $Fr$, which gradually increases to a large (but still smaller than 1) value above the hump and then gradually
decreases back to the original value. We compute the numerical solutions using 100 and 1000 uniform cells. As it can be seen in Figure
\ref{fig76}, the obtained solutions are in good agreement and both are non-oscillatory.

\smallskip
\noindent
{\em Case 2:} \underline{Transcritical Flow without a Stationary Shock}.

\smallskip
\noindent
We now take different upstream and downstream boundary conditions: $(hu)(0,t)=1.53$ and $w(25,t)=0.41$. In Figure \ref{fig75} (middle),
we plot the obtained Froude number $Fr$, which now gradually increases to a value greater than 1 above the hump and then remains constant.
Therefore, no stationary shocks appear on the surface. We compute the numerical solutions using 200 and 2000 uniform cells. As in the
subcritical case, the coarse and fine grid solutions are in a good agreement and both are practically non-oscillatory, see Figure
\ref{fig77}.

\smallskip
\noindent
{\em Case 3:} \underline{Transcritical Flow with a Stationary Shock}.

\smallskip
\noindent
In this case, the upstream and downstream boundary conditions are $(hu)(0,t)=0.18$ and $w(25,t)=0.33$. The obtained Froude number $Fr$ is
plotted in Figure \ref{fig75} (right). As in the previous case, the Froude number gradually increases to a value greater than 1 above the
hump, but then it jumps down to the value much smaller than 1. Therefore, a stationary shock appears on the surface. We compute the
numerical solutions, presented in Figure \ref{fig78}, using 100 and 1000 uniform cells. As in Cases 1 and 2, both solutions are almost
non-oscillatory and in a good agreement, and the stationary shock wave is sharply resolved.
\begin{figure}[ht!]
\centerline{\includegraphics[width=2.3in]{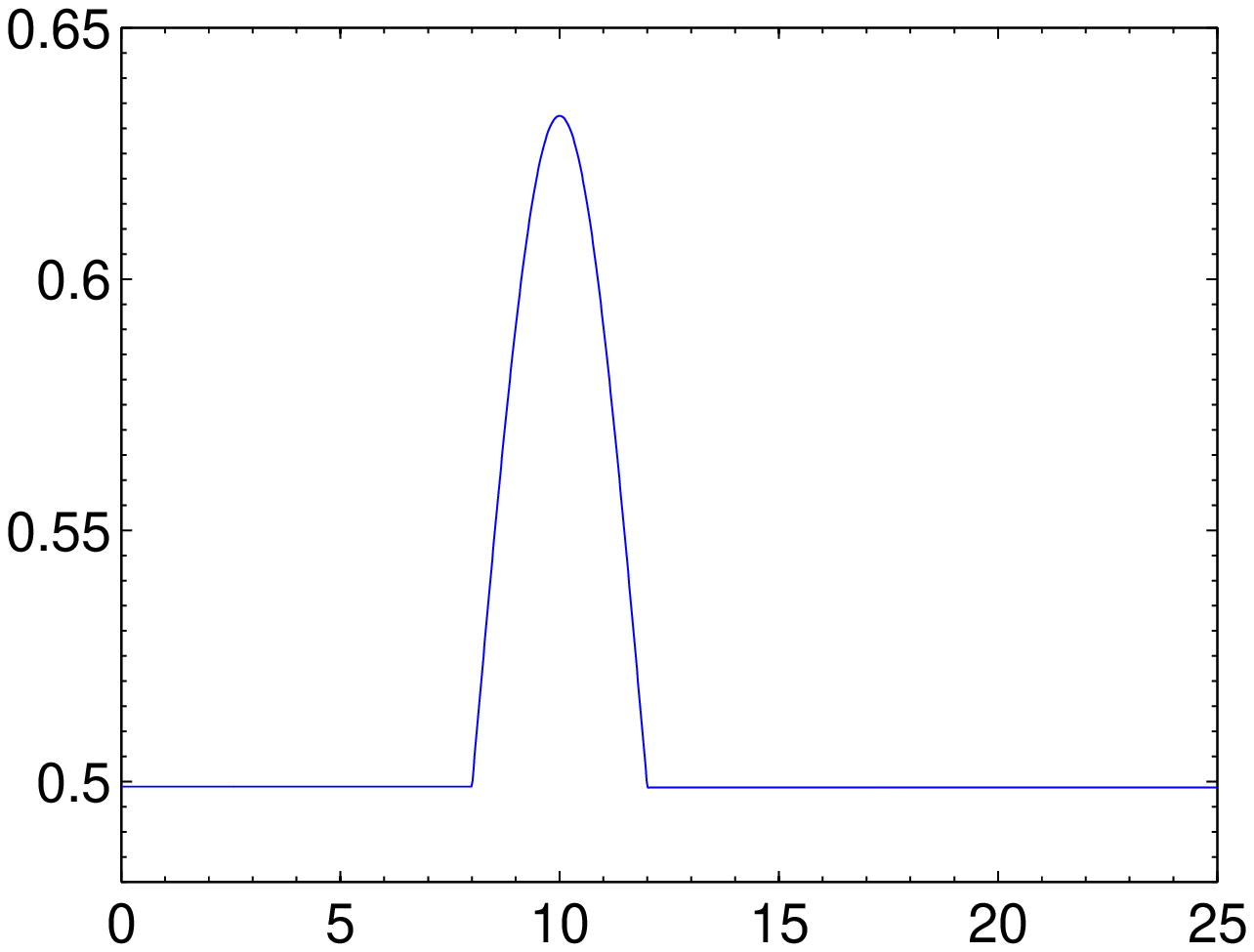}\includegraphics[width=2.3in]{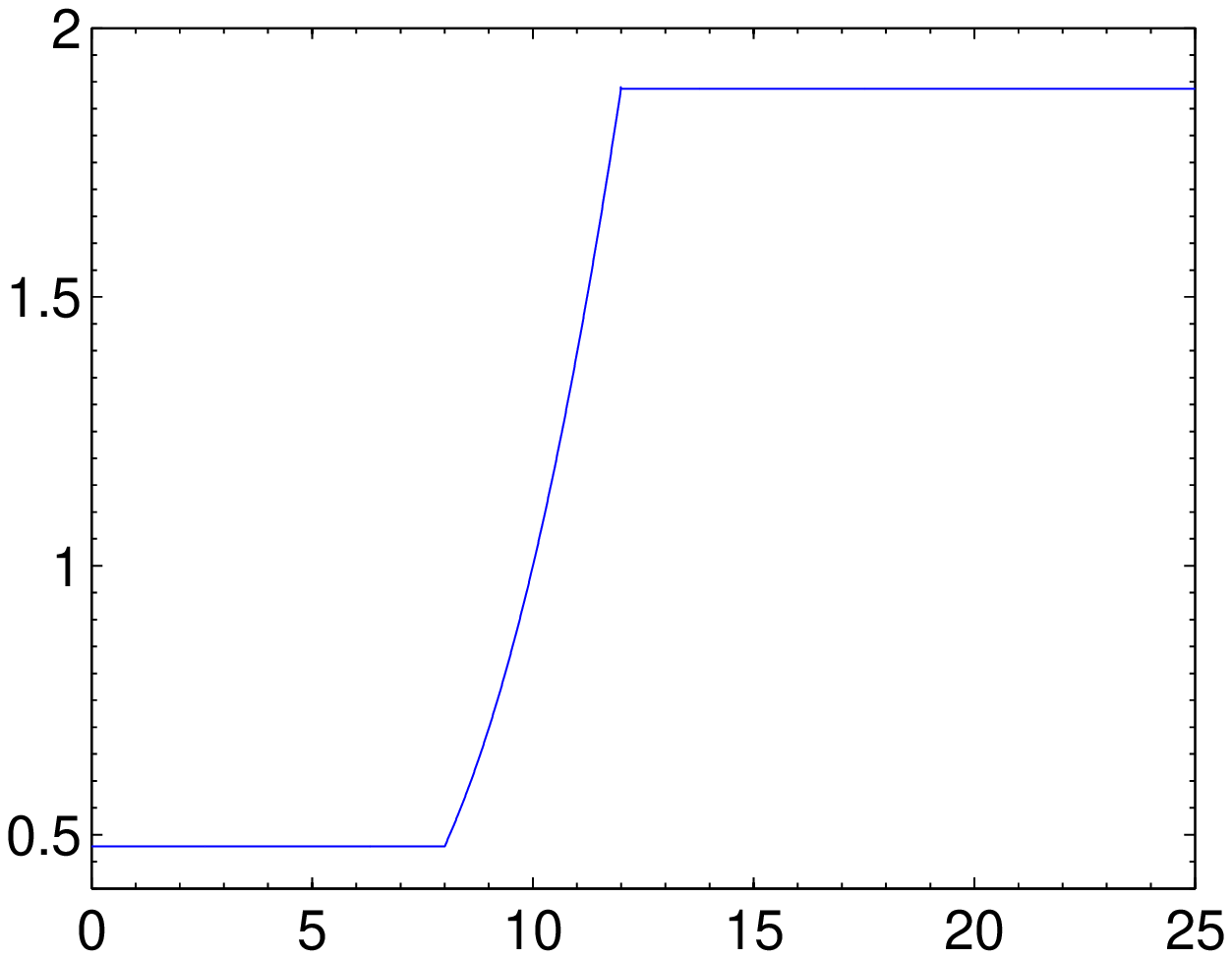}\includegraphics[width=2.3in]{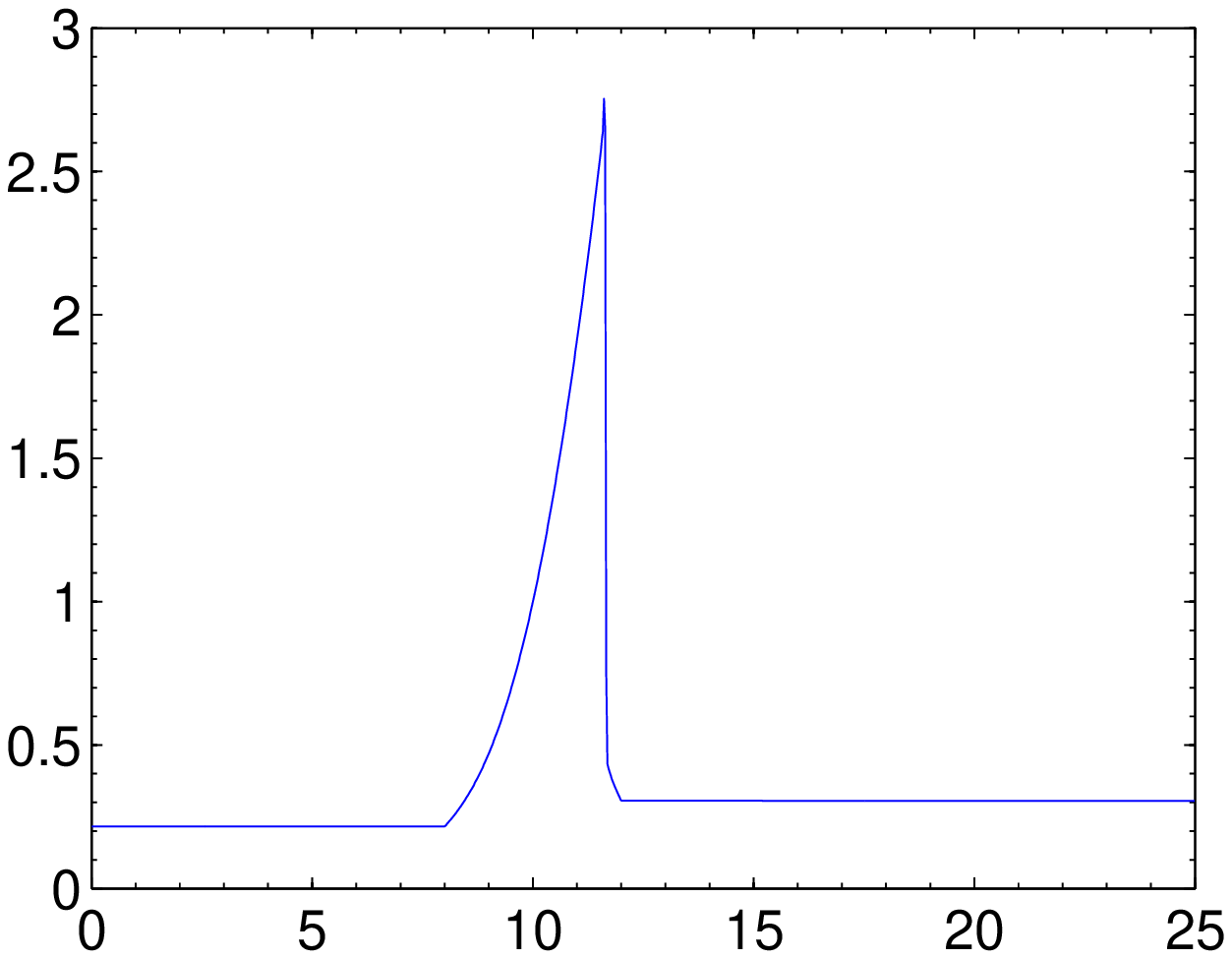}}
\caption{\sf Example \ref{ex75}: Froude number $Fr$ of the steady flows over a hump in the subcritical (left), transcritical without a 
stationary shock (middle) and transcritical with a stationary shock (right) cases.\label{fig75}}
\end{figure}
\begin{figure}[ht!]
\centerline{\includegraphics[width=2.65in]{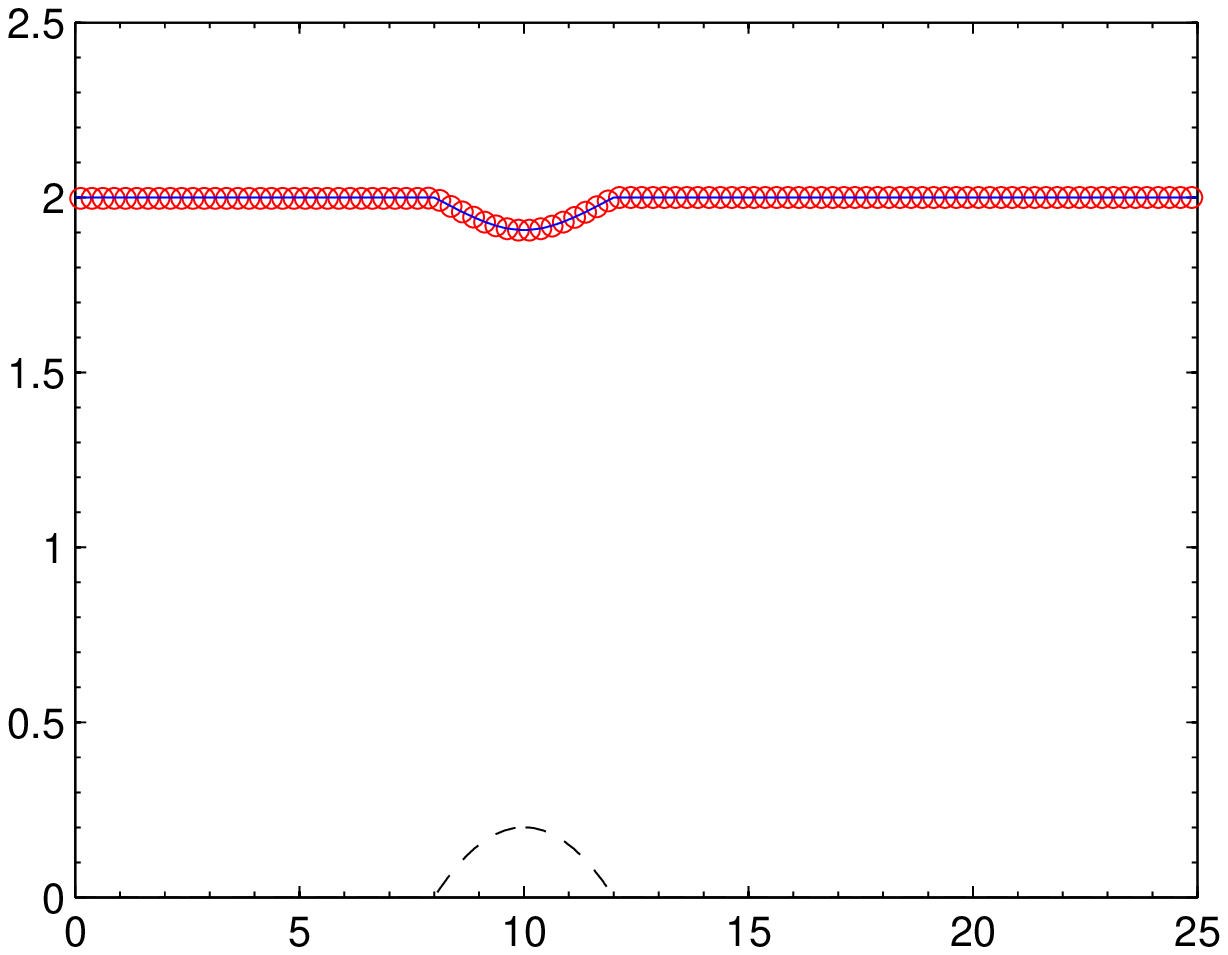}\includegraphics[width=2.65in]{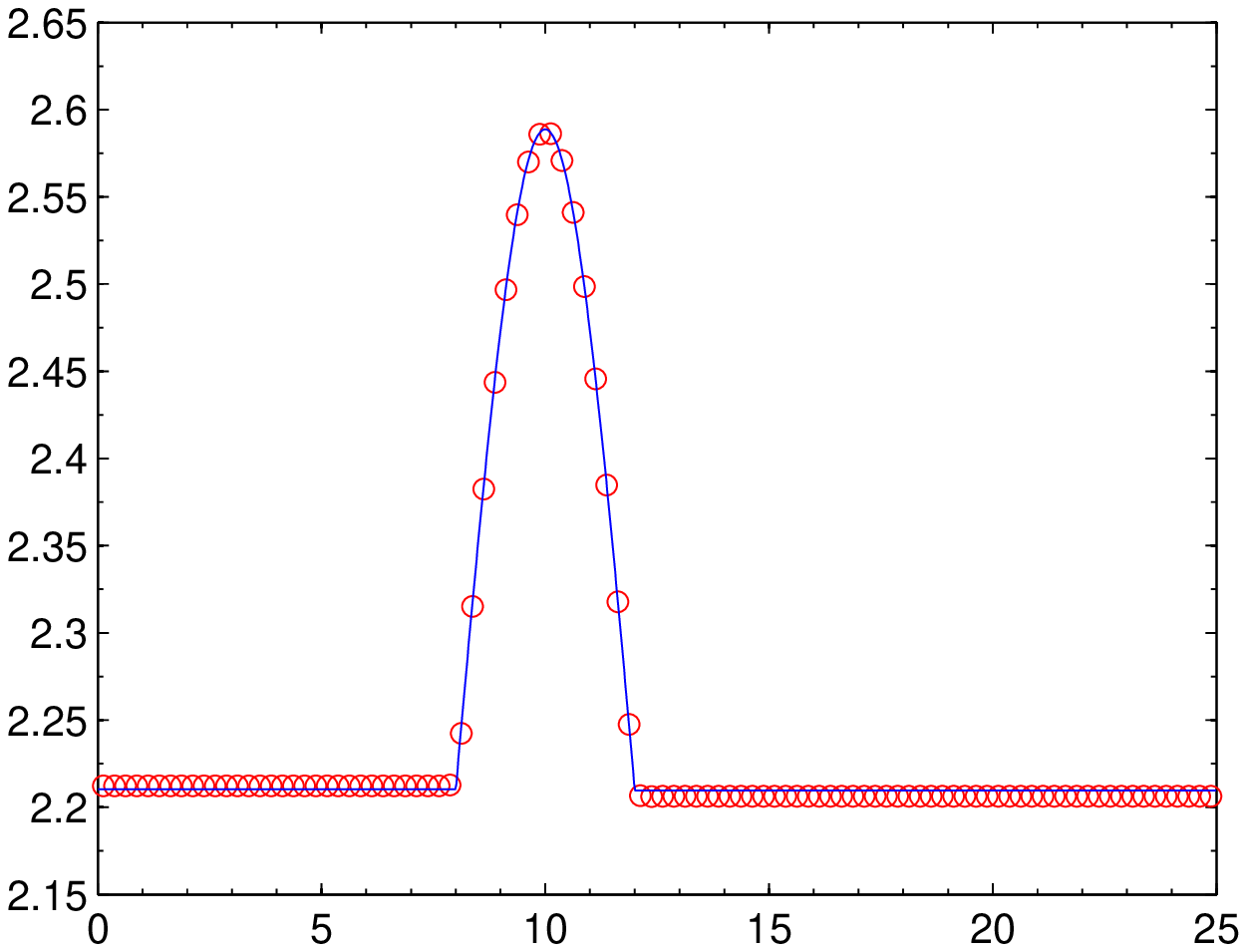}}
\caption{\sf Example \ref{ex75}, subcritical case: Solutions ($w$ together with $B$ is on the left, $u$ is on the right) computed using
uniform grids with 100 (circles) and 1000 (solid line) cells. The bottom topography $B$ is plotted with the dashed line.\label{fig76}}
\end{figure}
\begin{figure}[ht!]
\centerline{\includegraphics[width=2.65in]{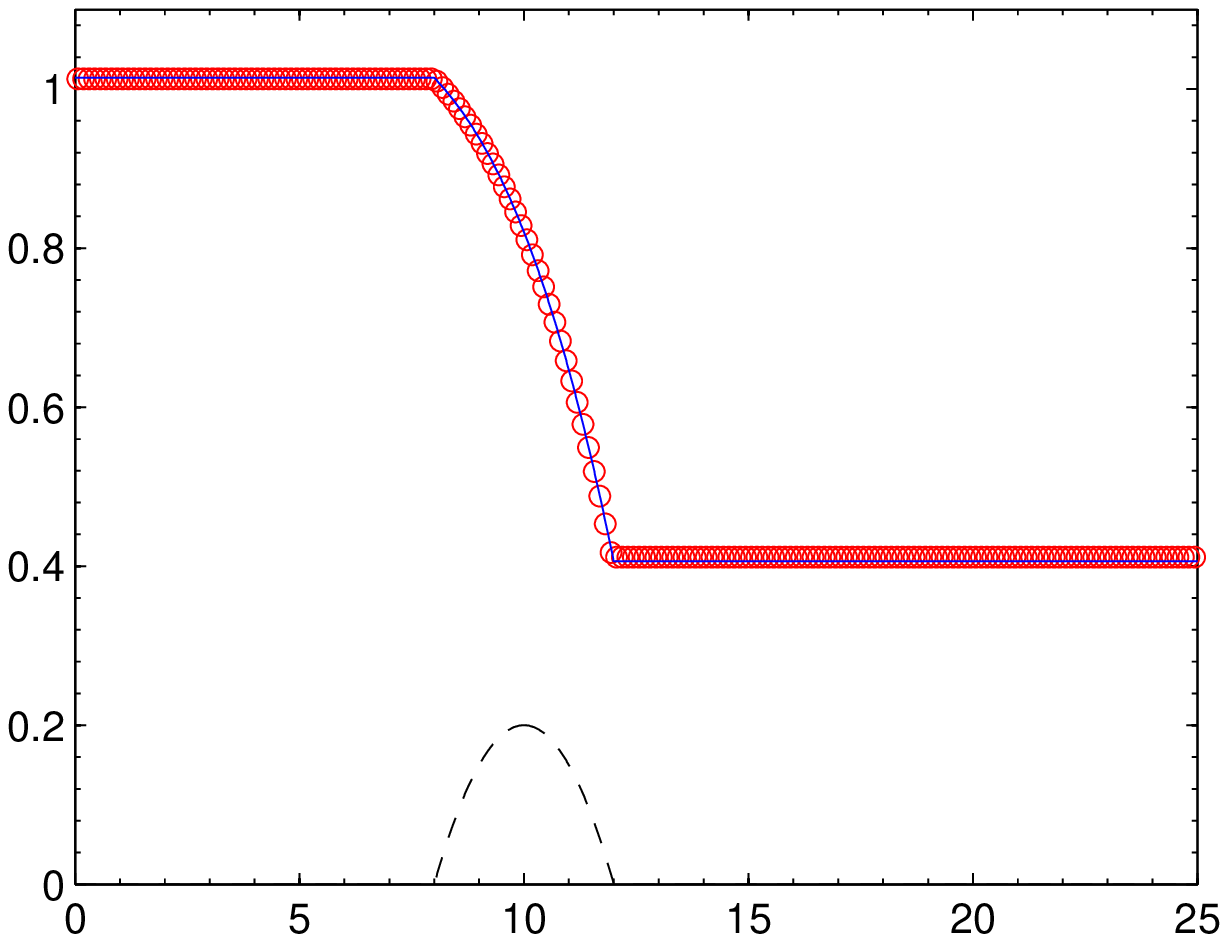}\includegraphics[width=2.65in]{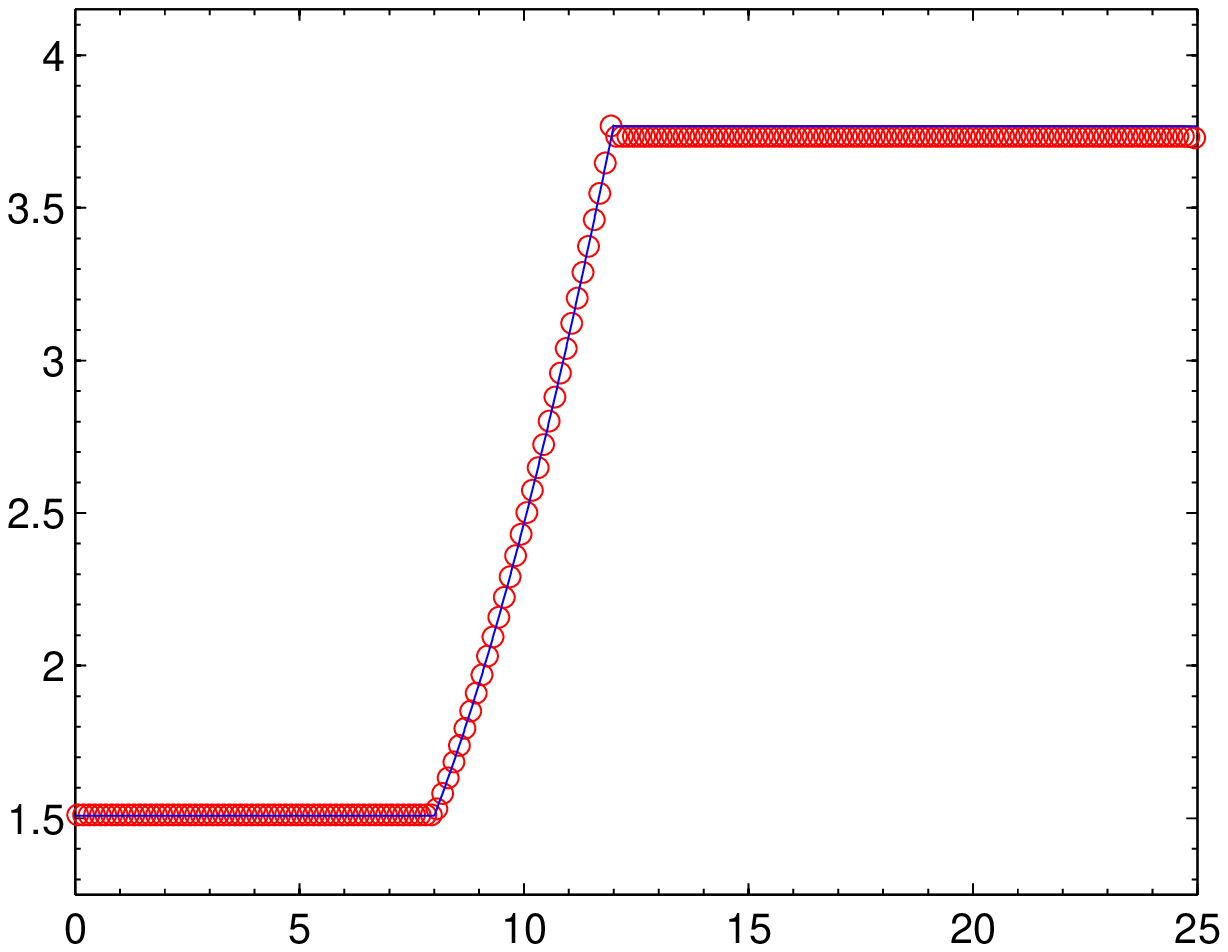}}
\caption{\sf Example \ref{ex75}, transcritical case without a stationary shock: Solutions ($w$ together with $B$ is on the left, $u$ is on
the right) computed using uniform grids with 200 (circles) and 2000 (solid line) cells. The bottom topography $B$ is plotted with the dashed
line.\label{fig77}}
\end{figure}
\begin{figure}[ht!]
\centerline{\includegraphics[width=2.65in]{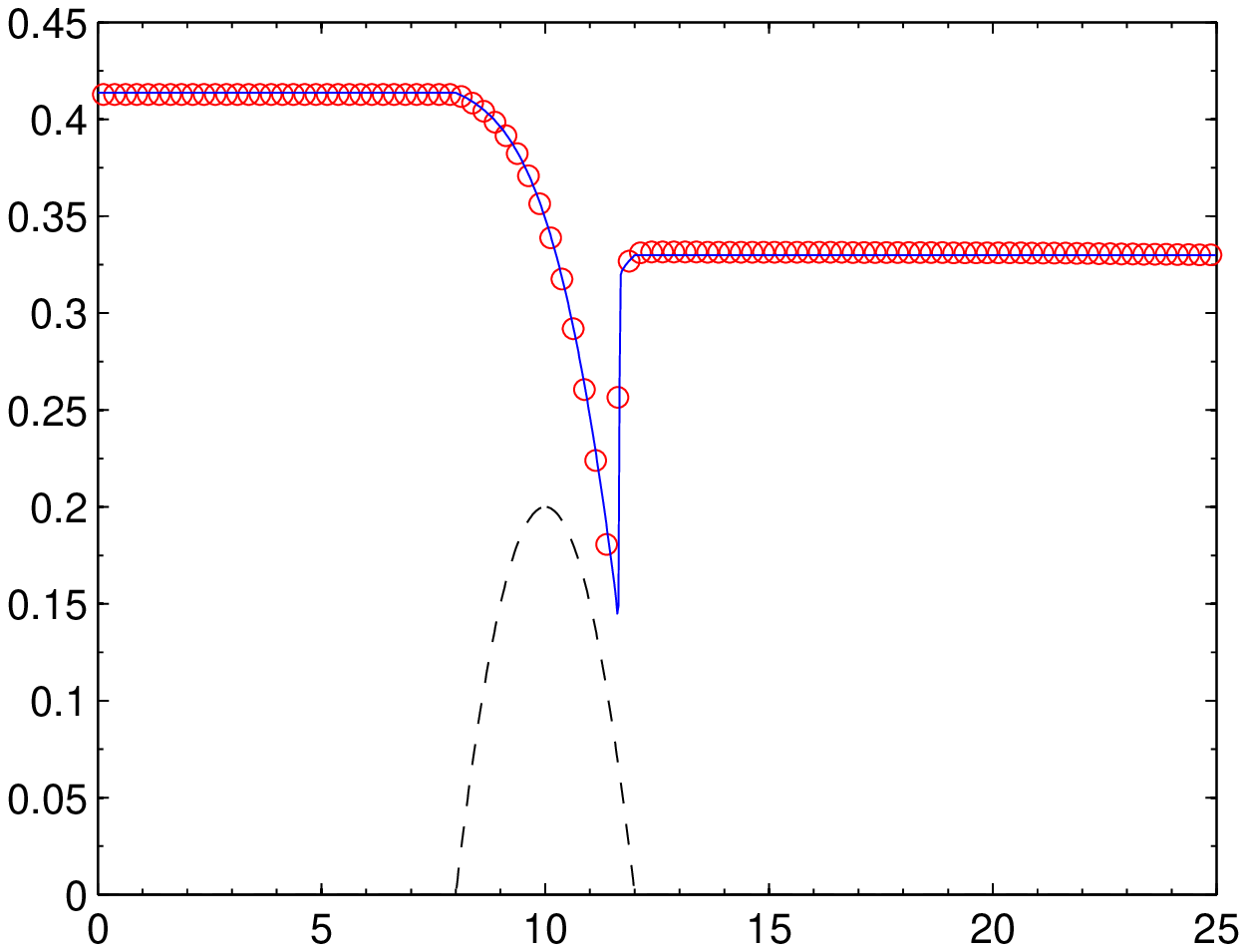}\includegraphics[width=2.65in]{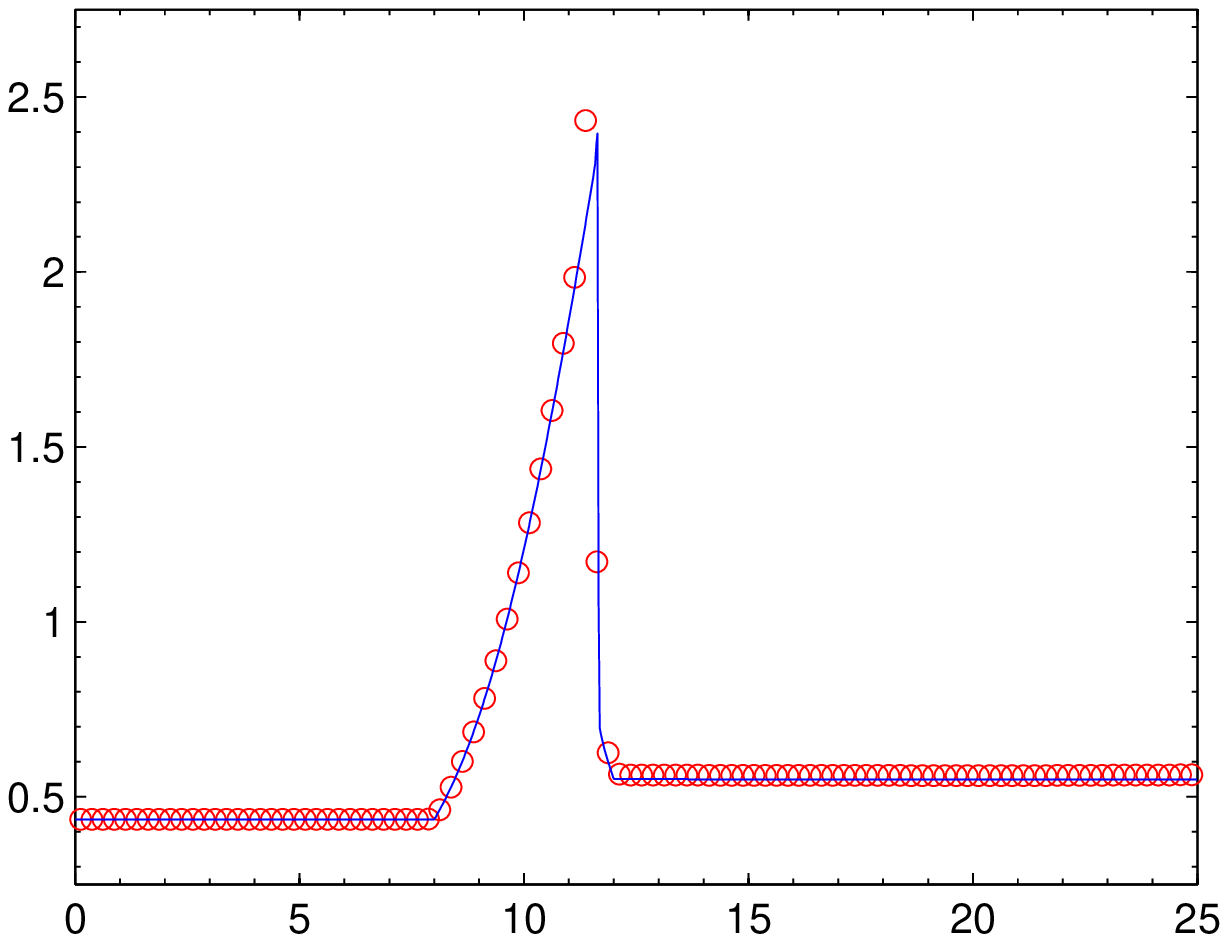}}
\caption{\sf Example \ref{ex75}, transcritical case with a stationary shock: Solutions ($w$ together with $B$ is on the left, $u$ is on the
right) computed using uniform grids with 100 (circles) and 1000 (solid line) cells. The bottom topography $B$ is plotted with the dashed
line.
\label{fig78}}
\end{figure}
\end{example}

\section{Two-Dimensional Numerical Examples}\label{sec8}
In this section, we consider the 2-D Saint-Venant system \eqref{eqn:2DSV}, for which ``lake at rest'' steady-state solutions are given by
\eqref{ss2D} and the corresponding equilibrium variables are $\bm{a}:=(w,hu,hv)^T$. When written in terms of $\bm{a}$, the source term
becomes $\bm{S}=\big(0,-g(w-B)B_x,-g(w-B)B_y\big)^T$.

Let $\xbar w(t)$ be the global spatial average of the water surface $w(x,y,t)$. We now apply the constant subtraction technique presented in
Section \ref{sec:wb} to the system \eqref{eqn:2DSV} and rewrite it as
\begin{equation}
\left\{\begin{aligned}&w_t+(hu)_x+(hv)_y=0,\\&(hu)_t+\left(\frac{(hu)^2}{w-B}+g[\xbar w(t)-w]B+\frac{g}{2}w^2\right)_x+
\left(\frac{(hu)(hv)}{w-B}\right)_y=g[\xbar w(t)-w]B_x,\\&(hv)_t+\left(\frac{(hu)(hv)}{w-B}\right)_x+
\left(\frac{(hv)^2}{w-B}+g[\xbar w(t)-w]B+\frac{g}{2}w^2\right)_y=g[\xbar w(t)-w]B_y.\end{aligned}\right.
\label{5.1}
\end{equation}
The system \eqref{5.1} is advantageous over the original system \eqref{eqn:2DSV} since at ``lake at rest'' steady states, the source terms
in the system \eqref{5.1} vanish and the fluxes are constant. Therefore, a direct application of the CSOC from \cite{Liu05} leads to the
2-D well-balanced CSOC.

As in the 1-D case, all of the simulations in this section are conducted by the third-order well-balanced CSOC with the HR limiter. In all
of the 2-D examples, we take the CFL number 0.45 and the gravitational acceleration constant $g=9.812$.

\begin{example}[Verification of the Well-Balanced Property]\label{ex81}
This test problem is taken from \cite{XingShu05}. The computational domain is $[0,1]\times[0,1]$, and the initial condition is the ``lake at
rest'' state with $w(x,y,0)\equiv1,\,(hu)(x,y,0)\equiv(hv)(x,y,0)\equiv0$ and $B(x,y)=0.8e^{-50[(x-0.5)^2+(y-0.5)^2]}$, which should be
exactly preserved. We use absorbing boundary conditions and compute the numerical solution at time $t=0.1$ using a $100\times100$ uniform
mesh. The $L^1$- and $L^\infty$-errors for both the surface level $w$ and discharges $hu$ and $hv$ are shown in Table \ref{tab81}. As one
can see, the $L^1$-errors are machine zeros, while the $L^\infty$-errors are also very close to the round-off errors and are smaller than
the errors reported in \cite{XingShu05}.
\begin{table}[ht!]
\begin{center}
\begin{tabular}{|c||c|c|c|}
\hline
                 & $w$                    & $hu$                   & $hv$ \\ \hline
$L^1$-error      & $2.2160\times10^{-17}$ & $8.4091\times10^{-18}$ & $9.5723\times10^{-18}$\\
$L^\infty$-error & $8.6597\times10^{-15}$ & $3.9053\times10^{-15}$ & $4.4746\times10^{-15}$\\
\hline
\end{tabular}
\caption{\sf Example \ref{ex81}: $L^1$- and $L^\infty$-errors at time $t=0.1$.\label{tab81}}
\end{center}
\end{table}

We have also computed a long time solution of this problem and observed that the $L^1$-errors remain equal to machine zeros.
\end{example}

\begin{example}[Small Perturbation of the ``Lake at Rest'' State]\label{ex82}
This test problem, proposed in \cite{LeVeque98}, is a 2-D version of Example \ref{ex73}. The computational domain is $[0,2]\times[0,1]$, the
initial data are
\begin{equation*}
(hu)(x,y,0)\equiv(hv)(x,y,0)\equiv0,\quad w(x,y,0)=\left\{\begin{aligned}&1.01,&&\mbox{if}~0.05\le x\le0.15,\\&1,&&\mbox{otherwise}.
\end{aligned}\right.
\end{equation*}
and absorbing boundary conditions are imposed at all of the boundaries. The bottom topography contains a vertical hump is given by
\begin{equation*}
B(x,y)=0.8e^{-5(x-0.9)^2-50(y-0.5)^2}.
\end{equation*}

We compute the solution and monitor how the right-going disturbance propagates past the hump (the left-going disturbance leaves the domain
and does not affect the solution after this). We use two uniform grids with $200\times100$ and $600\times300$ cells. The snapshots of the
computed solution at times $t=0.12,0.24,0.36,0.48$ and 0.6 are shown in Figure \ref{fig81}. Notice that the wave speed is smaller above the
hump than anywhere else, which distorts the initially planar disturbance. The obtained results clearly demonstrate that the CSOC can capture
the small perturbation and resolve the complicated features of the studied flow very well.
\begin{figure}[ht!]
\centerline{\includegraphics[width=3.05in]{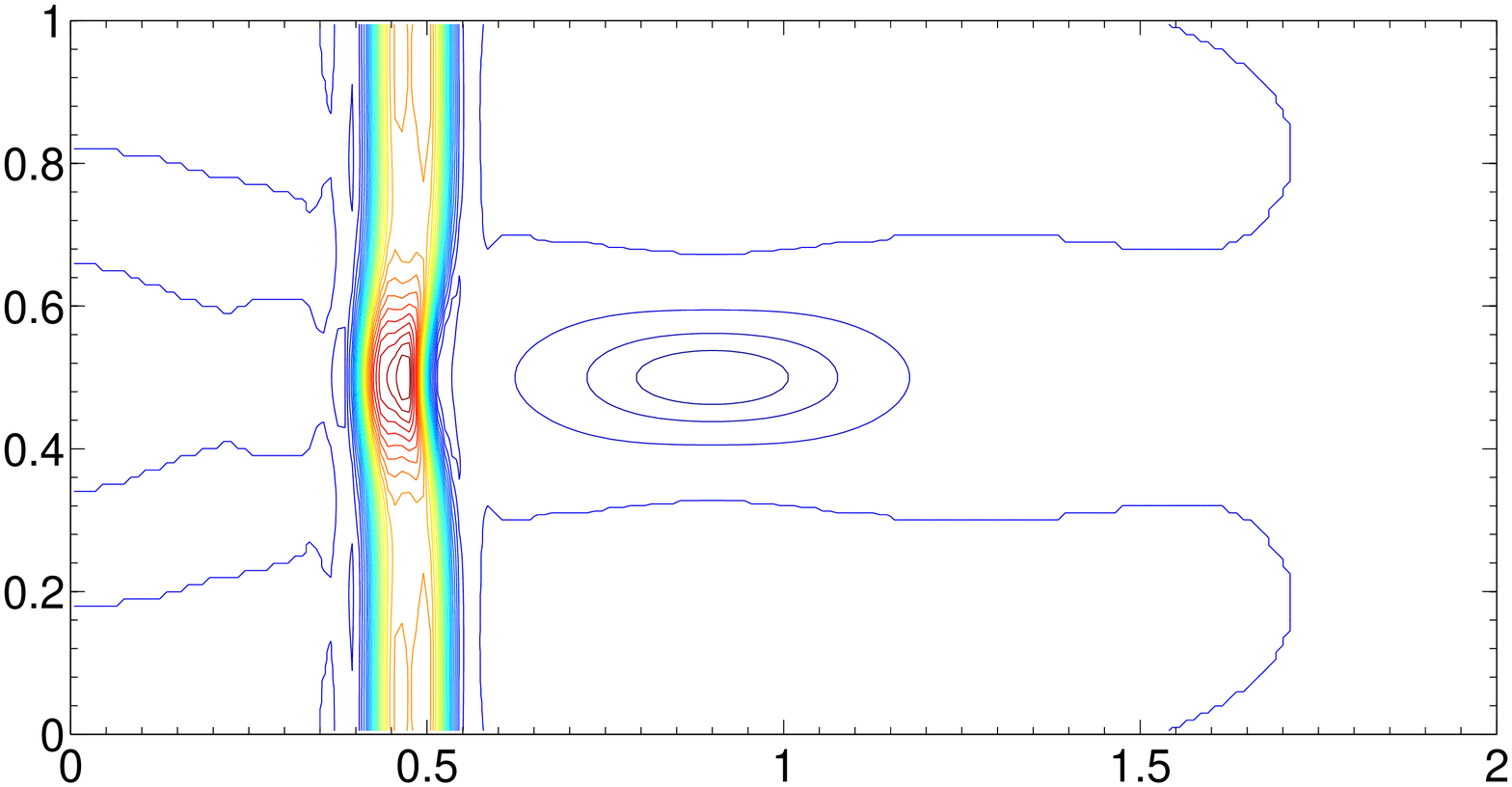}\includegraphics[width=3.05in]{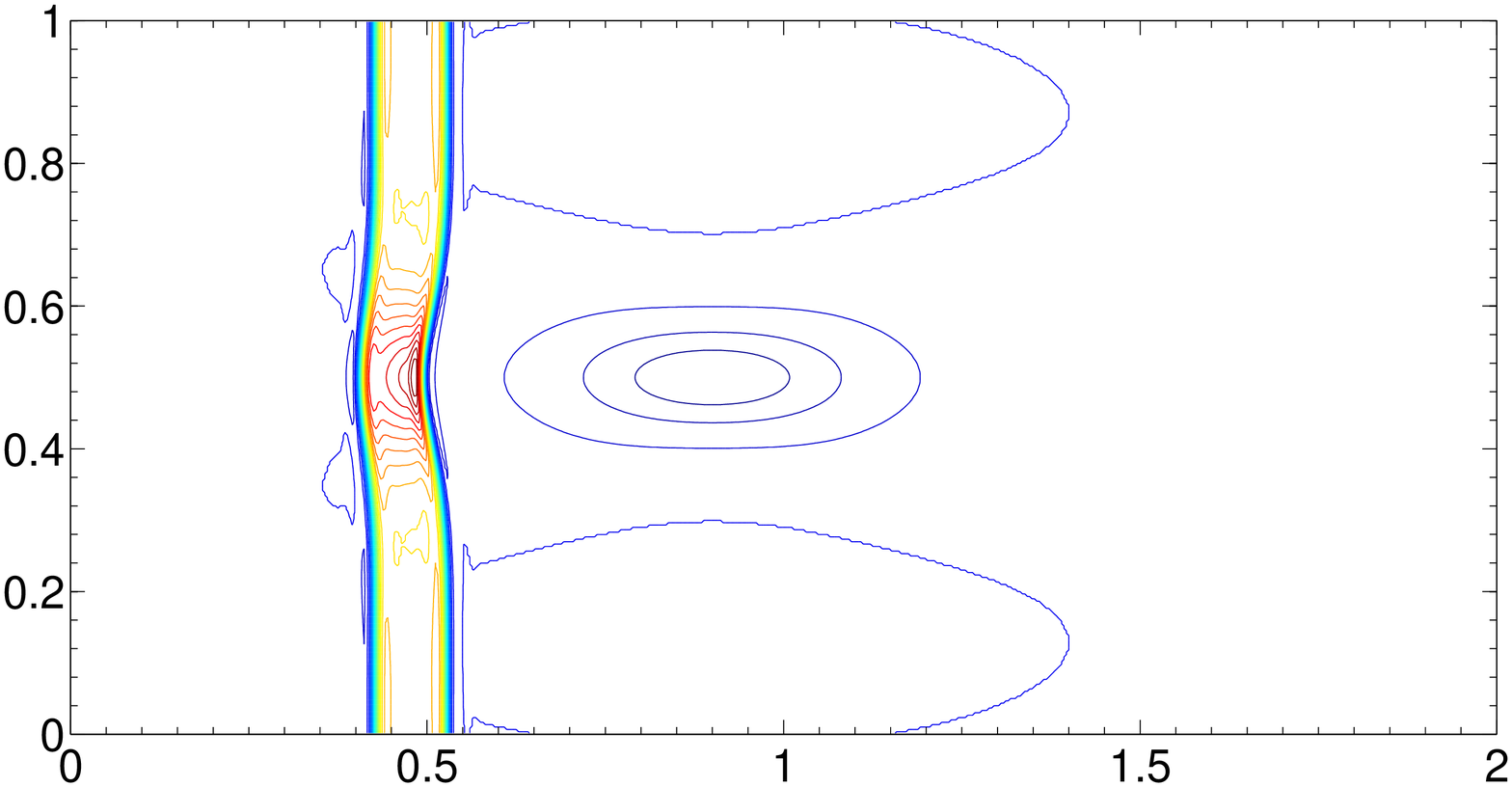}}
\centerline{\includegraphics[width=3.05in]{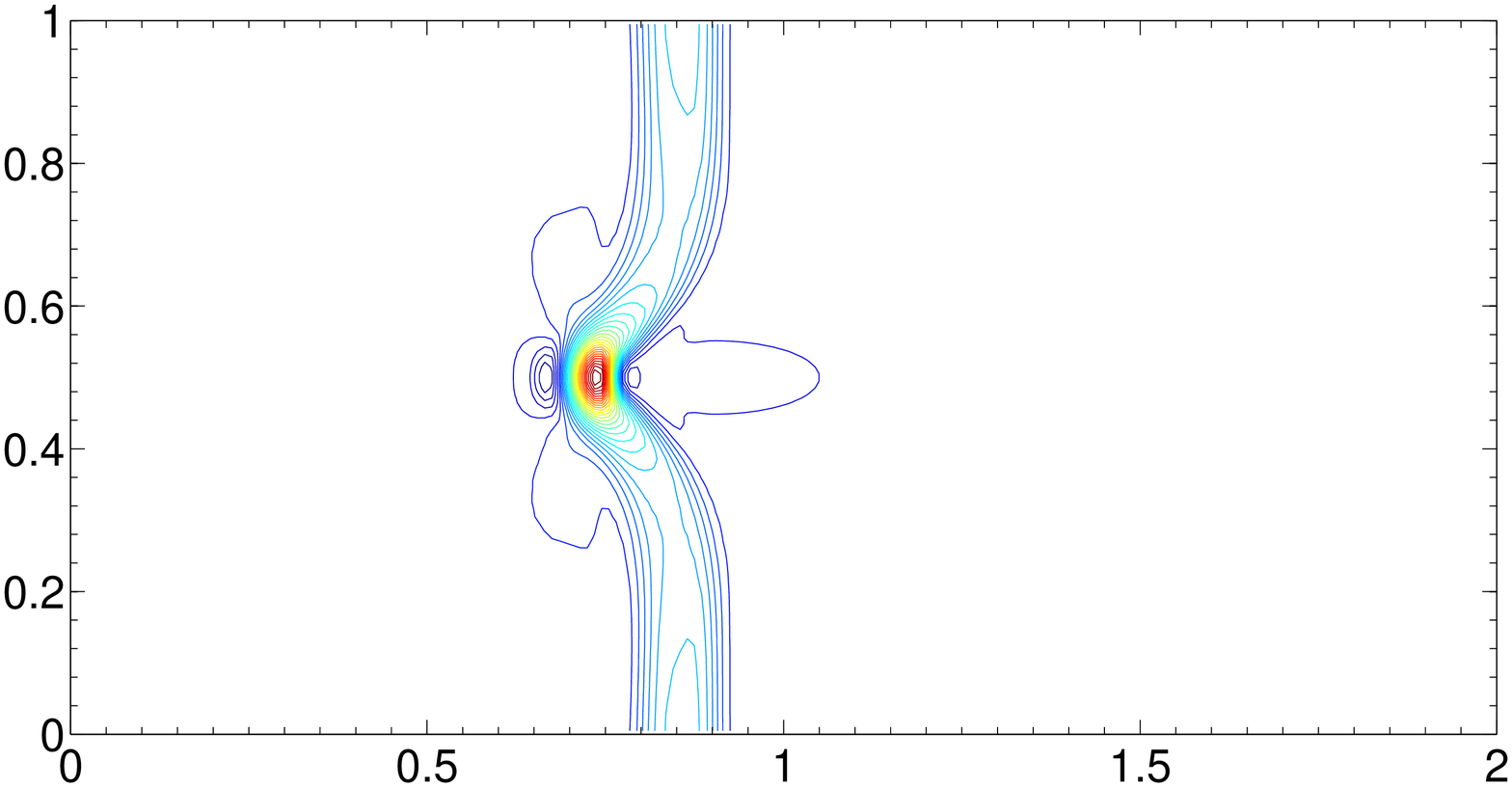}\includegraphics[width=3.05in]{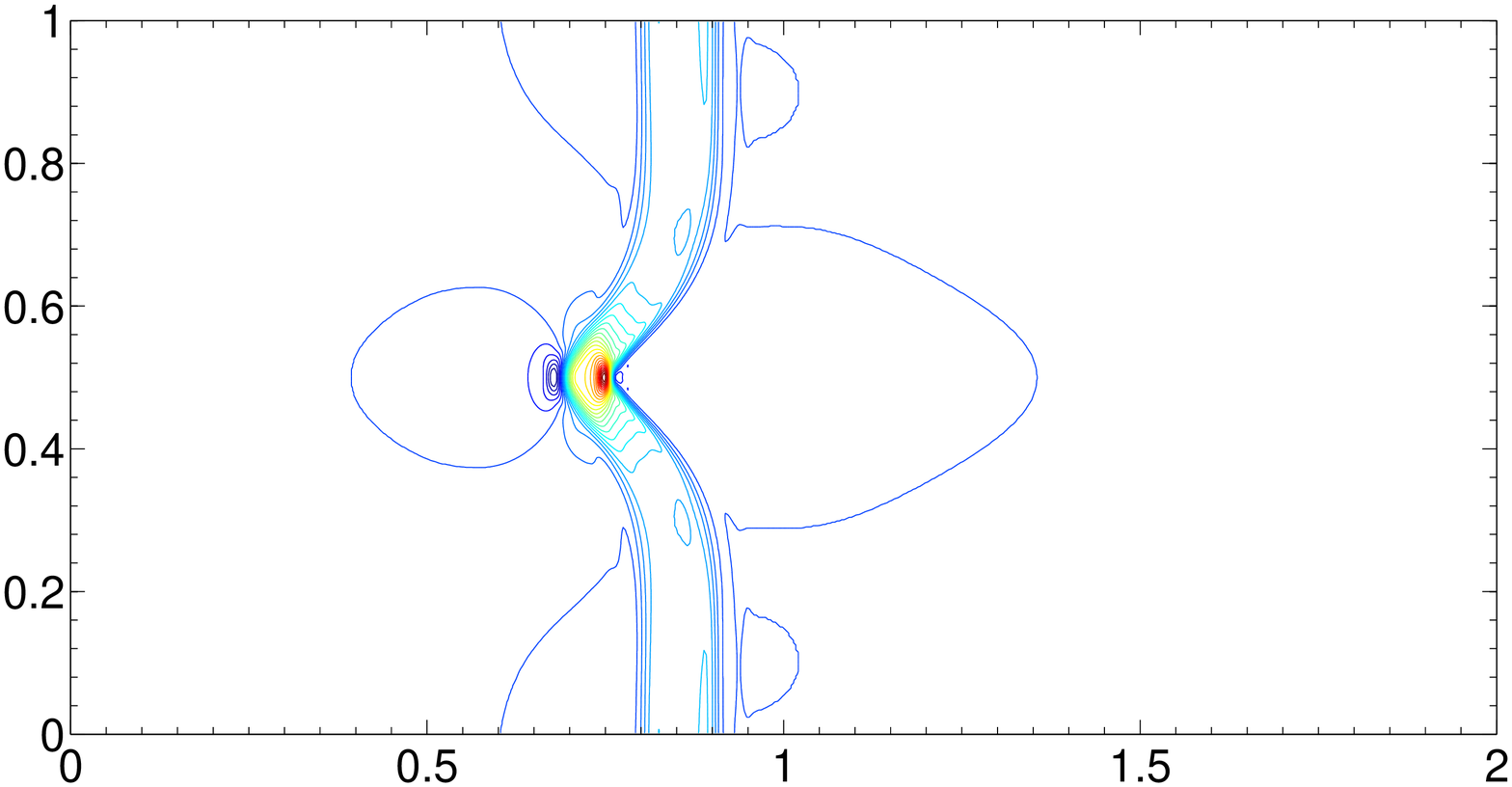}}
\centerline{\includegraphics[width=3.05in]{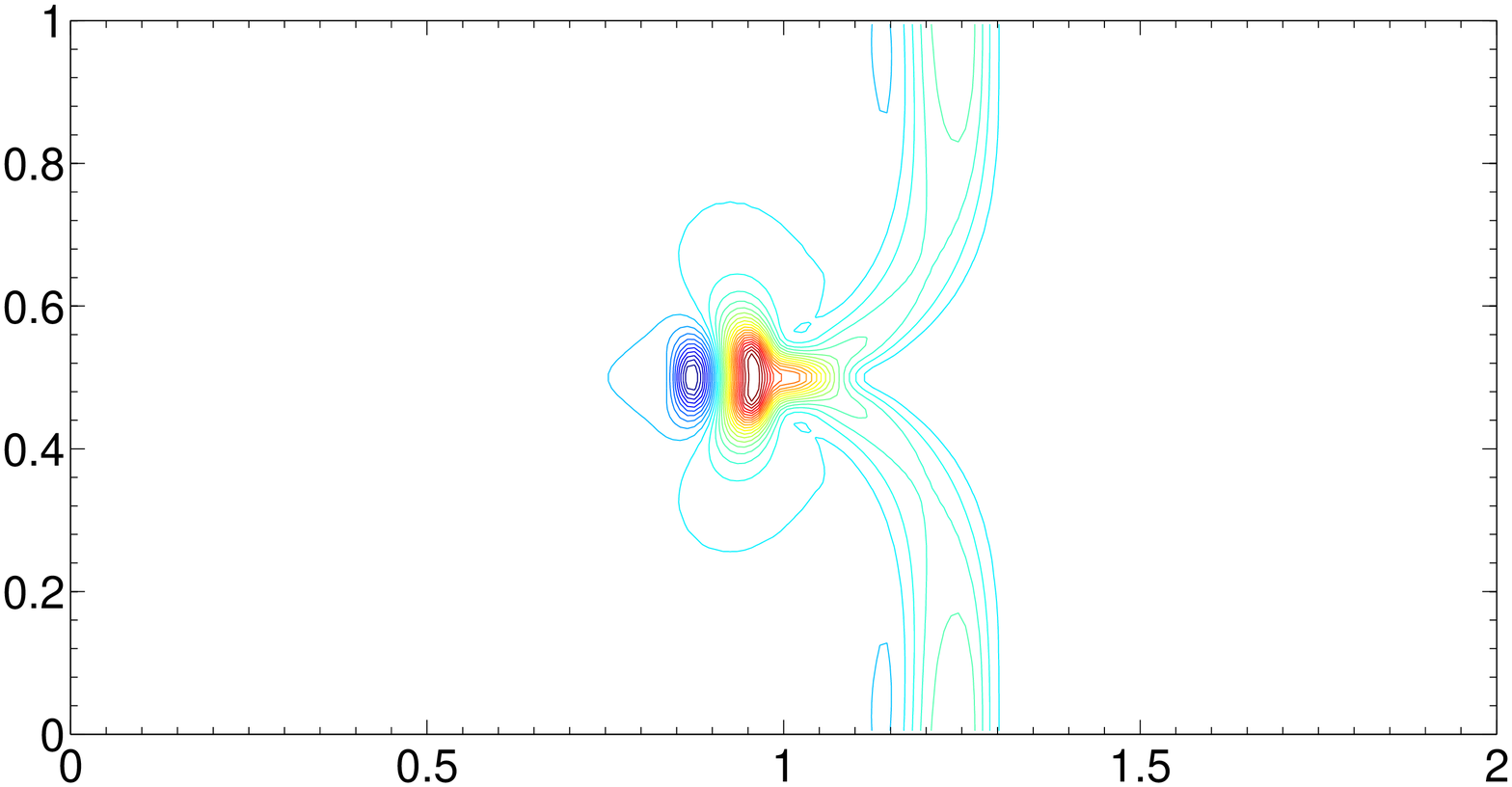}\includegraphics[width=3.05in]{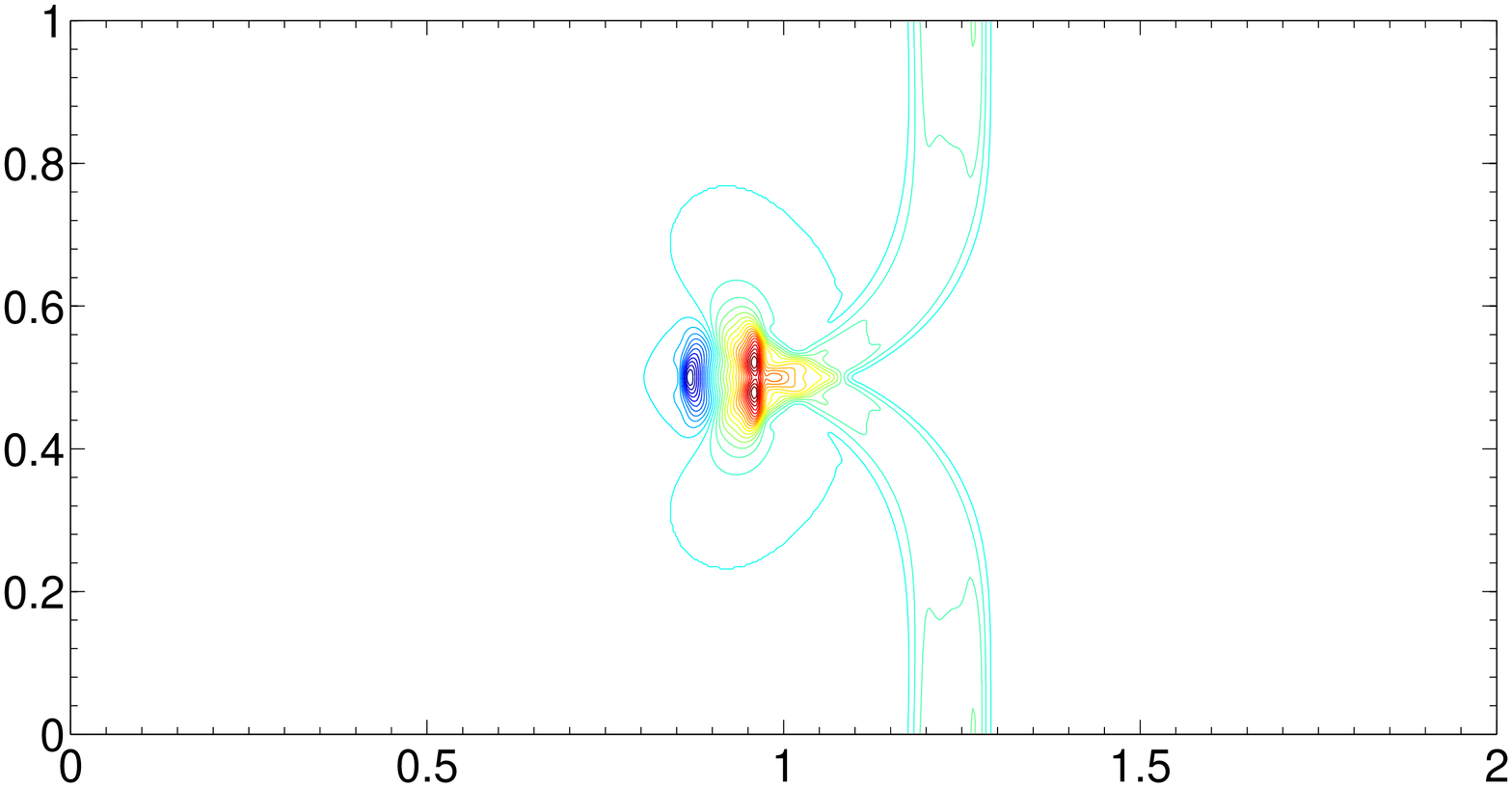}}
\centerline{\includegraphics[width=3.05in]{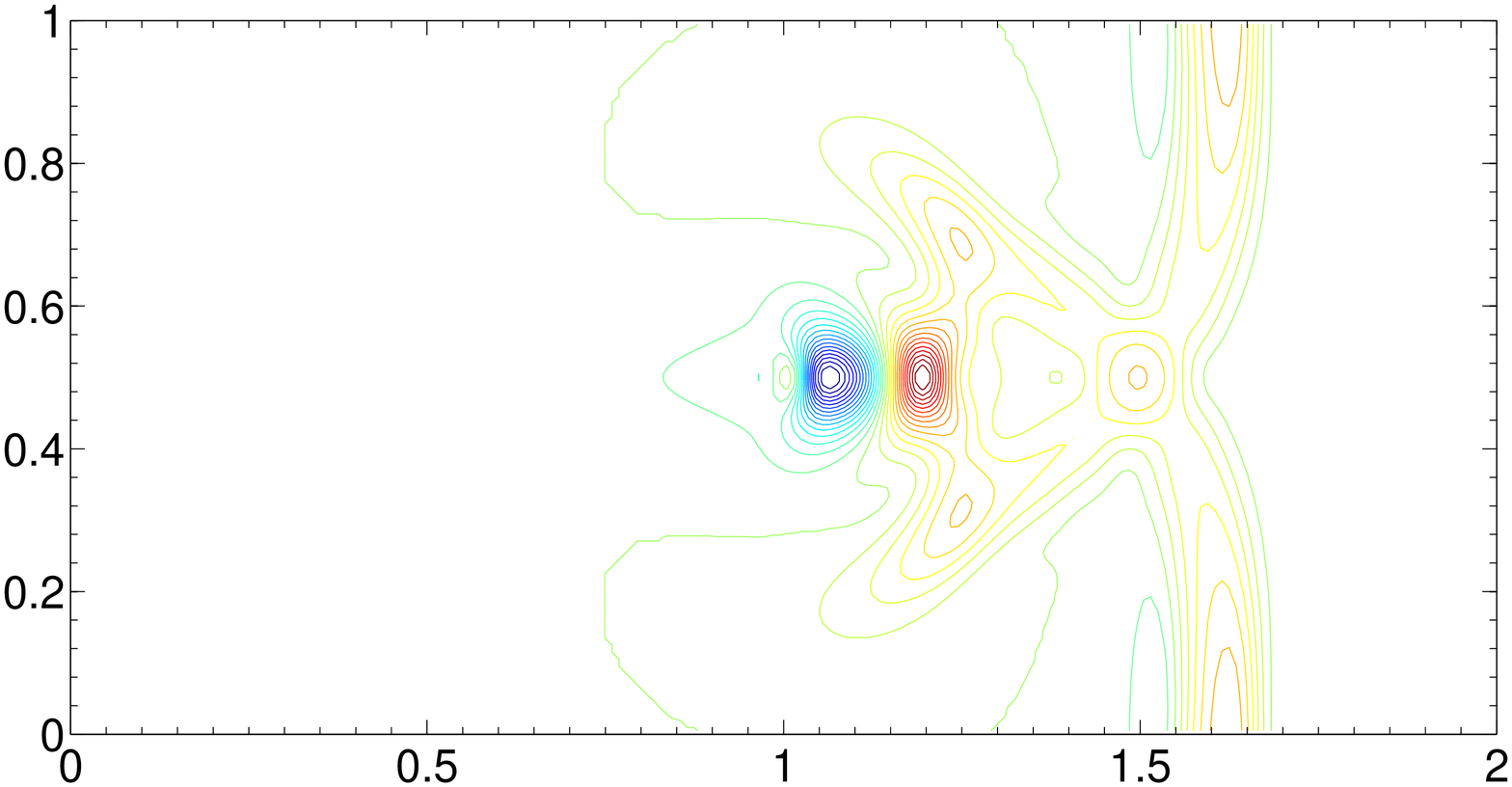}\includegraphics[width=3.05in]{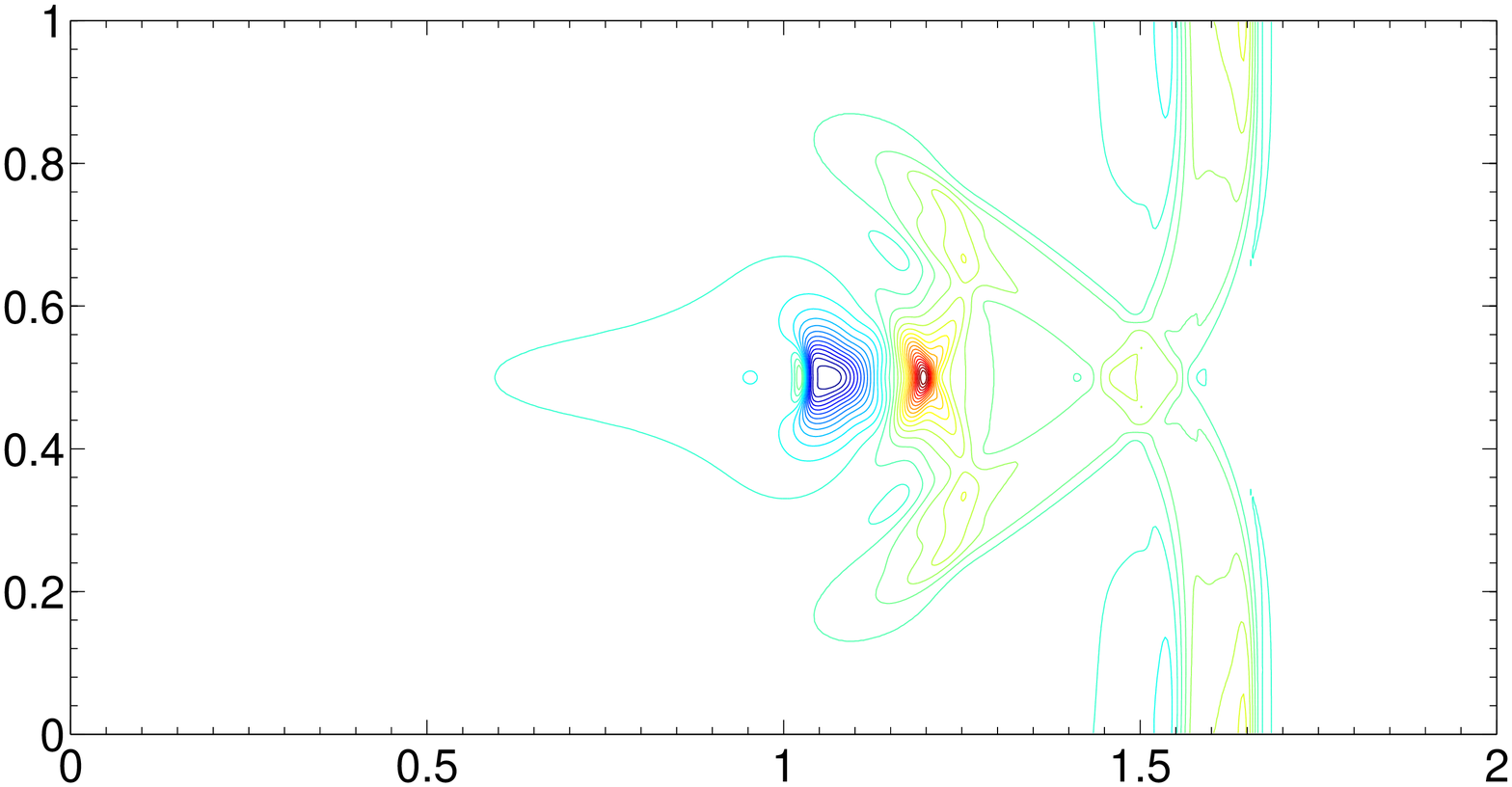}}
\centerline{\includegraphics[width=3.05in]{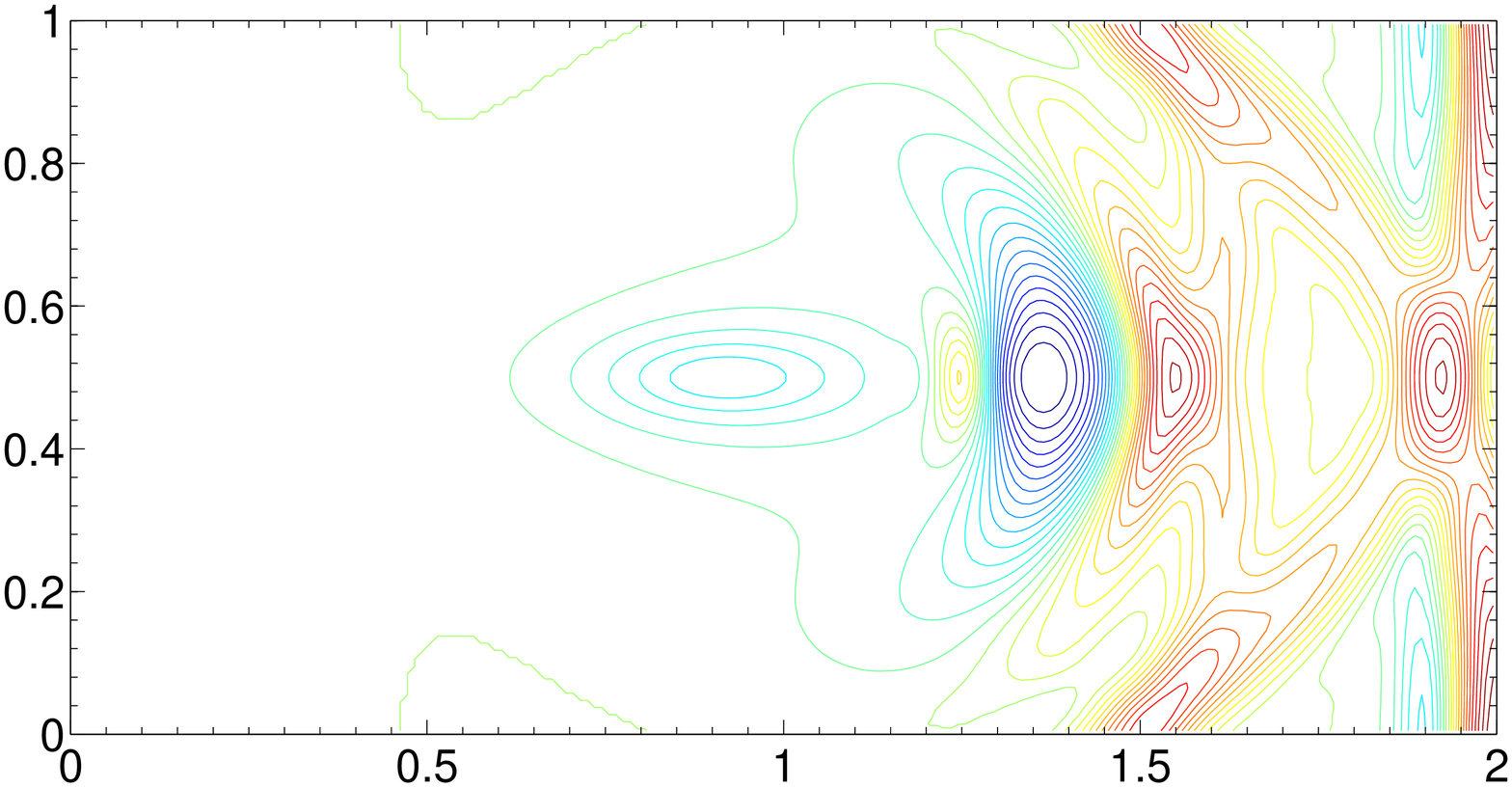} \includegraphics[width=3.05in]{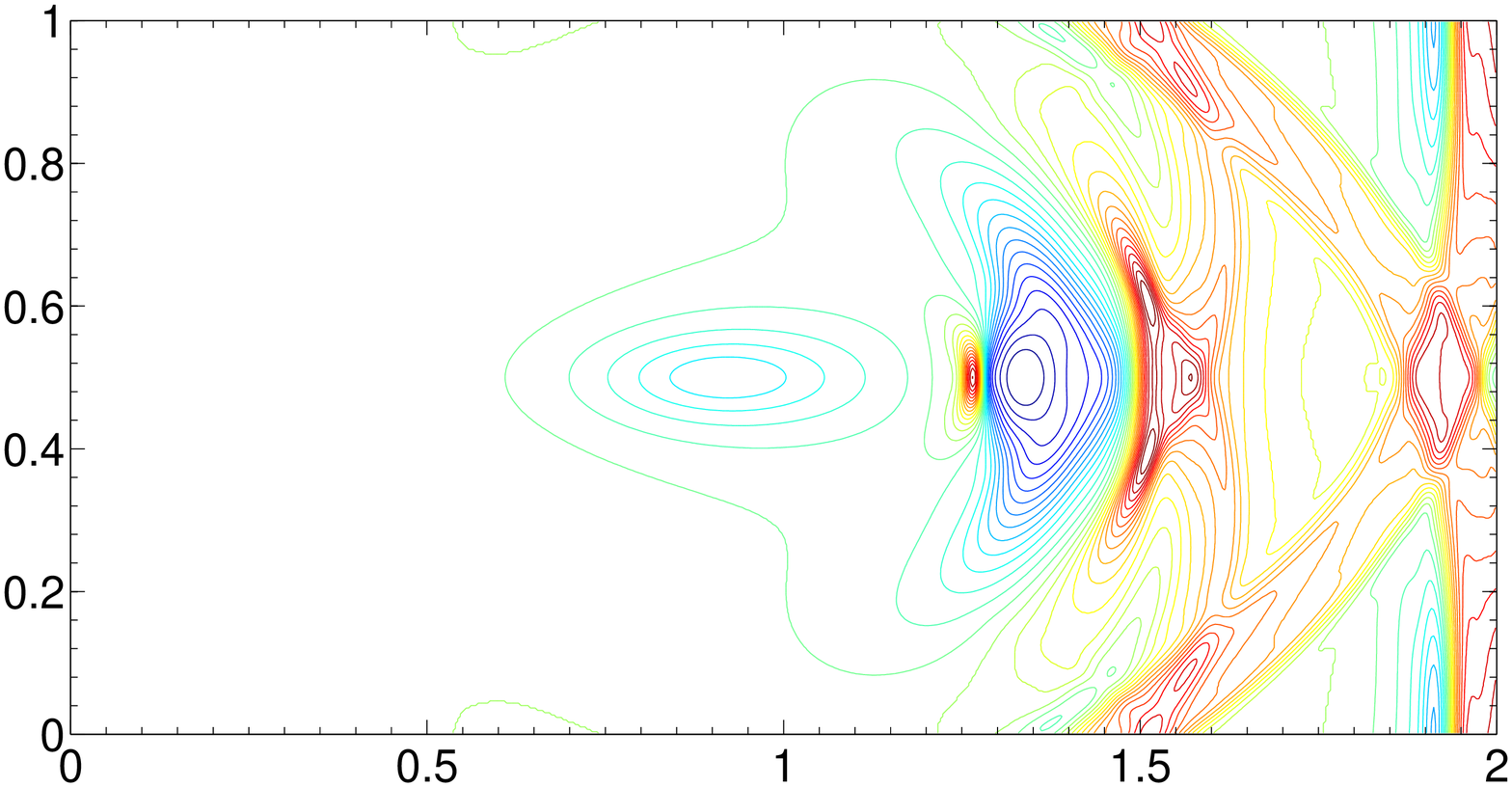}}
\caption{\sf Example \ref{ex82}: Contour-plot of $w$ on $200\times100$ (left) and $600\times300$ (right) uniform meshes. The solution is
shown at times $t=0.12, 0.24, 0.36, 0.48, 0.6$ (from top to bottom).\label{fig81}}
\end{figure}
\end{example}

\section{Conclusions and Future Works}\label{sec9}
In this paper, we have developed central schemes on overlapping cells (CSOC) for the Saint-Venant system of shallow water equations in both
one and two space dimensions. A new constant subtraction technique is proposed to make the CSOC well-balanced, that is, to guarantee that
they exactly preserve ``lake at rest'' steady states while still maintain the original high-order of accuracy and non-oscillatory property.
In fact, this technique is quite general and can be utilized for development other finite-volume schemes (this will be done in our future
works). We have provided extensive numerical results to demonstrate the well-balanced property, high-order of accuracy and non-oscillatory
nature of the proposed CSOC. Our future works will include development of positivity-preserving CSOC and also extension of their
well-balanced properties to the case of more general steady-state solutions.

\bibliographystyle{siam}
\bibliography{ref}
\end{document}